\documentclass[reqno]{amsart}
\usepackage{latexsym}
\usepackage{amsmath}
\usepackage{amssymb}
\usepackage{amsthm}
\usepackage{amscd}

\newtheorem{theorem}{Theorem}[section]
\newtheorem{lemma}[theorem]{Lemma}
\newtheorem{corollary}[theorem]{Corollary}
\newtheorem{proposition}[theorem]{Proposition}

\theoremstyle{definition}

\newtheorem{remark}[theorem]{Remark}
\newtheorem{definition}[theorem]{Definition}

\theoremstyle{remark}

\newcommand{\N}{{\mathbb N}}
\newcommand{\R}{{\mathbb R}}
\newcommand{\C}{{\mathbb C}}

\newcommand{\Z}{{\mathbb Z}}
\newcommand{\uH}{{\mathbb H}}

\newcommand{\spt}{{\rm supp}\,}

\newcommand{\shizensu}{\mathbb{N}}
\newcommand{\jissu}{\mathbb{R}}

\newcommand{\fourier}{\mathcal{F}}
\newcommand{\inversefourier}{\mathcal{F}^{-1}}
\newcommand{\qu}{{\rm qu}}
\newcommand{\tr}{{\rm Tr}_{\mathbb{R}^n}} 
\newcommand{\utr}{{\rm Tr}_{\mathbb{R}^n_+}} 
\newcommand{\abs}[1]{\lvert#1\rvert}
\begin{document}

\title[]{Trace and extension operators for Besov spaces
and Triebel--Lizorkin spaces
with variable exponents.}

\author{Takahiro Noi}
\date{}

\maketitle

\begin{abstract}
This paper is concerned with a boundedness of 
trace and extension operators for Besov spaces and Triebel--Lizorkin spaces on upper half space with variable exponents. 
To define trace and extension operators, we introduce a quarkonial decomposition for  Besov spaces
and Triebel--Lizorkin spaces
with variable exponents on $\R^n$.  
Furthermore, we study trace and extension operators for 
 Besov spaces and Triebel--Lizorkin spaces
with variable exponents on upper half spaces $\R^n_+$. 
\end{abstract}

\vspace{1.5cm}

\noindent{\bf $2010$ Mathematics Subject Classification}\ \
Primary 42B35\,; Secondary 41A17.

% \linenumbers

\section{Introduction}

The function spaces with variable exponent(s) have a long history since  the discovery by Orlicz\cite{Orlicz} and  
in recent years, these spaces received great attention in connection with  electrorheological fluids \cite{Rajagopal-Ruzicka}. 

Besov spaces with variable exponents $B_{p(\cdot),q(\cdot)}^{s(\cdot)}(\R^n)$  and  Triebel--Lizorkin spaces with variable exponents $F_{p(\cdot),q(\cdot)}^{s(\cdot)}(\R^n)$  were introduced by 
Almeida and H\"ast\"o \cite{Almeida} and Diening, H\"ast\"o and Roudenko \cite{Diening}, respectively.  
 Diening, H\"ast\"o and Roudenko \cite{Diening} proved the atomic decomposition for  $F_{p(\cdot),q(\cdot)}^{s(\cdot)}(\R^n)$ and applied the result to trace theorem. 
 Kempka \cite{Kempka} proved the atomic, molecular and wavelet expansion for $2$-microlocal Besov and Triebel-Lizorkin spaces 
with variable integrability. But, in the case of Besov space, the summability index $q$ is a constant. 
Recently, Moura, Neves and Schneider \cite{Moura} proved the boundedness of 
the trace operator for $2$-microlocal Besov spaces by using atomic decomposition, 
but summability index $q$ is a constant. 
Present author \cite{Noi} studied a Fourier multiplier 
 for  $F_{p(\cdot),q(\cdot)}^{s(\cdot)}(\R^n)$ and $B_{p(\cdot),q(\cdot)}^{s(\cdot)}(\R^n)$ and the author and Sawano \cite{Noi-Sawano} studied 
atomic decomposition and complex interpolation for  $F_{p(\cdot),q(\cdot)}^{s(\cdot)}(\R^n)$ and $B_{p(\cdot),q(\cdot)}^{s(\cdot)}(\R^n)$.  
Drihem \cite{Drihem} obtained a  detailed atomic decomposition for  $B_{p(\cdot),q(\cdot)}^{s(\cdot)}(\R^n)$. 
Recently, the author and Izuki \cite{Noi-Izuki} studied a duality of $F_{p(\cdot),q(\cdot)}^{s(\cdot)}(\R^n)$, $B_{p(\cdot),q(\cdot)}^{s(\cdot)}(\R^n)$
 and Herz spaces with variable exponents $K_{p(\cdot)}^{\alpha(\cdot),q(\cdot)}(\R^n)$. 
 
To prove a boundedness of 
the trace operator ,  we introduce  quarkonial decompositions. 

This paper concerns itself with quarkonial decompositions, trace operators and extension operators 
for Besov spaces and Triebel--Lizorkin spaces
with variable exponents. 
First, we state atomic and quarkonial decompositions of Besov spaces and Triebel--Lizorkin spaces
with variable exponents. 
Secondly, we extend trace operators to Besov spaces and Triebel--Lizorkin spaces
with variable exponents. 
Finally, we study trace and extension operators for 
 Besov spaces and Triebel--Lizorkin spaces
with variable exponents on upper half spaces $\R^n_+$. 

\section{Definition of Besov spaces and Triebel--Lizorkin spaces
with variable exponents}

Denote by $\mathcal{P}_0(\R^n)$ 
the set of all measurable functions $p(\cdot):\R^n \to (0,\infty)$ such that 
\begin{equation}\label{eq:bounded}
0<p^-=\mathop{{\rm{ess}}\,\,\rm{inf}}_{x\in\R^n}p(x), \ \ 
\mathop{{\rm{ess}}\,\,\rm{sup}}_{x\in\R^n}p(x)=p^+<\infty. 
\end{equation}
For $p \in \mathcal{P}_0(\R^n)$, 
let $L^{p(\cdot)}(\R^n)$ be
the set of measurable functions $f$ on $\R^n$ 
such that for some $\lambda>0$,
\[
\int_{\R^n}\left(\frac{|f(x)|}{\lambda}\right)^{p(x)}\,{\rm d}x\le 1.
\]
The infimum of such $\lambda$ will be denoted by $||f||_{L^{p(\cdot)}}$.
The set $L^{p(\cdot)}(\R^n)$
becomes a quasi Banach function space equipped 
with the Luxemburg--Nakano norm $||f||_{L^{p(\cdot)}}$.
More precisely,
\[
||f||_{L^{p(\cdot)}}
=
\inf\left\{\lambda>0\,:\,
\int_{\R^n}\left(\frac{|f(x)|}{\lambda}\right)^{p(x)}\,{\rm d}x\le 1
\right\}.
\]
If $\Omega\subset\R^n$ is a measurable set, then 
we define 
\[
||f||_{L^{p(\cdot)}(\Omega)}
=
\inf\left\{\lambda>0\,:\,
\int_{\Omega}\left(\frac{|f(x)|}{\lambda}\right)^{p(x)}\,{\rm d}x\le 1
\right\}.
\]
To define Besov and Triebel--Lizorkin spaces with variable exponents, 
we postulate the following conditions:
There exists a positive constant $C_{\log}(p)$ such that
\begin{equation}
\label{eq:log-Holder-continuity}
|p(x)-p(y)| \le \frac{C_{\log}(p)}{\log(e+|x-y|^{-1})}
\quad
(x,y \in \R^n, \, x \ne y)
\end{equation}
and there exist a positive constant $C_{\log}(p)$ and real number $p_\infty$ such that
\begin{equation}
\label{eq:log-decay}
|p(x)-p_\infty| \le \frac{C_{\log}(p)}{\log(e+|x|)}
\quad (x \in \R^n).
\end{equation}
The set of all real valued functions $p:{\mathbb R}^n \to {\mathbb R}$
satisfying {\rm(\ref{eq:log-Holder-continuity})}
and {\rm(\ref{eq:log-decay})} is written by $C^{\rm log}(\R^n)$.

To define Besov spaces with variable exponents, we use mixed Lebesgue sequence space $\ell^{q(\cdot)}(L^{p(\cdot)})$. 

Let $p(\cdot)$, $q(\cdot)\in\mathcal{P}_0(\jissu^n)$. The space $\ell^{q(\cdot)}(L^{p(\cdot)})$ is the collection of 
all sequences $\{f_j\}_{j=0}^{\infty}$ of measurable functions on $\jissu^n$ such that 
\[
||\{f_j\}_{j=0}^{\infty}||_{\ell^{q(\cdot)}(L^{p(\cdot)})}
= \inf\left\{
\mu>0\,:\, \varrho_{\ell^{q(\cdot)}(L^{p(\cdot)})}\left(
\left\{
\frac{f_j}{\mu}
\right\}_{j=0}^{\infty}
\right)\le 1
\right\}<\infty, 
\]
where 
\[
\varrho_{\ell^{q(\cdot)}(L^{p(\cdot)})}\left(
\left\{
f_j
\right\}_{j=0}^{\infty}
\right)
=\sum_{j=0}^{\infty}\inf
\left\{
\lambda_j\,:\,
\int_{\jissu^n}\left(
\frac{|f_j(x)|}{\lambda_j^{\frac{1}{q(x)}}}
\right)^{p(x)}\,{\rm d}x \le 1
\right\}. 
\]

Since we assume that $q^+<\infty$,  
\begin{equation}
\varrho_{\ell^{q(\cdot)}(L^{p(\cdot)})}\left(
\left\{
f_j
\right\}_{j=0}^{\infty}
\right) = \sum_{j=0}^{\infty}
\left|\left|
|f_j|^{q(\cdot)}
\right|\right|_{L^{\frac{p(\cdot)}{q(\cdot)}}}. \label{modular}
\end{equation}

Almeida and H\"ast\"o \cite{Almeida} proved that $||\cdot||_{\ell^{q(\cdot)}(L^{p(\cdot)})}$ 
is a quasi-norm for all $p(\cdot), q(\cdot)\in\mathcal{P}(\R^n)$ 
and that $||\cdot||_{\ell^{q(\cdot)}(L^{p(\cdot)})}$ is a norm when $\frac{1}{p(\cdot)}+\frac{1}{q(\cdot)}\le 1$. 
Kempka and Vyb{\'\i}ral \cite{Kempka2} proved that $||\cdot||_{\ell^{q(\cdot)}(L^{p(\cdot)})}$ is a norm 
if $p(\cdot),q(\cdot)\in\mathcal{P}(\R^n)$ satisfy either $1\le q(x)\le p(x)\le \infty$ alomost everywhere on $\R^n$
or $\frac{1}{p(x)}+\frac{1}{q(x)}\le 1$ for almost all $x\in\R^n$. 
Furthermore, they proved that 
there exist $p(\cdot),q(\cdot)\in\mathcal{P}(\R^n)$ satisfying $\inf_{x\in\R^n}(p(\cdot),q(\cdot))\ge 1$ such that 
a triangle inequality does not hold for $||\cdot||_{\ell^{q(\cdot)}(L^{p(\cdot)})}$. 
This means that $||\cdot||_{\ell^{q(\cdot)}(L^{p(\cdot)})}$ does not always become a norm even if $p(\cdot)$ and $q(\cdot)$ satisfy $p^-$, $q^-\ge 1$. 
However, we have following inequalities. 

\begin{lemma}
\label{min triangle}
\text{ \rm (i)} Let $p(\cdot)\in\mathcal{P}_0(\R^n)$. Then 
\[
\left\| f+g\right\|_{L^{p(\cdot)}}^{\min(p^-,1)} \le 
\left\| f\right\|_{L^{p(\cdot)}}^{\min(p^-,1)}+\left\| g\right\|_{L^{p(\cdot)}}^{\min(p^-,1)}. 
\]

\noindent  
\text{\rm (ii)} Let $p(\cdot), q(\cdot)\in\mathcal{P}_0(\R^n)$. Then 
\[
\left\| \{f_k+g_k\}_{k=0}^{\infty}\right\|_{L^{p(\cdot)}(\ell^{q(\cdot)})}^{\min(p^-,q^-,1)} 
\le 
\left\| \{f_k\}_{k=0}^{\infty}\right\|_{L^{p(\cdot)}(\ell^{q(\cdot)})}^{\min(p^-,q^-,1)} 
+ \left\|\{g_k\}_{k=0}^{\infty}\right\|_{L^{p(\cdot)}(\ell^{q(\cdot)})}^{\min(p^-,q^-,1)}. 
\]

\noindent
\text{\rm (iii)} 
 Let $p(\cdot), q(\cdot)\in\mathcal{P}_0(\R^n)$ and 
\[
\alpha=
\min
\left(
q^-,\,1
\right)
\min
\left(1,\,
\left(
\frac{p}{q}
\right)^-
\right). 
\]
Then 
\[
\left\| \{f_k+g_k\}_{k=0}^{\infty}\right\|_{\ell^{q(\cdot)}(L^{p(\cdot)})}^{\alpha} 
\le 
\left\| \{f_k\}_{k=0}^{\infty}\right\|_{\ell^{q(\cdot)}(L^{p(\cdot)})}^{\alpha} 
+ \left\|\{g_k\}_{k=0}^{\infty}\right\|_{\ell^{q(\cdot)}(L^{p(\cdot)})}^{\alpha}. 
\]
\end{lemma}

\begin{proof}
Let $r=\min(p^-,1)$ and 
\[
\lambda_1 = \left\| f\right\|_{L^{p(\cdot)}}^{\min(p^-,1)}, \ \ 
\lambda_2 = \left\| g\right\|_{L^{p(\cdot)}}^{\min(p^-,1)}. 
\]
Then we see that 
\begin{align*}
\int_{\R^n}
&\left(
\frac{|f+g|}{(\lambda_1+\lambda_2){^{1/r}}}
\right)^{p(x)} {\rm d}x \notag \\ 
&= 
\int_{\R^n}
\left(
\frac{|f+g|^r}{\lambda_1+\lambda_2}
\right)^{p(x)/r} {\rm d}x \\ 
&\le 
\int_{\R^n}
\left(
\frac{\lambda_1}{\lambda_1+\lambda_2}
\frac{|f|^r}{\lambda_1}
+
\frac{\lambda_2}{\lambda_1+\lambda_2}
\frac{|g|^r}{\lambda_2}
\right)^{p(x)/r} {\rm d}x \\ 
&\le 
\frac{\lambda_1}{\lambda_1+\lambda_2}
\int_{\R^n}
\left(
\frac{|f|^r}{\lambda_1}
\right)^{p(x)/r} {\rm d}x 
+
\frac{\lambda_2}{\lambda_1+\lambda_2}
\int_{\R^n}
\left(
\frac{|g|^r}{\lambda_2}
\right)^{p(x)/r} {\rm d}x \\ 
&=
\frac{\lambda_1}{\lambda_1+\lambda_2}
\int_{\R^n}
\left(
\frac{|f|}{\lambda_1^{1/r}}
\right)^{p(x)} {\rm d}x 
+
\frac{\lambda_2}{\lambda_1+\lambda_2}
\int_{\R^n}
\left(
\frac{|g|}{\lambda_2^{1/r}}
\right)^{p(x)} {\rm d}x \\ 
&\le 1. 
\end{align*}  
This implies {\rm (i)}. 

Next we will prove {\rm (ii)}. 
Let $r=\min(p^-,q^-,1)$ and 
\[
\lambda_1 =\left\| \{f_k\}_{k=0}^{\infty}\right\|_{L^{p(\cdot)}(\ell^{q(\cdot)})}^{\min(p^-,q^-,1)} , \ \ 
\lambda_2 = \left\| \{g_k\}_{k=0}^{\infty}\right\|_{L^{p(\cdot)}(\ell^{q(\cdot)})}^{\min(p^-,q^-,1)} . 
\]
Then we see that 
\begin{align*}
\int_{\R^n}
&\left(
\frac{\{\sum_{k=0}^{\infty}|f_k+g_k|^{q(x)}\}^{1/{q(x)}}}{(\lambda_1+\lambda_2)^{1/r}}
\right)^{p(x)} {\rm d}x \\ 
&\le 
\int_{\R^n}
\left(
\frac{\{\sum_{k=0}^{\infty}(|f_k|^r+|g_k|^r)^{q(x)/r}\}^{r/{q(x)}}}{\lambda_1+\lambda_2}
\right)^{p(x)/r} {\rm d}x \\ 
&\le 
\int_{\R^n}
\left(
\frac{\{\sum_{k=0}^{\infty}(|f_k|^r)^{q(x)/r}\}^{r/{q(x)}}+\{\sum_{k=0}^{\infty}(|g_k|^r)^{q(x)/r}\}^{r/{q(x)}}}{\lambda_1+\lambda_2}
\right)^{p(x)/r} {\rm d}x \\ 
&\le
\frac{\lambda_1}{\lambda_1+\lambda_2}
\int_{\R^n}
\left(
\frac{\{\sum_{k=0}^{\infty}(|f_k|)^{q(x)}\}^{1/{q(x)}}}{\lambda_1^{1/r}}
\right)^{p(x)} {\rm d}x  \notag \\ 
&\qquad + \frac{\lambda_2}{\lambda_1+\lambda_2}\int_{\R^n}
\left(
\frac{\{\sum_{k=0}^{\infty}(|g_k|)^{q(x)}\}^{1/{q(x)}}}{\lambda_2^{1/r}}
\right)^{p(x)} {\rm d}x \\ 
&\le 1. 
\end{align*}
This implies {\rm (ii)}. 

Finally, we will prove {\rm (iii)}. 
Let 
\[
s=\min
\left(
q^-,\,1
\right), 
t= \min
\left(1,\,
\left(
\frac{p}{q}
\right)^-
\right), 
\alpha=st 
\]
and 
\[
\lambda_1=
\left\| \{f_k\}_{k=0}^{\infty}\right\|_{\ell^{q(\cdot)}(L^{p(\cdot)})}^{\alpha}, 
\lambda_2=\left\|\{g_k\}_{k=0}^{\infty}\right\|_{\ell^{q(\cdot)}(L^{p(\cdot)})}^{\alpha}.  
\]
Then we see that 
\begin{align*}
\sum_{k=0}^{\infty} 
&\left\|
\left(
\frac{|f_k+g_k|}
{(\lambda_1+\lambda_2)^{1/st}}
\right)^{q(\cdot)}
\right\|_{L^{\frac{p(\cdot)}{q(\cdot)}}} \\ 
&=\sum_{k=0}^{\infty} 
\left\|
\left(
\frac{|f_k+g_k|^{st}}
{\lambda_1+\lambda_2}
\right)^{q(\cdot)/s}
\right\|_{L^{\frac{p(\cdot)}{tq(\cdot)}}}^{1/t} \\ 
&=\sum_{k=0}^{\infty} 
\left\|
\frac{\lambda_1}{\lambda_1+\lambda_2}
\left(
\frac{|f_k|^{st}}
{\lambda_1}
\right)^{q(\cdot)/s}
+
\frac{\lambda_2}{\lambda_1+\lambda_2}
\left(
\frac{|g_k|^{st}}
{\lambda_2}
\right)^{q(\cdot)/s}
\right\|_{L^{\frac{p(\cdot)}{tq(\cdot)}}}^{1/t} \\ 
&\le 
\sum_{k=0}^{\infty} 
\left(
\frac{\lambda_1}{\lambda_1+\lambda_2}
\left\|
\left(
\frac{|f_k|^{st}}
{\lambda_1}
\right)^{q(\cdot)/s}
\right\|_{L^{\frac{p(\cdot)}{tq(\cdot)}}}
+
\frac{\lambda_2}{\lambda_1+\lambda_2}
\left\|
\left(
\frac{|g_k|^{st}}
{\lambda_2}
\right)^{q(\cdot)/s}
\right\|_{L^{\frac{p(\cdot)}{tq(\cdot)}}}
\right)^{1/t} \\ 
&\le 
\sum_{k=0}^{\infty} 
\frac{\lambda_1}{\lambda_1+\lambda_2}
\left\|
\left(
\frac{|f_k|^{st}}
{\lambda_1}
\right)^{q(\cdot)/s}
\right\|_{L^{\frac{p(\cdot)}{tq(\cdot)}}}^{1/t}
+
\sum_{k=0}^{\infty} 
\frac{\lambda_2}{\lambda_1+\lambda_2}
\left\|
\left(
\frac{|g_k|^{st}}
{\lambda_2}
\right)^{q(\cdot)/s}
\right\|_{L^{\frac{p(\cdot)}{tq(\cdot)}}}^{1/t} \\ 
&= 
\sum_{k=0}^{\infty} 
\frac{\lambda_1}{\lambda_1+\lambda_2}
\left\|
\left(
\frac{|f_k|}
{\lambda_1^{1/{\alpha}}}
\right)^{q(\cdot)}
\right\|_{L^{\frac{p(\cdot)}{q(\cdot)}}}
+
\sum_{k=0}^{\infty} 
\frac{\lambda_2}{\lambda_1+\lambda_2}
\left\|
\left(
\frac{|g_k|}
{\lambda_2^{1/{\alpha}}}
\right)^{q(\cdot)}
\right\|_{L^{\frac{p(\cdot)}{q(\cdot)}}} \\ 
&\le 1. 
\end{align*}
Hence we have {\rm (iii)}. 
\end{proof}

The set 
$\Phi(\jissu^n)$ is the collection of all systems $\theta=\{\theta_j\}_{j=0}^{\infty}\subset\mathcal{S}(\jissu^n)$ such that 
\begin{equation*}
\begin{cases}
\spt \theta_0\subset\{x\,:\,|x|\le 2\}, \\
\spt \theta_j\subset\{x\,:\,2^{j-1}\le |x|\le 2^{j+1}\}\,\,\text{for } j=1,2,\cdots, 
\end{cases}
\end{equation*}
for every multi-index $\alpha$, there exists a positive number $c_{\alpha}$ such that 
\[
2^{j|\alpha|}|D^{\alpha}\theta_j(x)|\le c_{\alpha}
\]
for $j=0,1,\cdots$ and $x\in\jissu^n$ and 
\[
\sum_{j=0}^{\infty}\theta_j(x)=1
\]
for $x\in\jissu^n$. 

Let $\theta$ be a continuous function on $\R^n$ or the sum of finitely many characteristic functions of cubes in $\R^n$. Then $\theta(D)$ is defined by 
$\theta(D)f=\inversefourier[\theta\cdot \fourier f]$. 
 
\begin{definition}
\label{Def:T} 
Let $p(\cdot),q(\cdot)\in C^{\log}(\R^n)\cap\mathcal{P}_0(\jissu^n)$ and $\alpha(\cdot)\in C^{\log}(\R^n)$. 
Let $\theta=\{\theta_j\}_{j=0}^{\infty}\in\Phi(\jissu^n)$.  
Besov space $B_{p(\cdot),q(\cdot)}^{\alpha(\cdot)}(\jissu^n)$ with variable exponents is the collection of $f\in\mathcal{S}'(\jissu^n)$ such that  
\[
||f||_{B_{p(\cdot),q(\cdot)}^{\alpha(\cdot)}} = \left|\left|\left\{2^{j\alpha(\cdot)}\theta_j(D)f  \right\}_{j=0}^{\infty}\right|\right|_{\ell^{q(\cdot)}(L^{p(\cdot)})}<\infty. 
\] 

Triebel--Lizorkin space $F_{p(\cdot),q(\cdot)}^{\alpha(\cdot)}(\jissu^n)$ with variable exponents is the collection of $f\in\mathcal{S}'(\jissu^n)$ such that 
\[
||f||_{F_{p(\cdot),q(\cdot)}^{\alpha(\cdot)}}= 
\left|\left|\left\{2^{j\alpha(\cdot)}\theta_j (D)f  \right\}_{j=0}^{\infty}\right|\right|_{L^{p(\cdot)}(\ell^{q(\cdot)})}<\infty. 
\]

Here $L^{p(\cdot)}(\ell^{q(\cdot)})$ is the space of all sequences $\{g_j\}_0^{\infty}$ of measurable functions on $\jissu^n$ 
such that quasi-norms
\[
||\{g_j\}_{j=0}^{\infty} ||_{L^{p(\cdot)}(\ell^{q(\cdot)})} 
= \left|\left|\left(\sum_{j=0}^{\infty}|g_j(\cdot)|^{q(\cdot)}\right)^{\frac{1}{q(\cdot)}}\right|\right|_{L^{p(\cdot)}}<\infty. 
\]
\end{definition} 

Let $A_{p(\cdot),q(\cdot)}^{\alpha(\cdot)}(\jissu^n)$ be either $B_{p(\cdot),q(\cdot)}^{\alpha(\cdot)}(\jissu^n)$ or $F_{p(\cdot),q(\cdot)}^{\alpha(\cdot)}(\jissu^n)$. 

\subsection{Fundamental results for variable exponents analysis} 

Let $A$ and $B$ be positive constants or positive valued functions and $c$ be a positive constant. 
In this paper, we use the following notations : 
\begin{itemize}
\item If $A \le c B$ hold, then we write $A\lesssim B$.  
\item $A \gtrsim B$ means $B \lesssim A$. 
\item If $A\lesssim B$ and $B \lesssim A$, then we write $A \sim B$. 
\item If there exists a constant $c$ such that $A=cB$, then we write $A\simeq B$. 
\end{itemize}
When we emphasize that the constant $c$ as above is depend on some parameters $\alpha$, $\beta$, $\gamma, \cdots$, 
then we use the following notations : 
\begin{itemize}
\item We write $A \lesssim_{ \alpha, \beta, \gamma,\cdots  } B$ instead of $A\lesssim B$. 
\item We write $A \gtrsim_{\alpha, \beta, \gamma,\cdots} B$ instead of $A \gtrsim B$.
\item We write $A \sim_{\alpha, \beta, \gamma,\cdots} B$ instead of $A \sim B$. 
\item We write $A\simeq_{\alpha, \beta, \gamma,\cdots} B$ instead of $A\simeq B$.  
\end{itemize}

Similarly to classical theory, the following H\"older type inequalities \cite[Theorem 2.3]{Huang-Xu} hold. 
\begin{theorem}[{\rm \cite[Theorem 2.3]{Huang-Xu}}]
\label{Holder type inequality}
Let $p_0(\cdot), p_1(\cdot), p_2(\cdot)\in\mathcal{P}_0(\R^n)$ with $\frac{1}{p_0(\cdot)}=\frac{1}{p_1(\cdot)} + \frac{1}{p_2(\cdot)}$. 
Then we have $\| fg \|_{L^{p_0(\cdot)}}\lesssim \|f\|_{L^{p_1(\cdot)}}\|g\|_{L^{p_2(\cdot)}}$ for any $f\in L^{p_1(\cdot)}(\R^n)$ and $g\in L^{p_2(\cdot)}(\R^n)$. 
\end{theorem}

D.\,Cruz-Uribe et al.\,\cite{Cruz} proves the boundedness of classical operators, for example, singular integral operators and  
fractional integral operators on the space $L^{p(\cdot)}(\jissu^n)$.   
If $f(\cdot)$ is a complex-valued locally Lebesgue-integrable function on $\jissu^n$, then 
\[
(\mathcal{M}f)(x) = \sup\frac{1}{|B|}\int_B|f(y)|\,{\rm d}y
\]
is called Hardy--Littlewood maximal operator, where the supremum is taken over all balls $B$ centered at $x$. 
Furthermore, let $0<r\le 1$. If $f(\cdot)$ is a complex-valued locally Lebesgue-integrable function on $\jissu^n$, then 
\[
(\mathcal{M}_r f)(x) = \left( \sup\frac{1}{|B|}\int_B|f(y)|^r\,{\rm d}y\right)^{1/r}
\]
is also called Hardy--Littlewood maximal operator. 
The next theorem is corresponding to the well-known maximal vector-valued inequality in the classical theory. 

\begin{theorem}[{\rm \cite[Corollary 2.1]{Cruz}}] 
\label{thm:max}
If $p(\cdot)\in\mathcal{B}(\jissu^n)$, then, for all $q\in (1, \infty)$, there exists a constant $c$ such that 
\begin{equation}
||\{\mathcal{M}f_k\}_{k=0}^{\infty}||_{L^{p(\cdot)}(\ell^q)}\le c||\{f_k\}_{k=0}^{\infty}||_{L^{p(\cdot)}(\ell^q)} \label{max}
\end{equation}
for all sequences $\{f_k\}_{k=0}^{\infty}\subset L^{p(\cdot)}(\jissu^n)$. 
\end{theorem}

It is well-known that $(\ref{max})$ does not always hold if $q(\cdot)\in\mathcal{B}(\jissu^n)$ is not 
a constant function. However, Diening et\,al.\,\cite{Diening} showed the following helpful theorem 
which takes the place of Theorem $\ref{thm:max}$. Let 
\[
\eta_m(x)=(1+|x|)^{-m} \ \ \text{and} \ \ \eta_{\nu,m}(x)=2^{\nu n}\eta_m(2^{\nu}x)
\]
for $\nu\in\shizensu_0$ and a positive real number $m$. 
\begin{theorem}[{\rm \cite[Theorem 3.2]{Diening}}]  
\label{thm:max2}
Let $p(\cdot),q(\cdot)\in C^{\log}(\R^n)$ with $1<p^-\le p^+<\infty$ and $1<q^-\le q^+<\infty$. 
Then the inequality 
\[
||\{\eta_{k,m}\ast f_k\}_{k=0}^{\infty}||_{L^{p(\cdot)}(\ell^{q(\cdot)})} \le c||\{f_k\}_{k=0}^{\infty}||_{L^{p(\cdot)}(\ell^{q(\cdot)})}
\]
holds for every sequence $\{f_k\}_{k=0}^{\infty}$ of $L_{\rm loc}^1$-functions and $m>n$. 
\end{theorem}

Almeida et al.\,\cite{Almeida} showed the following helpful theorem for $\ell^{q(\cdot)}(L^{p(\cdot)})$ quasi norm.  

\begin{theorem}[{\rm \cite[Lemma 4.7]{Almeida}}]  
\label{thm:max3}
Let $p(\cdot),q(\cdot)\in C^{\log}(\R^n)$ with $1<p^-\le p^+<\infty$ and $1<q^-\le q^+<\infty$. 
Then the inequality 
\[
||\{\eta_{k,m}\ast f_k\}_{k=0}^{\infty}||_{\ell^{q(\cdot)}(L^{p(\cdot)})} \lesssim ||\{f_k\}_{k=0}^{\infty}||_{\ell^{q(\cdot)}(L^{p(\cdot)})}
\]
holds for every sequence $\{f_k\}_{k=0}^{\infty}$ of $L_{\rm loc}^1$-functions and $m>2n$. 
\end{theorem}

\begin{remark}
\label{remark maximal}
Let $p(\cdot)\in C^{\log}(\R^n)$ with $1<p^-\le p^+<\infty$. 
It is easy to see that 
the inequality 
\[
||\{\eta_{k,m}\ast f_k\}_{k=0}^{\infty}||_{\ell^{\infty}(L^{p(\cdot)})} \lesssim ||\{f_k\}_{k=0}^{\infty}||_{\ell^{\infty}(L^{p(\cdot)})}
\]
holds for every sequence $\{f_k\}_{k=0}^{\infty}$ of $L_{\rm loc}^1$-functions and $m>n$ by Theorem $\ref{thm:max2}$. 

By the proof of \cite[Lemma 5.4]{Diening}, we see that 
the inequality 
\[
||\{\eta_{k,m}\ast f_k\}_{k=0}^{\infty}||_{L^{p(\cdot)}(\ell^{\infty})} \lesssim ||\{f_k\}_{k=0}^{\infty}||_{L^{p(\cdot)}(\ell^{\infty})}
\]
holds for every sequence $\{f_k\}_{k=0}^{\infty}$ of $L_{\rm loc}^1$-functions and $m>2n$. 
\end{remark}

We often use the following relation between $s(x)$ and $s(y)$. 
\begin{lemma}[{\rm \cite[Lemma 6.1]{Diening}}] 
\label{lemma:2}
Let $s(\cdot)\in C^{\log}(\R^n)$. Then there exists a positive constant $c$ such that 
\[
2^{ks(x)}\eta_{k,2m}(x-y) \le c 2^{ks(y)}\eta_{k,m}(x-y)
\]
for all $x,y\in\jissu^n$ and $m>C_{\log}(s)$. 
\end{lemma} 

\begin{lemma}
\label{lemma:r}
Let $r>0$, $\nu\in\shizensu_0$ and $m\ge n+1$. Then there exists $c=c(r,m,n)>0$ such that 
\[
\frac{|f(x-z)|}{(1+|2^{\nu} z|)^{\frac{m}{r}}} \le c \left(\eta_{\nu,m}\ast |f|^r(x) \right)^{\frac{1}{r}} 
\]
for all $x\in\jissu^n$, $z\in\jissu^n$ and $f\in\mathcal{S}'(\jissu^n)$ with $\spt \fourier f \subset \{\xi$ $:$ $|\xi|\le 2^{\nu+1}\}$. 
\end{lemma}

\begin{definition}
\label{def:1}
\text{\rm (i)} Let $\Omega$ be a compact subset of $\jissu^n$. 
Then $\mathcal{S}^{\Omega}(\jissu^n)$ denotes the space of all elements $\varphi\in\mathcal{S}(\jissu^n)$ satisfying  
$\spt\fourier\varphi\subset\Omega$.  \\ 
\text{\rm (ii)} Let $p(\cdot), q(\cdot) \in C^{\log}(\R^n)\cap \mathcal{P}_0(\jissu^n)$. 
For a sequence  $\Omega=\{\Omega_k\}_{k=0}^{\infty}$ of compact subsets of $\jissu^n$, 
$L_{p(\cdot)}^{\Omega}$ is the space of all sequences $\{f_k\}_{k=0}^{\infty}$ of $\mathcal{S}'(\jissu^n)$ such that 
\begin{equation}
\spt \fourier f_k\subset \Omega_k 
\end{equation}
and $||f_k||_{L^{p(\cdot)}}<\infty$ for $k=0,1,2,\cdots$. 
\end{definition}
The author \cite{Noi} proved the following Theorem. 
\begin{theorem}
\label{thm:5}
Let $p(\cdot),q(\cdot)\in C^{\log}(\R^n)\cap\mathcal{P}_0(\jissu^n)$ or $q(\cdot)\equiv \infty$ and 
$s(\cdot)\in C^{\log}(\R^n)$.
Let $\Omega=\{\Omega_k\}_{k=0}^{\infty}$ be a sequence of compact subsets of $\jissu^n$ such that 
$\Omega_k\subset\{\xi\in\jissu^n$ $:$ $|\xi|\le2^{k+1}\}$. \\ 
{\rm (i)} If $v>\frac{n}{2}+\frac{4\max\{n, C_{\log}(s)\}}{\min\{p^-,q^-\}}$, then there exists a number $c$ such that 
\begin{align*}
||\{2^{k s(\cdot)} M_k(D) f_k\}_{k=0}^{\infty}&||_{L^{p(\cdot)}(\ell^{q(\cdot)})} 
\le c\sup_{l}||M_{l}(2^{l}\cdot)||_{H_2^v}||\{2^{ks(\cdot)}f_k\}_{0}^{\infty}||_{L^{p(\cdot)}(\ell^{q(\cdot)})}
\end{align*}
for $\{f_k(x)\}_{k=0}^{\infty}\in L_{p(\cdot)}^{\Omega}$ and $\{ M_k(x)\}_{k=0}^{\infty}\in H_2^v(\jissu^n)$. \\ 
{\rm (ii)} If $v>\frac{n}{2}+\frac{4\max\{2n, C_{\log}(s)\}}{\min\{p^-,q^-\}}$, then there exists a number $c$ such that 
\begin{align*}
||\{2^{ks(\cdot)} M_k(D) f_k\}_{k=0}^{\infty}&||_{\ell^{q(\cdot)}(L^{p(\cdot)})} 
\le c\sup_{l}||M_{l}(2^{l}\cdot)||_{H_2^v}||\{2^{ks(\cdot)}f_k\}_{0}^{\infty}||_{\ell^{q(\cdot)}(L^{p(\cdot)})}
\end{align*}
for $\{f_k(x)\}_{k=0}^{\infty}\in L_{p(\cdot)}^{\Omega}$ and $\{ M_k(x)\}_{k=0}^{\infty}\in H_2^v(\jissu^n)$. 
\end{theorem}
Therefore, we obtain the lifting properties as a corollary of Theorem $\ref{thm:5}$. 
\begin{corollary}[Lifting properties]
\label{cor:lifting} 
Let $p(\cdot),q(\cdot)\in C^{\log}(\R^n)\cap\mathcal{P}_0(\jissu^n)$ and 
$s(\cdot)\in C^{\log}(\R^n)$. 
Let $\sigma\in\R$, $k=1,2,\cdots,m$ and $m\in\N$. 
Then 
\[
\partial_k:A_{p(\cdot),q(\cdot)}^{s(\cdot)} \longrightarrow A_{p(\cdot),q(\cdot)}^{s(\cdot)-1}
\]
is a continuous map. 
Furthermore, we have following properties: \\ 
{\rm (1) } The linear mapping $(1-\Delta)^{\sigma}$ is an isomorphism 
between 
$A_{p(\cdot),q(\cdot)}^{s(\cdot)}$ 
and 
$A_{p(\cdot),q(\cdot)}^{s(\cdot)-2\sigma}$. \\ 
{\rm (2) } The linear mapping $\left(1+(-\Delta)^m\right)$ is an isomorphism 
between 
$A_{p(\cdot),q(\cdot)}^{s(\cdot)}$ 
and 
$A_{p(\cdot),q(\cdot)}^{s(\cdot)-2m}$. \\ 
{\rm (3) } The linear mapping $(1+\partial_1^{4m}+\cdots +\partial_n^{4m})$ is an isomorphism 
between 
$A_{p(\cdot),q(\cdot)}^{s(\cdot)}$ 
and 
$A_{p(\cdot),q(\cdot)}^{s(\cdot)-4m}$.
\end{corollary}

\section{Embeddings for $A_{p(\cdot), q(\cdot)}^{s(\cdot)}(\R^n)$}
\label{section Embeddings}

In this section, we deal with embeddings for $A_{p(\cdot), q(\cdot)}^{s(\cdot)}(\R^n)$. 

\begin{definition} We define three linear spaces consist of bounded functions. 
\begin{enumerate}
\item Denote by ${\rm BC}$ the linear space of all bounded continuous functions. 
Let $f\in {\rm BC}$. Then we define $||f||_{{\rm BC}}$ such that $||f||_{{\rm BC}}:=||f||_{L^{\infty}}$. 
\item Let $m\in\N_0$. Then, denote by $\mathcal{B}^m$ the linear space of functions $f\,:\,\R^n\longrightarrow \C$ such that 
$f\in C^m$ and $\partial^{\alpha}f\in{\rm BC}$ for any multi-index $\alpha$ with $|\alpha|\le m$. 
We define the norm such that 
\[
\| f\|_{\mathcal{B}^m} = \sum_{|\alpha|\le m}\left\| \partial^{\alpha}f\right\|_{ {\rm BC} }. 
\]
\item Denote by ${\rm BUC}$ the linear space consist of bounded uniformly continuous functions. 
 Then we define $||f||_{{\rm BUC}}$ such that $||f||_{{\rm BUC}}:=||f||_{L^{\infty}}$.
\end{enumerate}
\end{definition}
Then the following embeddings are well-known. 
\begin{lemma}
\label{1.1.20} 
For each $m\in\N_0$. $B_{\infty,1}^m \hookrightarrow \mathcal{B}^m$ and $B_{\infty,1}^0 \hookrightarrow {\rm BUC}$ holds. 
\end{lemma}
Furthermore, we can prove the following embedding. 
\begin{proposition}
\label{proposition 3.1.33} 
Let $p(\cdot),q(\cdot)\in C^{\log}(\R^n)\cap\mathcal{P}_0(\R^n)$ and $s(\cdot)\in C^{\log}(\R^n)$ satisfy $s(\cdot)>\frac{n}{p(\cdot)}$ and $0>\left(\frac{n}{p(\cdot)}-s(\cdot) \right)^+$. 
Then $A_{p(\cdot),q(\cdot)}^{s(\cdot)}(\R^n)\hookrightarrow {\rm BUC}$.  
\end{proposition}
To prove Proposition $\ref{proposition 3.1.33}$, we need Theorem $\ref{thm:3}$. 

\begin{theorem}
\label{thm:3} 
 Let $s(\cdot)\in C^{\log}(\R^n)$ and $p(\cdot)$, $q(\cdot)\in C^{\log}(\R^n)\cap\mathcal{P}_0(\jissu^n)$ satisfy $0< p(\cdot)\le q(\cdot) < \infty$. 
Let $k>0$ be a fixed number and $B_{2^{j+k}}=\{x\in\jissu^n$ $:$ $\abs{x}\le 2^{j+k}\}$.  
If $\varphi_j \in\mathcal{S}^{B_{2^{j+k}}}(\jissu^n)$, then there exists a positive number $c$ such that  
\[
|| 2^{js(\cdot)} \varphi_j ||_{L^{ q(\cdot) } } 
\le c  \left\|  2^{ js(\cdot) +\frac{nj}{p(\cdot)}-\frac{nj}{q(\cdot)}} \varphi_j  \right\|_{L^{ p(\cdot) }} 
\]
and 
\[  
|| 2^{js(\cdot)}\varphi_j ||_{L^{\infty}}\le c \left\|  2^{ js(\cdot)+\frac{nj}{p(\cdot)}}  \varphi_j \right\|_{L^{p(\cdot)} }, 
\]
where  $c$ is independent of $j$. 
\end{theorem} 

We can prove Theorem \ref{thm:3} by an argument similar to proof of \cite[Theorem 4.7]{Noi}. 
So we omit the proof.  

Now we prove Proposition $\ref{proposition 3.1.33}$. 
\begin{proof}[Proof of Proposition $\ref{proposition 3.1.33}$] 
Let $f\in A_{p(\cdot),q(\cdot)}^{s(\cdot)}(\R^n)$. 
By Theorem $\ref{thm:3}$, we have 
\[
\left\| \varphi_j(D)f \right\|_{L^{\infty} }\lesssim \left\| 2^{\frac{jn}{p(\cdot)}}\varphi_j(D)f\right\|_{L^{p(\cdot)}}. 
\]
Therefore, we see that 
\begin{align*}
\sum_{j=0}^{\infty} \left\| \varphi_j(D)f \right\|_{L^{\infty}} \lesssim \sum_{j=0}^{\infty}  \left\| 2^{\frac{jn}{p(\cdot)}}\varphi_j(D)f\right\|_{L^{p(\cdot)}}  
 \lesssim \sum_{j=0}^{\infty}  2^{j\left( \frac{n}{p(\cdot)}  -s(\cdot) \right)^+ } \left\| f \right\|_{A_{p(\cdot),q(\cdot)}^{s(\cdot)}}.  
\end{align*} 
This implies that $A_{p(\cdot),q(\cdot)}^{s(\cdot)}(\R^n) \hookrightarrow B_{\infty,1}^0(\R^n) \hookrightarrow {\rm BUC}$ by Lemma $\ref{1.1.20}$. 
\end{proof}

Almeida and H\"ast\"o \cite[Theorem 6.4]{Almeida} proved the Sobolev embedding for $B_{p(\cdot),q(\cdot)}^{s(\cdot)}$. 
We need a spacial case of Sobolev embeddings for $F_{p(\cdot),q(\cdot)}^{s(\cdot)}$ and $B_{p(\cdot),q(\cdot)}^{s(\cdot)}$. 
\begin{proposition}
\label{Sobolev}
Let $p(\cdot), q(\cdot)\in\mathcal{P}(\jissu^n)$ and $s(\cdot)\in C^{\log}(\R^n)$. 
Then we have 
\[
A_{p(\cdot),q(\cdot)}^{s(\cdot)}(\R^n) \hookrightarrow B_{\infty,\infty}^{s(\cdot)-\frac{n}{p(\cdot)}}(\R^n). 
\]
\end{proposition} 

\begin{proof}
As we mentioned above, Sobolev embedding for $B_{p(\cdot),q(\cdot)}^{s(\cdot)}$ has been proved by Almeida and H\"ast\"o \cite[Theorem 6.4]{Almeida}. 
Hence, we prove the $F_{p(\cdot),q(\cdot)}^{s(\cdot)}$ case. 

Let $f\in F_{p(\cdot),q(\cdot)}^{s(\cdot) }(\R^n)$. Without loss of generality, we can assume that $\| f\|_{F_{p(\cdot),q(\cdot)}^{s(\cdot) }}=1$. 
By Theorem $\ref{thm:3}$, we have 
\[
\left\| 2^{j\left( s(\cdot)-\frac{n}{p(\cdot)}\right) }\varphi_j(D)f \right\|_{L^{\infty} }\lesssim \left\| 2^{js(\cdot) }\varphi_j(D)f\right\|_{L^{p(\cdot)}} \lesssim 1. 
\]
By taking the $\ell^{\infty}$ norm, we have $ \| f\|_{F_{p(\cdot),q(\cdot)}^{s(\cdot) }}\lesssim \| f\|_{ B_{\infty,\infty}^{s(\cdot)-\frac{n}{p(\cdot)}} }$. 
\end{proof}

Almeida and H\"ast\"o \cite{Almeida} proved the following inclusion for $A_{p(\cdot),q(\cdot)}^{s(\cdot)}(\jissu^n)$. 
\begin{proposition}[{ \cite[Theorem 6.1]{Almeida} }]
\label{dual-thm} 
Let $p(\cdot)\in\mathcal{P}(\jissu^n)$ and $s(\cdot)\in C^{\log}(\R^n)$.

\text{\rm (i)} 
Let $q_1(\cdot), q_2(\cdot)\in\mathcal{P}(\jissu^n)$ satisfy  $q_1(\cdot)\le q_2(\cdot)$. Then 
\[
B_{p(\cdot),q_1(\cdot)}^{s(\cdot)}(\jissu^n)\subset B_{p(\cdot),q_2(\cdot)}^{s(\cdot)}(\jissu^n). 
\]

\text{\rm (ii)} 
Let $q_1(\cdot), q_2(\cdot)\in\mathcal{P}(\jissu^n)$ and $s_1(\cdot), s_2(\cdot)\in C^{\log}(\R^n)$ such that $\inf_{x\in\R^n}(s_1(x)-s_2(x))>0$. Then 
\[
B_{p(\cdot),q_1(\cdot)}^{s_1(\cdot)}(\jissu^n)\subset B_{p(\cdot),q_2(\cdot)}^{s_2(\cdot)}(\jissu^n). 
\]

\text{\rm (iii)} 
If $q(\cdot)\in\mathcal{P}(\R^n)$, then 
\[
B_{p(\cdot),\min\{p(\cdot),q(\cdot)\}}^{s(\cdot)}(\jissu^n)\subset F_{p(\cdot),q(\cdot)}^{s(\cdot)}(\jissu^n)\subset 
B_{p(\cdot),\max\{p(\cdot),q(\cdot)\}}^{s(\cdot)}(\jissu^n). 
\]
\end{proposition} 

We have the counterpart of {\rm (i)} and {\rm (ii)} of Proposition $\ref{dual-thm}$ for $F_{p(\cdot),q(\cdot)}^{s(\cdot)}(\R^n)$ by using similar argument of \cite[Theorem 6.1]{Almeida}.  

\begin{proposition}
\label{dual-thm1} 
Let $p(\cdot)\in\mathcal{P}(\jissu^n)$ and $s(\cdot)\in C^{\log}(\R^n)$.

\text{\rm (i)} 
Let $q_1(\cdot), q_2(\cdot)\in\mathcal{P}(\jissu^n)$ satisfy  $q_1(\cdot)\le q_2(\cdot)$. Then 
\[
F_{p(\cdot),q_1(\cdot)}^{s(\cdot)}(\jissu^n)\subset F_{p(\cdot),q_2(\cdot)}^{s(\cdot)}(\jissu^n). 
\]

\text{\rm (ii)} 
Let $q_1(\cdot), q_2(\cdot)\in\mathcal{P}(\jissu^n)$ and $s_1(\cdot), s_2(\cdot)\in C^{\log}(\R^n)$ such that $\inf_{x\in\R^n}(s_1(x)-s_2(x))>0$. Then 
\[
F_{p(\cdot),q_1(\cdot)}^{s_1(\cdot)}(\jissu^n)\subset F_{p(\cdot),q_2(\cdot)}^{s_2(\cdot)}(\jissu^n). 
\]
\end{proposition}

\section{Decompositions of Besov and Triebel--Lizorkin spaces with variable exponents} 

In order to introduce a quarkonial decomposition for $A_{p(\cdot),q(\cdot)}^{s(\cdot)}(\R^n)$, 
we use an atomic decomposition for  $A_{p(\cdot),q(\cdot)}^{s(\cdot)}(\R^n)$. 
Kempka \cite{Kempka} proved the atomic decomposition for $2$-microlocal Triebel--Lizorkin with variable exponents and $2$-microlocal Besov spaces with variable exponents, but 
summability index $q$ was constant in the Besov case. 
In a case that all exponents are variable exponents in Besov spaces,  Drihem \cite{Drihem} proved the atomic decomposition for $B_{p(\cdot),q(\cdot)}^{s(\cdot)}(\R^n)$. 
As we mentioned in Introduction, Moura, Neves and Schneider \cite{Moura} proved the boundedness of 
the trace operator for $2$-microlocal Besov spaces by using atomic decomposition, but summability index was constant. 
By using the results, we obtain the boundedness of the Trace operator for $B_{p(\cdot), q}^{s(\cdot)}(\R^n)$ under the condition  
\begin{equation}
s^->\frac{1}{p^-}+(n-1)\left( \frac{1}{\min(1, p^-) }-1\right). \label{13-12-26-1}
\end{equation} 
In this condition, it was considered essential infimum and essential supremum to each variable exponent $p(\cdot)$ and $s(\cdot)$. 
However, as a condition of \cite[Theorem 3.13]{Diening} (although the theorem is Triebel--Lizorkin spaces with variable exponents case), 
we would like to consider a condition 
where $s^-$ and $p^-$ are replaced by $s(\cdot)$ and $p(\cdot)$ respectively in $(\ref{13-12-26-1})$.  
That is, we consider the boundedness of Trace operator for $B_{p(\cdot),q(\cdot)}^{s(\cdot)}(\R^n)$ with the condition 
\begin{equation}
\mathop{{\rm{ess}}\,\,\rm{inf}}_{x\in\R^n}\left[ s(\cdot)-\left\{ \frac{1}{p(\cdot)} +(n-1)\left( \frac{1}{\min(1, p(\cdot), q(\cdot))} -1 \right) \right\}  \right] >0,  \label{13-12-28-1}
\end{equation}
which takes the place of the condition $(\ref{13-12-26-1})$. 

In the space $F_{p(\cdot),q(\cdot)}^{s(\cdot)}(\R^n)$ case, as we mentioned in Introduction, Diening, H\"ast\"o and Roudenko \cite{Diening} 
proved the boundedness of Trace operator ( Theorem \ref{theorem 5.4.4} (i) ). 
In this paper, we consider the further results of boundedness of Trace operator ( Theorem \ref{theorem 5.4.4} (ii) ). 

To this end, 
we slightly change the definition of smooth atom which was introduced in Diening, H\"ast\"o and Roudenko \cite{Diening} and 
a part of results on atomic decompositions by Kempka\cite{Kempka} and Drihem \cite{Drihem}. 
Therefore, the next subsection \ref{sub-sec} essentially overlap with the works of Kempka \cite{Kempka} and Drihem \cite{Drihem}.  

\subsection{Atomic decomposition for Besov and Triebel--Lizorkin spaces with variable exponents}
\label{sub-sec}

We define $\sigma_{p(\cdot),q(\cdot)}$ and $\sigma_{p(\cdot)}$ such that 
\begin{align*}
\sigma_{p(\cdot),q(\cdot)} &= n\left(\frac{1}{\min(1, p(\cdot), q(\cdot))}-1 \right) \\ 
\sigma_{p(\cdot)} &= n\left(\frac{1}{\min(1, p(\cdot))}-1 \right), 
\end{align*}
where $n$ is the spacial dimension. 

Let $\nu\in\N_0$ and $m\in\Z^n$. Then $Q_{\nu,m}=\prod_{i=1}^{n}[2^{-\nu}m_i, 2^{-\nu}(m_i+1))$ and 
$\chi_{\nu,m}$ is a characteristic function on $Q_{\nu,m}$. 

\begin{definition}
\label{definition 5.1.1} 
Let $p(\cdot),q(\cdot)\in\mathcal{P}_0(\R^n)\cap C^{\log}(\R^n)$ and $s(\cdot)\in C^{\log}(\R^n)$. 
Let $\nu\in\N_0$, $m\in\Z^n$. 
For double-index complex valued sequence $\lambda=\{\lambda_{\nu,m}\}_{\nu,m}$, 
we define 
\begin{align}
\left\| 
\lambda
\right\|_{b_{p(\cdot),q(\cdot)}^{s(\cdot)}} 
&:= 
\left\| 
\left\{
2^{\nu s(\cdot)} \sum_{m\in\Z^n} \lambda_{\nu,m}\chi_{\nu,m}
\right\}_{\nu=0}^{\infty} 
\right\|_{\ell^{q(\cdot)}(L^{p(\cdot)})} \label{besov atom suuretsu norm} \\ 
\left\| 
\lambda
\right\|_{f_{p(\cdot),q(\cdot)}^{s(\cdot)}} 
&:= 
\left\| 
\left\{
2^{\nu s(\cdot)} \sum_{m\in\Z^n} \lambda_{\nu,m}\chi_{\nu,m}
\right\}_{\nu=0}^{\infty}
\right\|_{L^{p(\cdot)}(\ell^{q(\cdot)})}. \label{triebel atom suuretsu norm}
\end{align}
We say that $\lambda\in b_{p(\cdot),q(\cdot)}^{s(\cdot)}$ if $(\ref{besov atom suuretsu norm})$ is finite and 
say that $\lambda\in f_{p(\cdot),q(\cdot)}^{s(\cdot)}$ if $(\ref{triebel atom suuretsu norm})$ is finite. 
\end{definition}

Let $a_{p(\cdot),q(\cdot)}^{s(\cdot)}$ be either $b_{p(\cdot),q(\cdot)}^{s(\cdot)}$ or $f_{p(\cdot),q(\cdot)}^{s(\cdot)}$. 

\begin{definition}[Atom]
\label{definition 5.1.2} 
Let $p(\cdot),q(\cdot)\in\mathcal{P}_0(\R^n)\cap C^{\log}(\R^n)$ and $s(\cdot)\in C^{\log}(\R^n)$. 
Let $K\in \N_0$, $L\in\Z$, $\nu\in\N_0$, $m\in\Z^n$ and let $\gamma>1$. 

{\rm (1)} A $K$-times continuous differentiable function $a\in C^{K}(\R^n)$ is called $[K,L]$-atom centered at $Q_{0,m}$, if 
$\spt a\subset \gamma Q_{0,m}$ and 
\begin{equation}
\left\| \partial^{\alpha} a \right\|_{\infty} \le 1, \ \ |\alpha|\le K. \label{5.2}
\end{equation}

{\rm (2)} A $K$-times continuous differentiable function $a\in C^{K}(\R^n)$ is called $[K,L]$-atom centered at $Q_{\nu,m}$, if 
$\spt a\subset \gamma Q_{\nu,m}$, 
\begin{equation}
\left\| \partial^{\alpha} a \right\|_{\infty} \le 2^{\nu|\alpha|}, \ \ |\alpha|\le K \label{5.3}
\end{equation}
and 
\begin{equation}
\int_{\R^n}x^{\beta}a(x){\rm d}x=0, \ \ |\beta|\le L. \label{5.4}
\end{equation}
The condition $(\ref{5.4})$ is called moment condition. 
If $L\le -1$, then no moment condition $(\ref{5.4})$ required. 
\end{definition}

To prove the trace theorem, we need Theorem $\ref{theorem 5.1.6-Besov}$. 
We define the family of $[K, L]$ smooth atoms. 

\begin{definition}
Let $K\in\N_0$ and $L\,:\,\R^n\to \R$. The family $\{a_{\nu,m}\}_{\nu\in\N_0, m\in\Z^n}$ is said to be a family of $[K, L]$ smooth atoms 
 if $a_{\nu,m}$ is a  $[K, \lfloor L_{Q_{\nu,m}}^-\rfloor]$ atom centered at $Q_{\nu,m}$ for any $\nu\in\N_0$ and $m\in\Z^n$. 
 Here $\lfloor A \rfloor =\max\{ n\in\Z\,:\, n\le A\}$ and $L_{Q_{\nu,m}}^- = {\rm ess}\inf_{x\in Q_{\nu, m}}L(x) $. 
\end{definition}

\begin{definition}
\label{smooth atom} 
$(1)$ We say that $\{ a_{\nu, m} \}_{\nu\in\N_0, m\in\Z^n}$ is a family of smooth atoms for $F_{p(\cdot),q(\cdot)}^{s(\cdot)}(\R^n)$ if it is a 
family of $[K, L+\epsilon]$ smooth atoms, where $K>s^+$ and 
\begin{equation}
L(\cdot)= \sigma_{p(\cdot),q(\cdot)} -s(\cdot)  \label{13-12-27-1}
\end{equation}
for some constant $\epsilon>0$. 

\noindent
$(2)$ We say that $\{ a_{\nu, m} \}_{\nu\in\N_0, m\in\Z^n}$ is a family of smooth atoms for $B_{p(\cdot),q(\cdot)}^{s(\cdot)}(\R^n)$ if it is a 
family of $[K, L+\epsilon]$ smooth atoms, where $K>s^+$ and 
\[
L(\cdot)= \sigma_{p(\cdot)} -s(\cdot)
\]
for some constant $\epsilon>0$.  
\end{definition}

\begin{remark}
\label{remark8-10}

$(\rm i)$ 
Drihem \cite{Drihem} proved the atomic decomposition for $B_{p(\cdot),q(\cdot)}^{s(\cdot)}(\R^n)$ 
under the moment condition $\displaystyle L=\lfloor \sigma_{p^-}-s^-\rfloor$ and  
Kempka \cite{Kempka} proved the atomic decomposition for $F_{p(\cdot),q(\cdot)}^{s(\cdot)}(\R^n)$ 
under the moment condition $\displaystyle L=\lfloor \sigma_{p^-,q^-}-s^-\rfloor$.

\noindent 
$({\rm ii})$  A family of smooth atoms for  $F_{p(\cdot),q(\cdot)}^{s(\cdot)}(\R^n)$ was introduced by Diening, H\"ast\"o and Roudenko \cite{Diening}. 
 Diening, H\"ast\"o and Roudenko \cite{Diening} defined the smooth atoms with 
$L(\cdot) = \sigma_{p(\cdot),q(\cdot)}-s(\cdot)$. 
If $\inf_{x\in\R^n} [s(x)-\sigma_{p(x),q(x)}]>0$, then we do not need the moment condition 
for the family of smooth atoms. 
For this reason,  Diening, H\"ast\"o and Roudenko \cite{Diening} proved the boundedness of Trace operator under the condition 
\[
\inf\left[ s(\cdot)-\left\{ \frac{1}{p(\cdot)} +(n-1)\left( \frac{1}{\min(1, p(\cdot))} -1 \right) \right\}  \right] >0. 
\]

\noindent
$({\rm iii})$ Let $\{ a_{\nu, m} \}_{\nu\in\N_0, m\in\Z^n}$ be a family of smooth atoms for $A_{p(\cdot),q(\cdot)}^{s(\cdot)}(\R^n)$. 
Then there exists a $\epsilon>0$ such that the atoms $a_{\nu, m}$ are $[K, L+4\epsilon]$ smooth atoms, where $L$ as in Definition $\ref{smooth atom}$.  
By the uniform continuity of $p(\cdot)$, $q(\cdot)$ and $s(\cdot)$, there exists a non negative integer $\nu_0$ 
such that  $L_{Q_{\nu_0, m}}^- > ( \sigma_{p(\cdot)} )_{Q_{\nu_0, m}}^+ -s_{Q_{\nu_0, m}}^- -\epsilon $ 
and $s_{Q_{\nu_0, m}}^->s_{Q_{\nu_0, m}}^+-\epsilon$ for any $m\in \Z^n$. 
Since  $p(\cdot)$, $q(\cdot)$ and $s(\cdot)$ also have a limit at infinity, there exists  compact sets $K\subset \R^n$ such that 
 $L_{\R^n\setminus K }^- > ( \sigma_{p(\cdot)} )_{ \R^n\setminus K }^+ -s_{  \R^n\setminus K }^- -\epsilon $ 
and $s_{\R^n\setminus K}^->s_{ \R^n\setminus K}^+-\epsilon$. 
Then, since $K\subset \R^n$ is a compact set, we can choose dyadic cubes $\Omega_i$, $i=1,2,\cdots, R$, of level $\nu_0$ such that $K\subset \cup_{i=1}^M \Omega_i $. 
Furthermore we define $\Omega_0 = \R^n\setminus \cup_{i=1}^R\Omega_i$. 
These implies that $L_{\Omega_i}^- >  ( \sigma_{p(\cdot)} )_{ \Omega_i }^+ -s_{  \Omega_i}^- -\epsilon$ holds for any $i=0,1,\cdots ,R$. 
Note that if $Q_{\nu, m} \subset \Omega_i$, then 
\[
L_{Q_{\nu, m}}^-\ge L_{\Omega_i}^- \ge (\sigma_{p(\cdot)})_{\Omega_i}^+ -s_{\Omega_i}^--\epsilon \ge (\sigma_{p(\cdot)})_{Q_{\nu, m} }^+ -s_{Q_{\nu, m} }^--\epsilon
\]
for $B_{p(\cdot),q(\cdot)}^{s(\cdot)}(\R^n)$ and 
\[
L_{Q_{\nu, m}}^-\ge L_{\Omega_i}^- \ge (\sigma_{p(\cdot),q(\cdot)})_{\Omega_i}^+ -s_{\Omega_i}^--\epsilon \ge (\sigma_{p(\cdot),q(\cdot)})_{Q_{\nu, m} }^+ -s_{Q_{\nu, m} }^--\epsilon
\]
for $F_{p(\cdot),q(\cdot)}^{s(\cdot)}(\R^n)$. 
Hence, if $Q_{\nu, m}\subset\Omega_i$, then $a_{\nu, m}$ is a $[K,  (\sigma_{p(\cdot)})_{\Omega_i}^+ -s_{\Omega_i}^-+3\epsilon]$ smooth atom for $B_{p(\cdot),q(\cdot)}^{s(\cdot)}(\R^n)$ 
and is also a  $[K,  (\sigma_{p(\cdot),q(\cdot)})_{\Omega_i}^+ -s_{\Omega_i}^-+3\epsilon]$ smooth atom for $F_{p(\cdot),q(\cdot)}^{s(\cdot)}(\R^n)$. 
\end{remark}

Let $\{a_{\nu,m}\}_{\nu\in\N_0, m\in\Z^n}$ is a family of $[K, L]$ smooth atoms. 
Then we need to check that $f=\sum_{\nu=0}^{\infty}\sum_{m\in\Z^n}\lambda_{\nu, m}a_{\nu, m}$ converges in $\mathcal{S}'(\R^n)$. 

\begin{proposition} 
\label{convergence} 
Let $p(\cdot),q(\cdot)\in\mathcal{P}_0(\R^n)\cap C^{\log}(\R^n)$ and $s(\cdot)\in C^{\log}(\R^n)$. 
If $\{a_{\nu,m}\}_{\nu\in\N_0, m\in\Z^n}$ is a family of smooth atoms for $A_{p(\cdot),q(\cdot)}^{s(\cdot)}(\R^n)$ and 
$\lambda=\{\lambda_{\nu,m} \}_{\nu\in\N_0, m\in\Z^n} \in a_{p(\cdot),q(\cdot)}^{s(\cdot)}$, 
then 
the sum 
\[
f=\sum_{\nu=0}^{\infty}\sum_{m\in\Z^n}\lambda_{\nu, m}a_{\nu, m}
\]
converges in $\mathcal{S}'(\R^n)$. 
\end{proposition}

\begin{proof}[Outline of the proof]
Let $\varphi\in\mathcal{S}(\R^n)$ arbitrary, $\nu_0$ be as in Remark $\ref{remark8-10}$  and natural number $k>\nu_0$. 
Then we have 
\begin{align*}
\left\langle
\sum_{\nu=0}^{k} \sum_{m\in\Z^n} \lambda_{\nu, m}a_{\nu, m}, \varphi
\right\rangle 
&= 
\left\langle
\sum_{\nu=0}^{\nu_0-1} \sum_{m\in\Z^n} \lambda_{\nu, m}a_{\nu, m}, \varphi
\right\rangle   + 
\sum_{i=0}^R 
\left\langle
\sum_{\nu=\nu_0}^{k} \sum_{ \substack{ m\in\Z^n : \\ Q_{\nu, m}\subset \Omega_i } } \lambda_{\nu, m}a_{\nu, m}, \varphi
\right\rangle, 
\end{align*}
where $\Omega_i$ and $R$ are as in Remark $\ref{remark8-10}$ 
and the summation $ \sum_{ \substack{ m\in\Z^n : \\ Q_{\nu, m}\subset \Omega_i } } $ is taken over all $m\in\Z^n$ such that 
$Q_{\nu, m}\subset \Omega_i$. 
Let fix non negative integer  $0\le i \le R$. 
We define 
\[
\lambda_{\nu,m}' = \begin{cases} \lambda_{\nu,m}, \quad &\text{if } Q_{\nu, m}\subset \Omega_i \\
                                          0 \quad &\text{otherwise}
                                          \end{cases}
\]
and 
\[
a_{\nu, m}' = \begin{cases} a_{\nu, m}, \quad &\text{if } Q_{\nu, m}\subset \Omega_i \\
                                          0 \quad &\text{otherwise}. 
                                          \end{cases}
\]
Then 
$a_{\nu,m}'$ is $[K,  ( \sigma_{p(\cdot)} )_{ \Omega_i }^+ -s_{  \Omega_i}^- +3\epsilon  ]$ atom centered at $Q_{\nu, m}$ by Remark $\ref{remark8-10}$ 
and $L_{\Omega_i}^-+\epsilon >  ( \sigma_{p(\cdot)} )_{ \Omega_i }^+ -s_{  \Omega_i}^-$ hold. 
Therefore, by using same argument of \cite[Lemma 6]{Kempka} and \cite[Theorem 3]{Drihem}, we can prove 
\[
\sum_{\nu=\nu_0}^{k}  \sum_{ \substack{ m\in\Z^n : \\ Q_{\nu, m}\subset \Omega_i } } \lambda_{\nu, m}a_{\nu, m} 
=\sum_{\nu=\nu_0}^{k}\sum_{ m\in\Z^n} \lambda_{\nu, m}' a_{\nu, m}' 
\]
converges in $\mathcal{S}'(\R^n)$ as $k\to \infty$ for $i = 0,1, \cdots, R$. 
This implies that the sum 
\[
\sum_{\nu= 0 }^{k}  \sum_{ m\in\Z^n } \lambda_{\nu, m}a_{\nu, m} 
= \sum_{\nu= 0 }^{\nu_0-1}  \sum_{ m\in\Z^n } \lambda_{\nu, m}a_{\nu, m} 
+ \sum_{i=0}^{R}\sum_{\nu=\nu_0}^{k}  \sum_{ \substack{ m\in\Z^n : \\ Q_{\nu, m}\subset \Omega_i } } \lambda_{\nu, m}a_{\nu, m} 
\]
convergence in $\mathcal{S}'(\R^n)$ as $k\to \infty$. 
\end{proof}
 
\begin{theorem}
\label{theorem 5.1.6-Besov} 
Let $p(\cdot),q(\cdot)\in\mathcal{P}_0(\R^n)\cap C^{\log}(\R^n)$ and $s(\cdot)\in C^{\log}(\R^n)$. 
Let $\{ a_{\nu,m} \}_{\nu\in\N_0, m\in\Z^n}$ is a family of smooth atoms for $B_{p(\cdot),q(\cdot)}^{s(\cdot)}(\R^n)$. 
If $\lambda=\{\lambda_{\nu,m}\}_{(\nu,m)\in\N_0\times\Z^n}\in b_{p(\cdot),q(\cdot)}^{s(\cdot)}$, 
then 
\begin{equation}
\left\|
\sum_{\nu\in\N_0}\sum_{m\in\Z^n}\lambda_{\nu,m}a_{\nu,m}
\right\|_{B_{p(\cdot),q(\cdot)}^{s(\cdot)}}  
\lesssim 
||\lambda||_{b_{p(\cdot),q(\cdot)}^{s(\cdot)}}. \label{5.24-1}
\end{equation}
\end{theorem}

\begin{remark}
\label{remark-13-12-27-1} 
Theorem $\ref{theorem 5.1.6-Besov}$ is obtained by slightly changing a part of \cite[Theorem 3]{Drihem}. 
\end{remark}

To denote the outline of the proof of Theorem \ref{theorem 5.1.6-Besov}, we need following two Lemma. 

\begin{lemma}[{\cite[Lemma 3]{Drihem}}]
\label{lemma3}
Let $0<a<1$, $0<q\le \infty$ and $\delta>0$ and let $\{\epsilon_k\}_{k\in\N_0}$ be sequence of positive real numbers, such that 
\[
\left\| \{ \epsilon_k\}_{k\in\N_0} \right\|_{\ell^q} = I  <\infty. 
\]
The sequence $\{ \delta_k\,:\,\delta_k=\sum_{j=0}^{\infty}a^{|k-j|\delta}\epsilon_j \}_{k\in\N_0}$ is in $\ell^q$ with 
\[
\left\| \{ \delta_k\}_{k\in\N_0} \right\|_{\ell^q} \le  cI,  
\]
where $c$ depends only on $a$ and $q$. 
\end{lemma}

\begin{lemma}[ { \cite[Lemma 3.3]{Fraizer1} } ]
\label{lemma-8-7-1} 
Let $\{  \varphi_j\}$, $j\in\N_0$ be a resolution of unity and let $a_{\nu, m}$ be an $[K, L]$-atom. Then 
\[
\left| \inversefourier\varphi_j\ast a_{\nu, m}(x) \right| \lesssim  \begin{cases} 2^{(\nu-j)K} \left(1+2^{\nu}\left| x-2^{-\nu}m\right|\right)^{-M} \ \ &\text{if }\nu \le j \\ 
                                                                     2^{(j-\nu)(L+n+1)} \left(1+2^{j}\left| x-2^{-\nu}m\right|\right)^{-M} \ \ &\text{if } j \le \nu, 
\end{cases} 
\]
where $M$ is sufficiently large. 
\end{lemma}

\begin{proof}[Outline of the proof of Theorem \ref{theorem 5.1.6-Besov} ]
Let $f= \sum_{\nu=0}^{\infty}\sum_{m\in\Z^n}\lambda_{\nu, m}a_{\nu, m}$. Without loss of generally, we assume that $\| \lambda \|_{b_{p(\cdot),q(\cdot)}^{s(\cdot)}}=1$. 
By using the similar argument of the proof of Theorem 3 of \cite{Drihem}, it suffices to show that 
\[
\sum_{j=0}^{\infty}\left\| \left|c2^{j s(\cdot)} \inversefourier\varphi_j\ast f \right|^{q(\cdot)} \right\|_{L^{\frac{p(\cdot)}{q(\cdot)}}} \le C \quad \text{whenever} \quad 
\sum_{j=0}^{\infty}\left\| \left| 2^{j s(\cdot)} \sum_{m\in\Z^n}\lambda_{\nu, m}a_{\nu, m} \right|^{q(\cdot)} \right\|_{L^{\frac{p(\cdot)}{q(\cdot)}}} =1, 
\] 
where $\{ \varphi_j\}_{j\in\N_0}$ is the resolution of unity as in Definition $\ref{Def:T}$. 
Let $0<r<\max(1/{q^+}, {p^-}/{q^+})$ and $\nu_0$ as in Remark $\ref{remark8-10}$. 
Then we have 
\begin{align*}
&\sum_{j=0}^{\infty}\left\| \left|c2^{j s(\cdot)} \inversefourier\varphi_j\ast f \right|^{q(\cdot)} \right\|_{L^{\frac{p(\cdot)}{q(\cdot)}}} \notag \\ 
&\le \sum_{j=0}^{\infty} \left( \sum_{\nu=0}^{\nu_0-1}  
\left\| \left| c\sum_{m\in\Z^n} 2^{js(\cdot)} (\inversefourier\varphi_j\ast \lambda_{\nu, m} a_{\nu, m})(\cdot) \right|^{rq(\cdot)}  \right\|_{L^{\frac{p(\cdot)}{rq(\cdot)}}} \right)^{\frac{1}{r}} \\ 
&+   \sum_{i=0}^R \sum_{j=0}^{\infty} \left( \sum_{\nu=\nu_0}^{\infty}  
\left\| \left| c\sum _{\substack{m\in \Z^n : \\ Q_{\nu, m} \subset \Omega_i }}2^{js(\cdot)} (\inversefourier\varphi_j\ast \lambda_{\nu, m} a_{\nu, m})(\cdot) \right|^{rq(\cdot)}  \right\|_{L^{\frac{p(\cdot)}{rq(\cdot)}}(\Omega_i)} \right)^{\frac{1}{r}} \\ 
&\le I + \sum_{i=0}^R I_i,  
\end{align*}
where 
\[
I_i = \sum_{j=0}^{\infty} \left( \sum_{\nu=\nu_0}^{\infty}  
\left\| \left| c\sum _{\substack{m\in \Z^n : \\ Q_{\nu, m} \subset \Omega_i }}2^{js(\cdot)} (\inversefourier\varphi_j\ast \lambda_{\nu, m} a_{\nu, m})(\cdot) \right|^{rq(\cdot)}  \right\|_{L^{\frac{p(\cdot)}{rq(\cdot)}}(\Omega_i)} \right)^{\frac{1}{r}}. 
\]

Firstly, we denote the outline of the proof of $I_i\lesssim  1$ for any $i=0,1,\cdots, R$. 
Let fix non negative integer  $0\le i \le R$. 
We define 
\[
\lambda_{\nu,m}' = \begin{cases} \lambda_{\nu,m}, \quad &\text{if } Q_{\nu, m}\subset \Omega_i \\
                                          0 \quad &\text{otherwise}
                                          \end{cases}
\]
and 
\[
a_{\nu, m}' = \begin{cases} a_{\nu, m}, \quad &\text{if } Q_{\nu, m}\subset \Omega_i \\
                                          0 \quad &\text{otherwise}. 
                                          \end{cases}
\]
Then we have 
\[
I_i= \sum_{j=0}^{\infty} \left( \sum_{\nu=\nu_0}^{\infty}  
\left\| \left| c \sum _{m\in \Z^n } 2^{js(\cdot)} (\inversefourier\varphi_j\ast \lambda_{\nu, m}' a_{\nu, m}')(\cdot) \right|^{rq(\cdot)}  \right\|_{L^{\frac{p(\cdot)}{q(\cdot)}}(\Omega_i)} \right)^{\frac{1}{r}} 
\]
and $a_{\nu,m}'$ is $[K,  ( \sigma_{p(\cdot)} )_{ \Omega_i }^+ -s_{  \Omega_i}^- +3\epsilon  ]$ atom centered at $Q_{\nu, m}$ by Remark $\ref{remark8-10}$.  
By 
using same argument of the proof of \cite[Theorem 3]{Drihem} with replacing $L^{\frac{p(\cdot)}{q(\cdot)}}$ by  
$L^{\frac{p(\cdot)}{q(\cdot)}}(\Omega_i)$, we can prove $I_i\lesssim 1$ for any $i=0,1,\cdots, R$.

Finally, we denote the outline of the proof of $I\lesssim 1$. 
For any $\nu\in\N_0$ and $m\in\Z^n$, it is easy to see that 
$ \displaystyle 
\lfloor L_{Q_{\nu, m}}^- \rfloor \ge \lfloor L_{\R^n}^- \rfloor. 
$
This implies that we can choose $L$ as in Lemma $\ref{lemma-8-7-1}$ such that $ \lfloor L_{\R^n}^- \rfloor $. 
By Lemma $\ref{lemma-8-7-1}$, 
we obtain 
\begin{equation}
|2^{js(x)}\inversefourier\varphi_j\ast a_{\nu,m}(x)| \le \begin{cases} c2^{-|j - \nu| (\lfloor L_{\R^n}^- \rfloor +1+n)} 2^{js(x)} \left(1+2^{j}\left| x-2^{-\nu}m\right|\right)^{-M}  \quad \text{if } j\le \nu \label{13-7-24-1} \\ 
                                                  c2^{-|j - \nu| (K-s^+)} 2^{\nu s(x)} \left(1+2^{\nu}\left| x-2^{-\nu}m\right|\right)^{-M} \quad \text{if } j\ge \nu. 
\end{cases}
\end{equation}
Therefore, we have 
\begin{align}
&\left\| \left| \sum_{m\in\Z^n} 2^{js(\cdot)}\inversefourier\varphi_j\ast \lambda_{\nu, m} a_{\nu,m}(\cdot)\right|^{rq(\cdot)} \right\|_{L^{\frac{p(\cdot)}{rq(\cdot)} }}  \notag \\ 
&\quad \le  
\left\| 
\left| 
\sum_{m\in\Z^n}
c2^{-|j - \nu| (\lfloor L_{\R^n}^- \rfloor +1+n) } 2^{js(\cdot)} \lambda_{\nu, m}\langle 2^{j}\cdot-2^{j-\nu}m\rangle^{-M}
\right|^{rq(\cdot)} 
\right\|_{L^{\frac{p(\cdot)}{rq(\cdot)} }}  
\end{align} 
for $j \le \nu$ and 
\begin{align}
&\left\| \left| \sum_{m\in\Z^n} 2^{js(\cdot)}\inversefourier\varphi_j\ast \lambda_{\nu, m} a_{\nu,m}(\cdot)\right|^{rq(\cdot)} \right\|_{L^{\frac{p(\cdot)}{rq(\cdot)} }}  \notag \\ 
&\le  
\left\| 
\left| 
\sum_{m\in\Z^n}
c2^{(\nu-j) (K-s^+ )} 2^{\nu s(\cdot)} \lambda_{\nu, m}\langle 2^{\nu}\cdot-m\rangle^{-M}
\right|^{rq(\cdot)} 
\right\|_{L^{\frac{p(\cdot)}{rq(\cdot)} }}
\end{align} 
for $j\ge \nu$. 
Let $0<t<\min(1, p^-)$. 
If there exists $c>0$ such that 
\begin{align}
\left\| 
 \left| c
\sum_{m\in\Z^n} 2^{\nu s(\cdot)} \lambda_{\nu, m}\langle 2^{\nu}\cdot-m\rangle^{-M}
\right|^{rq(\cdot)} 
\right\|_{L^{\frac{p(\cdot)}{rq(\cdot)} }}  \le 
\left\| 
\left| 
\sum_{m\in\Z^n}
 2^{\nu s(\cdot)} \lambda_{\nu, m}
\chi_{\nu, m}
\right|^{rq(\cdot)} 
\right\|_{L^{\frac{p(\cdot)}{rq(\cdot)} }} +2^{-\nu}   \label{13-7-24-2}
\end{align}
for $j\ge\nu$ and 
\begin{align}
&\left\| 
 \left| c2^{(j-\nu)(n/t-s^-)}
\sum_{m\in\Z^n} 2^{j s(\cdot)} \lambda_{\nu, m}\langle 2^{\nu}\cdot-m\rangle^{-M}
\right|^{rq(\cdot)} 
\right\|_{L^{\frac{p(\cdot)}{rq(\cdot)} }}  \notag \\ 
&\le
\left\| 
\left| 
\sum_{m\in\Z^n}
 2^{\nu s(\cdot)} \lambda_{\nu, m}
\chi_{\nu, m}
\right|^{rq(\cdot)} 
\right\|_{L^{\frac{p(\cdot)}{rq(\cdot)} }} +2^{-j}   \label{13-8-8-1}
\end{align}
for $j\le \nu$, then 
\begin{align}
&\left\| \left| \sum_{m\in\Z^n} 2^{js(\cdot)}\inversefourier\varphi_j\ast \lambda_{\nu, m} a_{\nu,m}(\cdot)\right|^{rq(\cdot)} \right\|_{L^{\frac{p(\cdot)}{rq(\cdot)} }}  \notag \\ 
&\quad \le  
2^{(\nu-j)(K-s^+)rq^-}
\left(
\left\| 
\left| 
\sum_{m\in\Z^n}
 2^{\nu s(\cdot)} \lambda_{\nu, m}
\chi_{\nu, m}
\right|^{rq(\cdot)} 
\right\|_{L^{\frac{p(\cdot)}{rq(\cdot)} }} +2^{-\nu} \label{13-8-10-5}
\right)
\end{align} 
for $j\ge \nu$ and 
\begin{align} 
&\left\| 
 \left| c 2^{-|j - \nu| (\lfloor L_{\R^n}^- \rfloor +1+n)  } 
\sum_{m\in\Z^n} 2^{j s(\cdot)} \lambda_{\nu, m}\langle 2^{\nu}\cdot-m\rangle^{-M}
\right|^{rq(\cdot)} 
\right\|_{L^{\frac{p(\cdot)}{rq(\cdot)} }} \notag \\ 
&\le 2^{(j-\nu)(\lfloor L^-\rfloor+n+1  -n/t+s^-)rq^-} \left(
\left\| 
\left| 
\sum_{m\in\Z^n}
 2^{\nu s(\cdot)} \lambda_{\nu, m}
\chi_{\nu, m}
\right|^{rq(\cdot)} 
\right\|_{L^{\frac{p(\cdot)}{rq(\cdot)} }} +2^{-j} \right) \notag \\ 
&\le C \left(
\left\| 
\left| 
\sum_{m\in\Z^n}
 2^{\nu s(\cdot)} \lambda_{\nu, m}
\chi_{\nu, m}
\right|^{rq(\cdot)} 
\right\|_{L^{\frac{p(\cdot)}{rq(\cdot)} }}+2^{-j} \right) 
 \label{13-8-8-2}
\end{align}
for $j \le \nu$, where $C= \max(1, 2^{(1-\nu_0)(\lfloor L^-\rfloor+n+1  -n/t+s^-)q^-   })$.  
It is easy to see that 
\begin{align}
c\sum_{j=0}^{\infty} 
&\left\|
 \left|
\sum_{\nu=0}^{\nu_0-1} \sum_{m\in\Z^n} 2^{js(\cdot)} (\inversefourier\varphi_j\ast \lambda_{\nu, m} a_{\nu, m})(\cdot)  
\right|^{rq(\cdot)}
\right\|_{L^{\frac{p(\cdot)}{rq(\cdot)}}}^{\frac{1}{r} } \notag \\ 
&\le c\sum_{j=0}^{\infty} 
\left( \sum_{\nu=0}^{\nu_0-1}
\left\|
 \left|
\sum_{m\in\Z^n} 2^{js(\cdot)} (\inversefourier\varphi_j\ast \lambda_{\nu, m} a_{\nu, m})(\cdot)  
\right|^{rq(\cdot)}
\right\|_{L^{\frac{p(\cdot)}{rq(\cdot)}}}
\right)^{\frac{1}{r} } \notag \\ 
&\lesssim c\sum_{j=0}^{\infty} 
 \sum_{\nu=0}^{\nu_0-1}
\left\|
 \left|
\sum_{m\in\Z^n} 2^{js(\cdot)} (\inversefourier\varphi_j\ast \lambda_{\nu, m} a_{\nu, m})(\cdot)  
\right|^{rq(\cdot)}
\right\|_{L^{\frac{p(\cdot)}{rq(\cdot)}}}^{\frac{1}{r} }. \label{13-8-10-11}
\end{align}
Let $\nu= w$. 
We estimate the right hand side of $(\ref{13-8-10-11})$. 
We have 
\begin{align*}
\sum_{j=0}^{\infty} 
&\left\|
 \left|
\sum_{m\in\Z^n} 2^{js(\cdot)} (\inversefourier\varphi_j\ast \lambda_{w, m} a_{w, m})(\cdot)  
\right|^{rq(\cdot)}
\right\|_{L^{\frac{p(\cdot)}{rq(\cdot)}}}^{\frac{1}{r} }  \notag \\ 
&= 
\sum_{j=0}^{w} 
\left\|
 \left|
\sum_{m\in\Z^n} 2^{js(\cdot)} (\inversefourier\varphi_j\ast \lambda_{w, m} a_{w, m})(\cdot)  
\right|^{rq(\cdot)}
\right\|_{L^{\frac{p(\cdot)}{rq(\cdot)}}}^{\frac{1}{r} }  \notag \\ 
&\qquad + \sum_{j=w+1}^{\infty} 
\left\|
 \left|
\sum_{m\in\Z^n} 2^{js(\cdot)} (\inversefourier\varphi_j\ast \lambda_{w, m} a_{w, m})(\cdot)  
\right|^{rq(\cdot)}
\right\|_{L^{\frac{p(\cdot)}{rq(\cdot)}}}^{\frac{1}{r} }. 
\end{align*}
Then, by $(\ref{13-8-10-5})$ and $(\ref{13-8-8-2})$, we obtain 
\begin{align*}
\sum_{j=0}^{\infty} 
&\left\|
 \left|
\sum_{m\in\Z^n} 2^{js(\cdot)} (\inversefourier\varphi_j\ast \lambda_{w, m} a_{w, m})(\cdot)  
\right|^{rq(\cdot)}
\right\|_{L^{\frac{p(\cdot)}{rq(\cdot)}}}^{\frac{1}{r} }  \notag \\ 
&\le \sum_{j=0}^{w} 
C^{1/r} \left(
\left\| 
\left| 
\sum_{m\in\Z^n}
 2^{w s(\cdot)} \lambda_{w, m}
\chi_{w, m}
\right|^{rq(\cdot)} 
\right\|_{L^{\frac{p(\cdot)}{rq(\cdot)} }}^{r\cdot 1/r} +2^{-j} 
\right)^{\frac{1}{r} } \notag \\ 
&+ \sum_{j=w+1}^{\infty} 
2^{(w-j)(K-s^+)q^-}
\left(
\left\| 
\left| 
\sum_{m\in\Z^n}
 2^{w s(\cdot)} \lambda_{w, m}
\chi_{w, m}
\right|^{rq(\cdot)} 
\right\|_{L^{\frac{p(\cdot)}{rq(\cdot)} }} +2^{-w}   \right)^{\frac{1}{r} }. 
\end{align*}
By 
\[
\left\| 
\left| 
\sum_{m\in\Z^n}
 2^{w s(\cdot)} \lambda_{w, m}
\chi_{w, m}
\right|^{rq(\cdot)} 
\right\|_{L^{\frac{p(\cdot)}{rq(\cdot)} }}^{1/r} 
= 
\left\| 
\left| 
\sum_{m\in\Z^n}
 2^{w s(\cdot)} \lambda_{w, m}
\chi_{w, m}
\right|^{q(\cdot)} 
\right\|_{L^{\frac{p(\cdot)}{q(\cdot)} }} \le 1, 
\]
we see that 
\begin{align*}
&\sum_{j=0}^{\infty} 
\left\|
 \left|
\sum_{m\in\Z^n} 2^{js(\cdot)} (\inversefourier\varphi_j\ast \lambda_{w, m} a_{w, m})(\cdot)  
\right|^{rq(\cdot)}
\right\|_{L^{\frac{p(\cdot)}{rq(\cdot)}}}^{\frac{1}{r} }  \notag \\ 
&\le  
(w+1)C^{1/r}2^{1/r} + \sum_{j=w+1}^{\infty} 
2^{|w-j|(s^+-K)q^-}
\left(
\left\| 
\left| 
\sum_{m\in\Z^n}
 2^{w s(\cdot)} \lambda_{w, m}
\chi_{w, m}
\right|^{rq(\cdot)} 
\right\|_{L^{\frac{p(\cdot)}{rq(\cdot)} }} +2^{-w}   \right)^{\frac{1}{r} } \notag \\ 
&\le  
(w+1)C^{\frac{1}{r}}2^{\frac{1}{r}} + \sum_{j=w+1}^{\infty} 
\left( \sum_{\nu=0}^{w}
2^{|\nu-j|(s^+-K)rq^-}
\left(
\left\| 
\left| 
\sum_{m\in\Z^n}
 2^{\nu s(\cdot)} \lambda_{\nu, m}
\chi_{\nu, m}
\right|^{rq(\cdot)} 
\right\|_{L^{\frac{p(\cdot)}{rq(\cdot)} }} +2^{-\nu}  \right) \right)^{\frac{1}{r} }  \notag \\ 
&\le 
(w+1)C^{\frac{1}{r}}2^{\frac{1}{r}} + \sum_{j=w+1}^{\infty} 
\left( \sum_{\nu=0}^{w}
2^{-|\nu-j|(K-s^+)rq^-}
\left(
\left\| 
\left| 
\sum_{m\in\Z^n}
 2^{\nu s(\cdot)} \lambda_{\nu, m}
\chi_{\nu, m}
\right|^{rq(\cdot)} 
\right\|_{L^{\frac{p(\cdot)}{rq(\cdot)} }} +2^{-\nu}  \right) \right)^{\frac{1}{r} } \notag 
\end{align*}
Then, by Lemma $\ref{lemma3}$, we obtain 
\begin{align*}
&\sum_{j=0}^{\infty} 
\left\|
 \left|
\sum_{m\in\Z^n} 2^{js(\cdot)} (\inversefourier\varphi_j\ast \lambda_{w, m} a_{w, m})(\cdot)  
\right|^{rq(\cdot)}
\right\|_{L^{\frac{p(\cdot)}{rq(\cdot)}}}^{\frac{1}{r} }  \notag \\ 
&\le  
(w+1)C^{\frac{1}{r}}2^{\frac{1}{r}} + \sum_{j=w+1}^{\infty} 
\left( \sum_{\nu=0}^{w}
2^{-|\nu-j|(K-s^+)rq^-}
\left(
\left\| 
\left| 
\sum_{m\in\Z^n}
 2^{\nu s(\cdot)} \lambda_{\nu, m}
\chi_{\nu, m}
\right|^{rq(\cdot)} 
\right\|_{L^{\frac{p(\cdot)}{rq(\cdot)} }} +2^{-\nu}  \right) \right)^{\frac{1}{r} } \notag \\ 
&\le  
(w+1)C^{1/r}2^{1/r} + \sum_{j=0}^{\infty} 
\left(
\left\| 
\left| 
\sum_{m\in\Z^n}
 2^{j s(\cdot)} \lambda_{j, m}
\chi_{j, m}
\right|^{rq(\cdot)} 
\right\|_{L^{\frac{p(\cdot)}{rq(\cdot)} }} +2^{-j}  \right)^{\frac{1}{r} }  \notag \\ 
&\le (w+1)C^{1/r}2^{1/r} + D<\infty. 
\end{align*}

Therefore, we have 
\begin{align*}
c\sum_{j=0}^{\infty} 
&\left\|
 \left|
\sum_{\nu=0}^{\nu_0-1} \sum_{m\in\Z^n} 2^{js(\cdot)} (\inversefourier\varphi_j\ast \lambda_{\nu, m} a_{\nu, m})(\cdot)  
\right|^{rq(\cdot)}
\right\|_{L^{\frac{p(\cdot)}{rq(\cdot)}}}^{\frac{1}{r} } \notag \\ 
&\lesssim \left( \frac{\nu_0(\nu_0-1)}{2} +1\right) C^{1/r}2^{1/r} + \nu_0 D <\infty
\end{align*}
by $(\ref{13-8-10-11})$. 

Therefore, we consider that $(\ref{13-7-24-2})$ and $(\ref{13-8-8-1})$. 
We can use similar argument of the proof of \cite[Theorem 3]{Drihem}, 
we have $(\ref{13-7-24-2})$ and $(\ref{13-8-8-1})$ because $\nu < \nu_0-1$.

\end{proof}

By using similar arguments of the outline of the proof of Theorem $\ref{theorem 5.1.6-Besov}$ and 
similar arguments about atomic decompositions \cite{Kempka}, 
we have the case of $F_{p(\cdot),q(\cdot)}^{s(\cdot)}$. 

\begin{theorem}
\label{theorem 5.1.6-Triebel} 
Let $p(\cdot),q(\cdot)\in\mathcal{P}_0(\R^n)\cap C^{\log}(\R^n)$ and $s(\cdot)\in C^{\log}(\R^n)$. 
Let $\{ a_{\nu,m} \}_{\nu\in\N_0, m\in\Z^n}$ is a family of smooth atoms for $F_{p(\cdot),q(\cdot)}^{s(\cdot)}(\R^n)$. 
If $\lambda=\{\lambda_{\nu, m}\}_{(\nu,m)\in\N_0\times\Z^n}\in f_{p(\cdot),q(\cdot)}^{s(\cdot)}$, 
then 
\begin{equation}
\left\|
\sum_{\nu\in\N_0}\sum_{m\in\Z^n}\lambda_{\nu,m}a_{\nu,m}
\right\|_{F_{p(\cdot),q(\cdot)}^{s(\cdot)}}  
\lesssim 
||\lambda||_{f_{p(\cdot),q(\cdot)}^{s(\cdot)}}. \label{5.24-2}
\end{equation}
\end{theorem}

To denote the outline of the proof of Theorem \ref{theorem 5.1.6-Triebel}, 
we need the following Theorem and Lemmas. 

\begin{theorem}[{\cite[Corollary 5.6]{Kempka}}]
\label{theorem 3/18-1}
Let $p(\cdot),q(\cdot)\in\mathcal{P}_0(\R^n)\cap C^{\log}(\R^n)$ and $s(\cdot)\in C^{\log}(\R^n)$. 
Furthermore let $\{ a_{\nu,m} \}_{\nu\in\N_0, m\in\Z^n}$ are  $[K,L]$ atoms for $F_{p(\cdot),q(\cdot)}^{s(\cdot)}(\R^n)$, 
where $K>s^+$ and $L=\lfloor \sigma_{p^-, q^-}-s^-\rfloor$. 
If $\lambda=\{\lambda_{\nu, m}\}_{(\nu,m)\in\N_0\times\Z^n}\in f_{p(\cdot),q(\cdot)}^{s(\cdot)}$, 
then 
\begin{equation}
\left\|
\sum_{\nu\in\N_0}\sum_{m\in\Z^n}\lambda_{\nu,m}a_{\nu,m}
\right\|_{F_{p(\cdot),q(\cdot)}^{s(\cdot)}}  
\lesssim 
||\lambda||_{f_{p(\cdot),q(\cdot)}^{s(\cdot)}}. \label{3/18-1}
\end{equation}
\end{theorem}

\begin{lemma}
\label{New-lemma}
Let $0<t<1$, $j,\nu\in\N_0$ and $\displaystyle \{ \lambda_{\nu,m}\}_{\nu\in\N_0, m\in\Z^n}$ be positive. 
Furthermore let $M>0$ be sufficiency large. 

\noindent
$(i)$ Then 
\begin{align*}
&\sum_{m\in\Z^n}2^{\nu s(x)}\lambda_{\nu,m}(1+2^j|x-2^{-\nu}m|)^{-M}  \notag \\ 
&\le c\max(1, 2^{(\nu-j)n/t}) 
\mathcal{M}_t\left( \sum_{m\in\Z^n}2^{\nu  s(\cdot)}\lambda_{\nu,m}\chi_{\nu,m}(\cdot) \right)  (x) 
\end{align*}
holds for any $x\in\R^n$. 

\noindent
$(ii)$ Then 
\begin{align*}
&\sum_{m\in\Z^n}2^{\nu s(x)}\lambda_{\nu,m}(1+2^j|x-2^{-\nu}m|)^{-M}  \notag \\ 
&\le c2^{h_n\alpha}\max(1, 2^{(\nu-j)n/t}) \left( \left[ \eta_{\nu,\alpha t} \ast \left( \sum_{m\in\Z^n}2^{\nu t s(\cdot)}\lambda_{\nu,m}\chi_{\nu,m}^t(\cdot) \right) \right](x) \right)^{1/t} 
\end{align*}
holds for any $x\in\R^n$ and for any positive real number $\alpha>0$, 
where $h_n$ is a positive number depend only on $n$. 
\end{lemma}

\begin{proof}
$(i)$ is proved in \cite{Drihem}. Hence we prove only $(ii)$. 

We use the argument similar to \cite{Drihem}. 
Let $k\in\N_0$. 
We define 
\[
\Omega_k=\left\{ m\in\Z^n\,:\, 2^{k-1}\le 2^{\min(\nu,j)}|x-2^{-\nu}m| \le 2^k\right\} 
\]
and 
\[
\Omega_0=\left\{ m\in\Z^n\,:\,  2^{\min(\nu,j)}|x-2^{-\nu}m| \le 1 \right\}. 
\] 

Firstly we consider the case of $\nu \le j$. 

Let $M=R+T$ and $T>\frac{n}{t}$. 
Then we obtain 
\begin{align*}
\sum_{m\in\Z^n}2^{\nu s(x)}\lambda_{\nu,m}(1+2^j|x-2^{-\nu}m|)^{-M} 
&\le 
\sum_{k=0}^{\infty}\sum_{m\in\Omega_k} 2^{\nu s(x)}\lambda_{\nu,m}(1+2^j|x-2^{-\nu}m|)^{-M}  \notag \\ 
&\le 
c \sum_{k=0}^{\infty}\sum_{m\in\Omega_k} 2^{\nu s(x)}\lambda_{\nu,m}2^{-Mk} \notag \\ 
&\le 
c \sum_{k=0}^{\infty} 2^{-(T-\frac{n}{t})k}\sum_{m\in\Omega_k} 2^{\nu s(x)}\lambda_{\nu,m}2^{-(R+\frac{n}{t})k}  \notag \\ 
&\le 
c \left( \sup_{k\in\N_0} 2^{-(Rt+n)k} \sum_{m\in\Omega_k} 2^{\nu t s(x)}\lambda_{\nu,m}^t\right)^{\frac{1}{t}}. 
\end{align*} 
Therefore we have 
\begin{align}
&\sum_{m\in\Z^n}2^{\nu s(x)}\lambda_{\nu,m}(1+2^j|x-2^{-\nu}m|)^{-M}  \notag \\ 
&\le 
c \left( \sup_{k\in\N_0} 2^{-(Rt+n)k} \sum_{m\in\Omega_k} 2^{\nu t s(x)}\lambda_{\nu,m}^t \right)^{\frac{1}{t}}  \notag \\ 
&\le 
c \left( \sup_{k\in\N_0} 2^{-Rtk+(\nu-k)n} \int_{\cup_{m\in\Omega_k}Q_{\nu,m}}\sum_{m\in\Omega_k} 2^{\nu t s(x)}\lambda_{\nu,m}^t \chi_{\nu,m}(y) {\rm d}y \right)^{\frac{1}{t}}, \label{3.18-1}
\end{align} 
where we use $\displaystyle |\cup_{m\in\Omega_k}Q_{\nu,m}| \sim  2^{(k-\nu)n}$. 
Using same argument in the proof of \cite[Theorem 3]{Drihem}, 
There exists a $h_n\in\N_0$ such that 
$\displaystyle |x-y|\le 2^{k-\nu+h_n}$ for $\displaystyle y\in\cup_{m\in\Omega_k}Q_{\nu,m}$. 
This implies that $y$ is located in some ball $B(x,2^{k-\nu+h_n})$ and that 
\[
1\le c \frac{2^{(k+h_n)\alpha t}}{(1+2^{\nu}|x-y|)^{\alpha t}}
\]
holds for any $\alpha>0$. 
Hence we see that 
\begin{align*}
&\sum_{m\in\Z^n}2^{\nu s(x)}\lambda_{\nu,m}(1+2^j|x-2^{-\nu}m|)^{-M} \notag \\ 
&\le c \left( \sup_{k\in\N_0} 2^{-Rtk+(\nu-k)n} \int_{\cup_{m\in\Omega_k}Q_{\nu,m}}\sum_{m\in\Omega_k} 2^{\nu t s(x)}\lambda_{\nu,m}^t \chi_{\nu,m}(y){\rm d}y \right)^{\frac{1}{t}} \notag \\ 
&\le c \left( \sup_{k\in\N_0} 2^{-Rtk+(k+h_n)\alpha t-kn} \int_{ B(x,2^{k-\nu+h_n}) }\sum_{m\in\Omega_k} \frac{2^{\nu n} 2^{\nu t s(x)}\lambda_{\nu,m}^t\chi_{\nu,m}(y)}{ (1+2^{\nu}|x-y|)^{\alpha t}  }{\rm d}y \right)^{\frac{1}{t}}. 
\end{align*}
Since $s(\cdot)\in C_{\log}(\R^n)$, we can prove that 
\[
2^{\nu s(x)} \le c2^{\beta k} 2^{\nu s(y)}, 
\] 
where $\displaystyle \beta = \max(c_{\log}(s), s^+-s^-)$. 
\begin{align*}
&\sum_{m\in\Z^n}2^{\nu s(x)}\lambda_{\nu,m}(1+2^j|x-2^{-\nu}m|)^{-M} \notag \\ 
&\le c \left( \sup_{k\in\N_0} 2^{-Rtk+(k+h_n)\alpha t-kn} \int_{ B(x,2^{k-\nu+h_n}) }\sum_{m\in\Omega_k} \frac{2^{\nu n}2^{\nu t s(x)}\lambda_{\nu,m}^t\chi_{\nu,m}(y)}{ (1+2^{\nu}|x-y|)^{\alpha t}  }{\rm d}y \right)^{\frac{1}{t}} \notag \\ 
&\le c \left( \sup_{k\in\N_0} 2^{-(R+\alpha-\beta)kt}2^{h_n\alpha t-kn} 
\left[
\eta_{\nu,\alpha t}\ast\left( \sum_{m\in\Z^n} 2^{\nu t s(y)}\lambda_{\nu,m}^t\chi_{\nu,m}(y)\right)
\right](x)
 \right)^{\frac{1}{t}}.   
\end{align*}
Since $R$ is sufficiency large such that 
\[
R> -\alpha+\max(c_{\log}(s), s^+-s^-), 
\]
we get 
\begin{align*}
\sum_{m\in\Z^n}2^{\nu s(x)}\lambda_{\nu,m}(1+2^j|x-2^{-\nu}m|)^{-M} 
&\le c 2^{h_n\alpha}
\left[
\eta_{\nu,\alpha t}\ast\left( \sum_{m\in\Z^n} 2^{\nu t s(y)}\lambda_{\nu,m}^t\chi_{\nu,m}(y)\right)
\right]^{\frac{1}{t}}(x).   
\end{align*}

Finally we consider the case of $j\le \nu$. 
By using same argument as above, we have 
\begin{align*}
&\sum_{m\in\Z^n}2^{\nu s(x)}\lambda_{\nu,m}(1+2^j|x-2^{-\nu}m|)^{-M}  \notag \\ 
&\le 
c \left( \sup_{k\in\N_0} 2^{-(Rt+n)k} \sum_{m\in\Omega_k} 2^{\nu t s(x)}\lambda_{\nu,m}^t \right)^{\frac{1}{t}}. 
\end{align*} 
In the case of $j\le \nu$, by $\displaystyle |\cup_{m\in\Omega_k}Q_{\nu,m}| \sim  2^{(k-j)n}$, we obtain 
\begin{align*}
&\sum_{m\in\Z^n}2^{\nu s(x)}\lambda_{\nu,m}(1+2^j|x-2^{-\nu}m|)^{-M}  \notag \\ 
&\le 
c \left( \sup_{k\in\N_0} 2^{-(Rt+n)k} \sum_{m\in\Omega_k} 2^{\nu t s(x)}\lambda_{\nu,m}^t \right)^{\frac{1}{t}}  \notag \\ 
&\le 
c \left( \sup_{k\in\N_0} 2^{-Rtk+(j-k)n} \int_{\cup_{m\in\Omega_k}Q_{\nu,m}}\sum_{m\in\Omega_k} 2^{\nu t s(x)}\lambda_{\nu,m}^t \chi_{\nu,m}(y) {\rm d}y \right)^{\frac{1}{t}}. 
\end{align*} 
This implies that 
\begin{align*}
&\sum_{m\in\Z^n}2^{\nu s(x)}\lambda_{\nu,m}(1+2^j|x-2^{-\nu}m|)^{-M}  \notag \\ 
&\le 
c \left( \sup_{k\in\N_0} 2^{-Rtk+(j-k)n} \int_{\cup_{m\in\Omega_k}Q_{\nu,m}}\sum_{m\in\Omega_k} 2^{\nu t s(x)}\lambda_{\nu,m}^t \chi_{\nu,m}(y) {\rm d}y \right)^{\frac{1}{t}} \notag \\ 
&\le 
c 2^{(j-\nu)n/t} \left( \sup_{k\in\N_0} 2^{-Rtk+(\nu-k)n} \int_{\cup_{m\in\Omega_k}Q_{\nu,m}}\sum_{m\in\Omega_k} 2^{\nu t s(x)}\lambda_{\nu,m}^t \chi_{\nu,m}(y) {\rm d}y \right)^{\frac{1}{t}}. 
\end{align*} 
Therefore, it is easy to see that 
\begin{align*}
\sum_{m\in\Z^n}2^{\nu s(x)}\lambda_{\nu,m}(1+2^j|x-2^{-\nu}m|)^{-M} 
\le c 2^{h_n\alpha}2^{(\nu-j)n/t} \left( \left[ \eta_{\nu,\alpha t} \ast \left( \sum_{m\in\Z^n}2^{\nu t s(\cdot)}\lambda_{\nu,m}\chi_{\nu,m}^t(\cdot) \right) \right](x) \right)^{1/t} 
\end{align*}
by same argument as above.

\end{proof}

\begin{lemma}[{ \cite[Lemma 4.2]{Kempka3}} ]
\label{Lemma4.2-Kempka} 
Let $p(\cdot),q(\cdot)\in\mathcal{P}(\R^n)$ with $0< q^-\le q^+<\infty$ and $0< q^-\le q^+<\infty$. 
For any sequences $\{ g_j\}_{j=0}^{\infty}$ of nonnegative measurable functions on $\R^n$ and $\delta>0$ let 
\[
G_j(x)=\sum_{k=0}^{\infty}2^{-|k-j|\delta}g_k(x)
\]
for all $x\in\R^n$ and $j\in\N_0$. 
Then with constant $c=c(p,q,\delta)$ we have 
\[
\left\| \{ G_j\}_{j\in\N_0} \right\|_{L^{p(\cdot)}(\ell^{q(\cdot)})} 
\le c \left\| \{ g_j\}_{j\in\N_0} \right\|_{L^{p(\cdot)}(\ell^{q(\cdot)})}. 
\]
\end{lemma}

Now we prove Theorem $\ref{theorem 3/18-1}$. 

As well as we mentioned in Remark $\ref{remark-13-12-27-1}$,  
Theorem $\ref{theorem 5.1.6-Triebel}$ is obtained by slightly changing a part of \cite[Corollary 5.6]{Kempka} with the $F_{p(\cdot),q(\cdot)}^{s(\cdot)}(\R^n)$ case. 

\begin{proof}[Outline of the proof of Theorem \ref{theorem 5.1.6-Triebel} ]
Let $f= \sum_{\nu=0}^{\infty}\sum_{m\in\Z^n}\lambda_{\nu, m}a_{\nu, m}$. Without loss of generally, we assume that $\| \lambda \|_{f_{p(\cdot),q(\cdot)}^{s(\cdot)}}=1$. 
Then we describe the outline of the proof of 
\[
\left\| \left\{ \left|2^{j s(\cdot)} \inversefourier\varphi_j\ast f \right|^{r} \right\}_{j=0}^{\infty} \right\|_{L^{p(\cdot)/r}(\ell^{q(\cdot)/r}) } ^{1/r} \le C,  
\] 
where $\{ \varphi_j\}_{j\in\N_0}$ is the resolution of unity as in Definition $\ref{Def:T}$. 
Let $0<r<\min(1, p^-)$. 
By using similar arguments of the outline of the proof of Theorem $\ref{theorem 5.1.6-Besov}$,  
we consider the following inequality
\begin{align*}
&\left\| \left\{ \left|2^{j s(\cdot)} \inversefourier\varphi_j\ast f \right|^{r} \right\}_{j=0}^{\infty} \right\|_{L^{p(\cdot)/r}(\ell^{q(\cdot)/r}) } ^{1/r} \notag \\ 
&\le \left\| \left\{ \left|2^{j s(\cdot)} \inversefourier\varphi_j\ast \sum_{\nu=0}^{\nu_0-1} \sum_{m\in\Z^n}\lambda_{\nu, m}a_{\nu, m} \right|^{r} \right\}_{j=0}^{\infty} \right\|_{L^{p(\cdot)/r}(\ell^{q(\cdot)/r}) } ^{1/r} \notag \\ 
&+ \sum_{i=0}^R\left\| \left\{ \left|2^{j s(\cdot)} \inversefourier\varphi_j\ast \sum_{\nu=\nu_0}^{\infty} \sum_{\substack{m\in\Z^n : \\ Q_{\nu, m}\subset \Omega_i}} \lambda_{\nu, m}a_{\nu, m} \right|^{r} \right\}_{j=0}^{\infty} \right\|_{L^{p(\cdot)/r}(\ell^{q(\cdot)/r})(\Omega_i) } ^{1/r} \notag \\ 
 &\le I + \sum_{i=0}^R I_i,  
\end{align*}
where 
\[
I_i = \left\| \left\{ \left|2^{j s(\cdot)} \inversefourier\varphi_j\ast \sum_{\nu=\nu_0}^{\infty} \sum_{\substack{m\in\Z^n : \\ Q_{\nu, m}\subset \Omega_i}}\lambda_{\nu, m}a_{\nu, m} \right|^{r} \right\}_{j=0}^{\infty} \right\|_{L^{p(\cdot)/r}(\ell^{q(\cdot)/r})(\Omega_i) } ^{1/r}. 
\]
Then it suffices to show that $I\lesssim 1$ and $I_i\lesssim  1$ for any $i=0,1,\cdots, R$. 

Firstly, we denote the outline of the proof of $I_i\lesssim  1$ for any $i=0,1,\cdots, R$. 
Let fix non negative integer  $0\le i \le R$. 
We define 
\[
\lambda_{\nu,m}' = \begin{cases} \lambda_{\nu,m}, \quad &\text{if } Q_{\nu, m}\subset \Omega_i \\
                                          0 \quad &\text{otherwise}
                                          \end{cases}
\]
and 
\[
a_{\nu, m}' = \begin{cases} a_{\nu, m}, \quad &\text{if } Q_{\nu, m}\subset \Omega_i \\
                                          0 \quad &\text{otherwise}. 
                                          \end{cases}
\]
Then we have 
\[
I_i = \left\| \left\{ \left|2^{j s(\cdot)} \inversefourier\varphi_j\ast \sum_{\nu=0}^{\infty} \sum_{m\in\Z^n}\lambda_{\nu, m}'a'_{\nu, m} \right|^{r} \right\}_{j=0}^{\infty} \right\|_{L^{p(\cdot)/r}(\ell^{q(\cdot)/r})(\Omega_i) } ^{1/r} 
\]
and $a_{\nu,m}'$ is $[K,  ( \sigma_{p(\cdot)} )_{ \Omega_i }^+ -s_{  \Omega_i}^- +3\epsilon  ]$ atom centered at $Q_{\nu, m}$ by Remark $\ref{remark8-10}$.  
By 
using Theorem $\ref{theorem 3/18-1}$, we can prove $I_i\lesssim 1$ for any $i=0,1,\cdots, R$.

Finally, we denote the outline of the proof of $I\lesssim 1$. 
For any $\nu\in\N_0$ and $m\in\Z^n$, it is easy to see that 
$ \displaystyle 
\lfloor L_{Q_{\nu, m}}^- \rfloor \ge \lfloor L_{\R^n}^- \rfloor. 
$
This implies that we can choose $L$ as in Lemma $\ref{lemma-8-7-1}$ such that $ \lfloor L_{\R^n}^- \rfloor $. 
By Lemma $\ref{lemma-8-7-1}$, 
we obtain 
\begin{equation}
|2^{js(x)}\inversefourier\varphi_j\ast a_{\nu,m}(x)| \le \begin{cases} c2^{-|j - \nu| (\lfloor L_{\R^n}^- \rfloor +1+n)} 2^{js(x)} \left(1+2^{j}\left| x-2^{-\nu}m\right|\right)^{-M}  \quad \text{if } j\le \nu \label{13-8-15-3} \\ 
                                                  c2^{-|j - \nu| (K-s^+)} 2^{\nu s(x)} \left(1+2^{\nu}\left| x-2^{-\nu}m\right|\right)^{-M} \quad \text{if } j\ge \nu,  
\end{cases} 
\end{equation}
where $M$ is sufficiently large. 
Therefore, we have 
\begin{align}
&\left\| \left\{ \left| 2^{j s(\cdot)} \inversefourier\varphi_j\ast \sum_{\nu=0}^{\nu_0-1} \sum_{m\in\Z^n}\lambda_{\nu, m}a_{\nu, m} \right|^{r} \right\}_{j=0}^{\infty} \right\|_{L^{p(\cdot)/r}(\ell^{q(\cdot)/r}) } ^{1/r} \notag \\ 
& \le \left\| \left\{    \sum_{j=0}^{\infty} \left(      \left|\sum_{\nu=0}^{\nu_0-1} \sum_{m\in\Z^n}   2^{j s(\cdot)} \inversefourier\varphi_j\ast \lambda_{\nu, m}a_{\nu, m} \right|^{r}           \right)^{q(\cdot)/r}   \right\}^{r/{q(\cdot)}}  \right\|_{L^{p(\cdot)/r}}^{1/r}.  \label{8-15-1}
\end{align}
Let $\nu= w$. Then 
\begin{align}
&\left\| \left\{    \sum_{j=0}^{\infty} \left(    \left| \sum_{m\in\Z^n}   2^{j s(\cdot)} \inversefourier\varphi_j\ast \lambda_{w, m}a_{w, m} \right|^{r}           \right)^{q(\cdot)/r}   \right\}^{r/{q(\cdot)}}  \right\|_{L^{p(\cdot)/r}}^{1/r} \notag \\ 
&\lesssim \left( \sum_{j=0}^{w-1} \left\| \left\{ \left(   \left| \sum_{m\in\Z^n}   2^{j s(\cdot)} \inversefourier\varphi_j\ast \lambda_{w, m}a_{w, m} \right|^{r}           \right)^{q(\cdot)/r}   \right\}^{r/{q(\cdot)}}  \right\|_{L^{p(\cdot)/r}}\right)^{1/r}  \notag \\ 
&+ \left( \sum_{j=w}^{\infty} \left\| \left\{ \left(    \left| \sum_{m\in\Z^n}   2^{j s(\cdot)} \inversefourier\varphi_j\ast \lambda_{w, m}a_{w, m} \right|^{r}           \right)^{q(\cdot)/r}   \right\}^{r/{q(\cdot)}}  \right\|_{L^{p(\cdot)/r}} \right)^{1/r}  \notag \\ 
&\lesssim \left( \sum_{j=0}^{w-1} \left\|  \left| \sum_{m\in\Z^n}  2^{j s(\cdot)} \inversefourier\varphi_j\ast \lambda_{w, m}a_{w, m} \right|^{r}    \right\|_{L^{p(\cdot)/r}}\right)^{1/r}  \notag \\ 
&+ \left( \sum_{j=w}^{\infty} \left\|  \left| \sum_{m\in\Z^n}  2^{j s(\cdot)} \inversefourier\varphi_j\ast \lambda_{w, m}a_{w, m} \right|^{r}          \right\|_{L^{p(\cdot)/r}} \right)^{1/r} \notag \\ 
&=: J_1+J_2. \label{8-15-2}
\end{align}
We estimate $J_1$. By $(\ref{13-8-15-3})$, we obtain 
\begin{align}
J_1^r &=  \sum_{j=0}^{w-1} \left\|  \left| \sum_{m\in\Z^n}  2^{j s(\cdot)} \inversefourier\varphi_j\ast \lambda_{w, m}a_{w, m} \right|^{r}    \right\|_{L^{p(\cdot)/r}} \notag \\ 
&\le  \sum_{j=0}^{w-1} 
\left\| 
\sum_{m\in\Z^n}
2^{-|j - w|(\lfloor L_{\R^n}^- \rfloor +1+n) } 2^{js(\cdot)} \left| \lambda_{w, m} \right|\langle 2^{j}\cdot-2^{j-w}m\rangle^{-M}
\right\|_{L^{p(\cdot)}}^r  \notag \\ 
&= 
 \sum_{j=0}^{w-1} 2^{-|j - w|r(\lfloor L_{\R^n}^- \rfloor +1+n) }
\left\| 
\sum_{m\in\Z^n}
 2^{js(\cdot)} \left| \lambda_{w, m} \right|\langle 2^{j}\cdot-2^{j-w}m\rangle^{-M}
\right\|_{L^{p(\cdot)}}^r.   
\end{align} 
Using similar arguments of the proof of \cite[Theorem 3.13]{Kempka}, we see that 
\begin{align*}
J_1^r 
&\le  
 \sum_{j=0}^{w-1} 2^{-|j - w|r(\lfloor L_{\R^n}^- \rfloor +1+n) }
\left\| 
\sum_{m\in\Z^n}
 2^{js(\cdot)} \left| \lambda_{w, m} \right|\langle 2^{j}\cdot-2^{j-w}m\rangle^{-M}
\right\|_{L^{p(\cdot)}}^r  \notag \\ 
&\lesssim 
 \sum_{j=0}^{w-1} 2^{-|j - w|r(\lfloor L_{\R^n}^- \rfloor +1+n+s^--n/r) }
\left\| 
 \mathcal{M}_r \left(  2^{ws(y)} \sum_{m\in\Z^n} |\lambda_{w, m}| \chi_{w, m}(y)  \right) (\cdot)
\right\|_{L^{p(\cdot)}}^r   \notag 
\end{align*}
since we can take $M$ sufficiency large and Lemma $\ref{New-lemma}$. 
It is well known that $\mathcal{M}$ is bounded on $L^{p(\cdot)/r}$ by \cite[Theorem 4.3.8]{Diening-book}, we have 
\begin{align*}
J_1^r 
&\lesssim 
 \sum_{j=0}^{w-1} 2^{-|j - w|r(\lfloor L_{\R^n}^- \rfloor +1+n+s^--n/r) }
\left\| 
\mathcal{M}_r \left( 2^{ws(y)}\sum_{m\in\Z^n} |\lambda_{w, m}| \chi_{w, m}(y)  \right)(\cdot)
\right\|_{L^{p(\cdot)}}^r  \notag \\ 
&\lesssim 
 \sum_{j=0}^{w-1} 2^{-|j - w|r(\lfloor L_{\R^n}^- \rfloor +1+n+s^--n/r) }
\left\| 
2^{ws(\cdot)}\sum_{m\in\Z^n} |\lambda_{w, m}| \chi_{w, m}(\cdot) 
\right\|_{L^{p(\cdot)}}^r  \notag \\ 
&\lesssim 
w \max\left( 1, 2^{wr(\lfloor L_{\R^n}^- \rfloor +1+n+s^--n/r) } \right)
\left\| 
2^{ws(\cdot)}\sum_{m\in\Z^n} |\lambda_{w, m}| \chi_{w, m}(\cdot) 
\right\|_{L^{p(\cdot)}}^r. 
\end{align*}
By the assumption that $\| \lambda \|_{f_{p(\cdot),q(\cdot)}^{s(\cdot)}}=1$, we obtain 
\[
\left\| 
2^{ws(\cdot)}\sum_{m\in\Z^n} |\lambda_{w, m}| \chi_{w, m}(\cdot) 
\right\|_{L^{p(\cdot)}} \le 1
\]
because 
\[
\left\| 
2^{ws(\cdot)}\sum_{m\in\Z^n} |\lambda_{w, m}| \chi_{w, m}(\cdot) 
\right\|_{L^{p(\cdot)}} \le 
\left\| \left\{ \sum_{w=0}^{\infty} \left( 
2^{ws(\cdot)}\sum_{m\in\Z^n} |\lambda_{w, m}| \chi_{w, m}(\cdot) 
\right)^{q(\cdot)} \right\}^{1/{q(\cdot)}}
\right\|_{L^{p(\cdot)}}.  
\]
Therefore, we see that 
\begin{equation}
J_1^r 
\lesssim 
w \max\left( 1, 2^{wr(\lfloor L_{\R^n}^- \rfloor +1+n+s^--n/r) } \right).  \label{J_1}
\end{equation}
By using similar calculation as above, we obtain
\begin{align}
J_2^r
&=
 \sum_{j=w}^{\infty} \left\|  \left| \sum_{m\in\Z^n}  2^{j s(\cdot)} \inversefourier\varphi_j\ast \lambda_{w, m}a_{w, m} \right|^{r}          \right\|_{L^{p(\cdot)/r}}  \notag \\ 
 &\lesssim 
  \sum_{j=w}^{\infty} \left\|  \sum_{m\in\Z^n}  2^{(w-j) (K-s^+)} 2^{w s(\cdot)} \lambda_{w, m}\langle 2^{w}\cdot-m\rangle^{-M}           \right\|_{L^{p(\cdot)}}^r \notag \\ 
 &\lesssim 
  \sum_{j=w}^{\infty} \left\|  2^{(w-j) (K-s^+)} \mathcal{M}_r\left( 2^{w s(y)}\sum_{m\in\Z^n}   \lambda_{w, m}\chi_{w, m}(y) \right)(\cdot)          \right\|_{L^{p(\cdot)}}^r   \notag \\ 
   &\lesssim 
  \sum_{j=w}^{\infty} 2^{(w-j)r(K-s^+)} \left\| \mathcal{M}_r\left(2^{ws(y)}\sum_{m\in\Z^n}   \lambda_{w, m}\chi_{w, m}(y) \right)(\cdot)          \right\|_{L^{p(\cdot)}}^r.   
\end{align}
By the fact that $\mathcal{M}_r$ is bounded on $L^{p(\cdot)}$, we have 
\begin{align}
J_2^r
     &\lesssim 
  \sum_{j=w}^{\infty} 2^{(w-j)r(K-s^+)} \left\| 2^{ws(\cdot)}\sum_{m\in\Z^n}   \lambda_{w, m}\chi_{w, m}       \right\|_{L^{p(\cdot)}}^r    \notag \\ 
       &\lesssim 
  \sum_{j=w}^{\infty} 2^{(w-j)r(K-s^+)}. 
\end{align}
Since $K>s^+$, we have $J_2^r <\infty $ for any $0<w<\nu_0-1$. 
By $(\ref{8-15-1})$, $(\ref{J_1})$ and $J_2^r<\infty$, we see that 
\begin{align}
&\left\| \left\{ \left| 2^{j s(\cdot)} \inversefourier\varphi_j\ast \sum_{\nu=0}^{\nu_0-1} \sum_{m\in\Z^n}\lambda_{\nu, m}a_{\nu, m} \right|^{r} \right\}_{j=0}^{\infty} \right\|_{L^{p(\cdot)/r}(\ell^{q(\cdot)/r}) }  \notag \\ 
& \le \left\| \left\{    \sum_{j=0}^{\infty} \left(   \sum_{\nu=0}^{\nu_0-1}   \left| \sum_{m\in\Z^n}   2^{j s(\cdot)} \inversefourier\varphi_j\ast \lambda_{\nu, m}a_{\nu, m} \right|^{r}           \right)^{q(\cdot)/r}   \right\}^{r/{q(\cdot)}}  \right\|_{L^{p(\cdot)/r}} \notag \\ 
& \lesssim 
\sum_{w=0}^{\nu_0-1} \left( 
 \left\| \left\{    \sum_{j=0}^{w-1} \left(   \left| \sum_{m\in\Z^n}   2^{j s(\cdot)} \inversefourier\varphi_j\ast \lambda_{\nu, m}a_{\nu, m} \right|^{r}           \right)^{q(\cdot)/r}   \right\}^{r/{q(\cdot)}}  \right\|_{L^{p(\cdot)/r}} \right. \notag \\ 
&\qquad + \left. \left\| \left\{    \sum_{j=w}^{\infty} \left(   \left| \sum_{m\in\Z^n}   2^{j s(\cdot)} \inversefourier\varphi_j\ast \lambda_{\nu, m}a_{\nu, m} \right|^{r}           \right)^{q(\cdot)/r}   \right\}^{r/{q(\cdot)}}  \right\|_{L^{p(\cdot)/r}} 
 \right)  \notag \\ 
 &\lesssim \sum_{w=0}^{\nu_0-1}\left( J_1^r+J_2^r\right)<\infty. 
\end{align}
Therefore, we have $I\lesssim  \left(\sum_{w=0}^{\nu_0-1}\left( J_1^r+J_2^r\right) \right)^{1/r}<\infty$. 

\end{proof}

\subsection{Quarkonial decomposition for  Besov and Triebel--Lizorkin spaces with variable exponents : Regular case}
\label{sec quark}

%\begin{definition}
%\label{generate quark}
In this Section $\ref{sec quark}$, 
we fix a function $\psi\in\mathcal{S}(\R^n)$ uniquely such that 
\[
\sum_{m\in\Z^n}\psi(x-m)=1
\] 
holds for any $x\in\R^n$. 
We also fix a number $r>0$ such that 
\begin{equation}
\spt(\psi) \subset B(2^r). \label{quark r}
\end{equation}
Here $B(r):= \{y\in\R^n\,:\, |y|<r\}$. 
%\end{definition}

\begin{definition}
\label{quark coeficient seq}
For a triple-index sequence $\lambda=\{\lambda_{\nu,m}^{\beta}\}_{\beta\in\N_0^n, \nu\in\N_0, m\in\Z^n}$, we define 
\begin{equation}
\lambda^{\beta} := \{\lambda_{\nu,m}^{\beta}\}_{\nu\in\N_0, m\in\Z^n}, \ \ 
\left\|
\lambda
\right\|_{a_{p(\cdot),q(\cdot),\rho}^{s(\cdot)}} := 
\sup_{\beta\in\N_0^n}2^{\rho|\beta|}
\left\|
\lambda^{\beta}
\right\|_{a_{p(\cdot),q(\cdot)}^{s(\cdot)}}. \label{definition norm quark}
\end{equation}
\end{definition}

\begin{definition}
\label{def quark}
Let $p(\cdot),q(\cdot)\in\mathcal{P}(\R^n)\cap C^{\log}(\R^n)$ 
and $s(\cdot)\in C^{\log}(\R^n)$ satisfy 
\begin{eqnarray}
               &\mathop{{\rm{ess}}\,\,\rm{inf}}_{x\in\R^n}\left\{  s(\cdot) - \sigma_{p(\cdot),q(\cdot)}\right\}>0  &\text{(Triebel--Lizorkin case)} \label{Triebel case} 
\end{eqnarray}
and 
\begin{eqnarray}
  &\mathop{{\rm{ess}}\,\,\rm{inf}}_{x\in\R^n}\left\{ s(\cdot) - \sigma_{p(\cdot)}\right\}>0 &\text{(Besov case)} .  \label{Besov case} 
\end{eqnarray}
Let $\beta\in\N_0^n$, $\nu\in\N_0$ and $m\in\Z^n$. 
Then we define 
\begin{equation}
\psi^{\beta}(x) := x^{\beta}\psi(x), \ \ 
(\beta \qu)_{\nu,m}(x) := \psi^{\beta}(2^{\nu} x -m). \label{definition quark} 
\end{equation}
Furthermore we assume 
\begin{equation} 
\rho>r, \label{rho quark}
\end{equation} 
where $r$ is as in $(\ref{quark r})$. 
\end{definition}

Then we have a quarkonial decompositions for $A_{p(\cdot),q(\cdot)}^{s(\cdot)}(\R^n)$
\begin{theorem}[Quarkonial decomposition of regular cases]
\label{quark decomposition}
Let $\rho$ be as in Definition $\ref{def quark}$ and let  
$p(\cdot)$, $q(\cdot)$ and $s(\cdot)$ satisfy the condition of  Definition $\ref{def quark}$. 
Let $f\in\mathcal{S}'(\R^n)$. Then 
$f\in A_{p(\cdot),q(\cdot)}^{s(\cdot)}(\R^n)$ if and only if 
there exists a triple-index sequence $\lambda=\{\lambda_{\nu,m}^{\beta}\}_{\beta\in\N_0^n, \nu\in\N_0, m\in\Z^n}$ such that 
\begin{equation}
f=\sum_{\beta\in\N_0^n}\sum_{\nu\in\N_0}\sum_{m\in\Z^n}\lambda_{\nu,m}^{\beta}(\beta\qu)_{\nu,m} \label{expansion quark}
\end{equation}
and 
\begin{equation}
\left\|
\lambda
\right\|_{a_{p(\cdot),q(\cdot),\rho}^{s(\cdot)}} <\infty. \label{norm quark}
\end{equation}
Furthermore we can choose a coefficient $\lambda$ such that 
\begin{equation}
\left\|
\lambda
\right\|_{a_{p(\cdot),q(\cdot),\rho}^{s(\cdot)}} \sim 
\left\|
f
\right\|_{A_{p(\cdot),q(\cdot)}^{s(\cdot)}. \label{douchi quark}}
\end{equation}
\end{theorem}

\begin{remark}
In the classical setting, quarkonial decompositions of not only regular cases but also general cases for $A_{p,q}^{s}(\R^n)$ are found in \cite{Triebel2}. 
\end{remark}

To prove Theorem $\ref{quark decomposition}$, we need following Theorem $\ref{Fraizer-Jewerth}$, 
Corollary $\ref{Fraizer-Jewerth corollary}$, Lemma $\ref{heikouidou}$ and Lemma $\ref{kappa}$. 

\begin{theorem}[Fraizer--Jawerth $\varphi$ transform]
\label{Fraizer-Jewerth}
Let $\kappa\in\mathcal{S}(\R^n)$ satisfy 
$\chi_{Q(3)}\le \kappa \le \chi_{Q(3+1/{100})}$, where $Q(r):=\{ y\in\R^n\,:\,\max(|y_1|,|y_2|,\cdots,|y_n|) \le r)\}$.  
Then 
\begin{equation}
f=(2\pi)^{-\frac{n}{2}}\sum_{m\in\Z^n}f\left(\frac{m}{R}\right)\inversefourier\kappa (R\cdot-m)
\label{Fraizer}
\end{equation}
holds for any $f\in\mathcal{S}'(\R^n)$ such that $\spt(\fourier f)\subset Q(3R)$ $(R>0)$. 
\end{theorem}

Let $\tau, \varphi\in\mathcal{S}(\R^n)$ such that 
\[
\chi_{B(2)} \le \tau \le \chi_{B(3)}, \qquad \varphi_j(x) = \tau(2^{-j}x) -\tau(2^{-j+1}x),\,j\in\N_0.  
\]
Since  $f=\tau(D)f+\sum_{\nu=1}^{\infty}\varphi_{\nu}(D)f$ for any $f\in\mathcal{S}'(\R^n)$, 
we have following Corollary. 
\begin{corollary}
\label{Fraizer-Jewerth corollary} 
For any $f\in\mathcal{S}'(\R^n)$, we can write 
\begin{align*}
f= &(2\pi)^{-\frac{n}{2}}\sum_{m\in\Z^n}\tau(D)f(m)\inversefourier\kappa(\cdot-m) \\ 
&+(2\pi)^{-\frac{n}{2}}\sum_{\nu\in\N}
\left(
\sum_{m\in\Z^n}\varphi_{\nu}(D)f\left(\frac{m}{2^{\nu}}\right)\inversefourier\kappa(2^{\nu}\cdot-m)
\right). 
\end{align*}
\end{corollary}

\begin{lemma} 
\label{heikouidou} 
Let $p(\cdot), q(\cdot)\in C^{\log}(\R^n)\cap \mathcal{P}_0(\R^n)$ and $s(\cdot)\in C^{\log}(\R^n)$. 
Furthermore, let  $0<\alpha<\min(1, p^-,q^-)$ and $\lambda^{l}:=\{\lambda_{\nu,m+l}\}_{\nu\in\N_0,m\in\Z^n}$ 
for any $l\in\Z^n$. 
\begin{enumerate}
\item If $M>2\max(C_{\log}(s), 2n)$, then 
\begin{equation}
\left\| 
\lambda^{l}
\right\|_{b_{p(\cdot),q(\cdot)}^{s(\cdot)}} 
\lesssim \langle l \rangle^{\frac{M}{\alpha}} 
\left\| 
\lambda
\right\|_{b_{p(\cdot),q(\cdot)}^{s(\cdot)}} \label{5.61-1}
\end{equation}
holds, where $\langle  l \rangle := \sqrt{1+l_1^2+l_2^2+\cdots+l_n^2}$. 
\item If $M>2\max(C_{\log}(s), n)$, then 
\begin{equation}
\left\| 
\lambda^{l}
\right\|_{f_{p(\cdot),q(\cdot)}^{s(\cdot)}} 
\lesssim \langle l \rangle^{\frac{M}{\alpha}} 
\left\| 
\lambda
\right\|_{f_{p(\cdot),q(\cdot)}^{s(\cdot)}} \label{5.61-2}
\end{equation}
holds. 
\end{enumerate}
\end{lemma}

\begin{proof}
Let $0<\alpha<\min(p^-,q^-)$. 
We fix $\nu\in\N_0$ and $x\in Q_{\nu,m}$. Then $Q_{\nu,m+l}\subset B(x, \sqrt{2n}2^{-\nu}\langle l \rangle)$ holds. 
Furthermore, we have 
\begin{align*}
1\lesssim \frac{((\sqrt{2n}+1)\langle l \rangle)^M}{(1+2^{\nu}|x-y|)^M},  
\end{align*}
where $y\in  B(x,\sqrt{2n}2^{-\nu}\langle l \rangle)$. 
Hence we see that 
\begin{align*}
&|\lambda_{\nu,m+l}|^{\alpha}\chi_{\nu,m}(x) \notag \\ 
&\lesssim \frac{\chi_{\nu,m}(x) }{|Q_{\nu,m+l}|}\int_{B(x,\sqrt{2n}2^{\nu}\langle l\rangle)} \sum_{m'\in\Z^n}|\lambda_{\nu,m'}^{\alpha}|\chi_{\nu,m'}(y){\rm d}y \\ 
&\lesssim ((\sqrt{2n}+1)\langle l \rangle)^M \chi_{\nu,m}(x)\int_{B(x,\sqrt{2n}2^{\nu}\langle l\rangle)}\frac{2^{\nu n}}{(1+2^{\nu}|x-y|)^M}
\left(\sum_{m'\in\Z^n}|\lambda_{\nu,m'}|^{\alpha}\chi_{\nu,m'}(y)\right){\rm d}y \\ 
&\lesssim ((\sqrt{2n}+1)\langle l \rangle)^M \chi_{\nu,m}(x)\left( \eta_{\nu,M}\ast \left(\sum_{m'\in\Z^n}|\lambda_{\nu,m'}|\chi_{\nu,m'}\right)^{\alpha}\right)(x). 
\end{align*} 
By Lemma $\ref{lemma:2}$ and $M\ge 2C_{\log}(s)$, we have 
\begin{equation*}
\left|2^{\nu s(x)}\sum_{m\in\Z^n}\lambda_{\nu,m+l}\chi_{\nu,m}(x)\right| 
\lesssim \langle l \rangle^{\frac{M}{\alpha}} 
\left( \eta_{\nu,\frac{M}{2}}\ast \left(2^{\nu s(\cdot)}\sum_{m\in\Z^n}|\lambda_{\nu,m}|\chi_{\nu,m}(\cdot)\right)^{\alpha}\right)^{\frac{1}{\alpha}}(x). 
\end{equation*} 
Thus $(\ref{5.61-1})$ and $(\ref{5.61-2})$ hold by Theorem $\ref{thm:max2}$ and $\ref{thm:max3}$. 
\end{proof}

\begin{lemma}
\label{kappa}
Let $\kappa\in\mathcal{S}(\R^n)$ satisfy 
$\chi_{Q(3)}\le \kappa \le \chi_{Q(3+1/{100})}$. Then, 
for any $N\gg 1$, we have 
\begin{equation}
\left|
\partial^{\alpha}\inversefourier\kappa(y)\right| \lesssim_N\langle\alpha\rangle^{2N}\langle y\rangle^{-2N}. \label{kappa estimate}
\end{equation}
\end{lemma}

\begin{proof}
By using integration by parts, we can see that 
\begin{align*}
\partial^{\alpha} \inversefourier\kappa(y) &\simeq_n \int_{\R^n} (iz)^{\alpha}\kappa(z)\exp(iz\cdot y)\,{\rm d}z \notag \\ 
&\simeq_n \langle y\rangle^{-2N} \int_{\R^n} (iz)^{\alpha}\kappa(z) \left( (1-\Delta_z)^N\exp(iz\cdot y)\right)\,{\rm d}z \notag \\ 
&\simeq_n \langle y\rangle^{-2N} \int_{\R^n} \left( (1-\Delta_z)^N (iz)^{\alpha}\kappa(z) \right) \exp(iz\cdot y) \,{\rm d}z, 
\end{align*}
where $\Delta_z = \sum_{i=1}^n\partial^2/{\partial z_i^2} $. 
Hence we have $(\ref{kappa estimate})$. 
\end{proof}

From now, we prove Theorem $\ref{quark decomposition}$. 
\begin{proof}[Proof of Theorem $\ref{quark decomposition}$] We divide the proof into the parts of sufficiency and necessity.  

{\bf Sufficiency. } 
Since we assume $(\ref{rho quark})$, 
we can take $\epsilon>0$ such that 
$0<\epsilon<\rho-r$. 
For any $\nu\in\N_0$ and $m\in\Z^n$, there exists $d>0$ such that 
$\spt (\beta\qu)_{\nu,m}\subset dQ_{\nu,m}$. 
The conditions  $(\ref{Triebel case})$ and $(\ref{Besov case})$ imply that smooth atoms in $A_{p(\cdot),q(\cdot)}^{s(\cdot)}(\R^n)$ are not required to satisfy any moment conditions. 
Hence we can regard $2^{-(r+\epsilon)|\beta|}(\beta\qu)_{\nu,m}$ as smooth atoms in $A_{p(\cdot),q(\cdot)}^{s(\cdot)}(\R^n)$.  
We define 
\[
f^{\beta}:= \sum_{\nu\in\N_0}\sum_{m\in\Z^n}\lambda_{\nu,m}^{\beta}(\beta\qu)_{\nu,,m}. 
\]
By Theorem $\ref{theorem 5.1.6-Besov} $ and $\ref{theorem 5.1.6-Triebel}$, 
we have 
\[
\left\| 
f^{\beta} 
\right\|_{A_{p(\cdot),q(\cdot)}^{s(\cdot)}}
\lesssim 
2^{-(\rho-r-\epsilon)|\beta|}
\left\|
 \lambda 
\right\|_{a_{p(\cdot),q(\cdot),\rho}^{s(\cdot)}}. 
\] 
Let 
\[
\sigma = 
\min
\left\{ 
\min
\left(
q^-,\,1
\right)
\min
\left(1,\,
\left(
\frac{p}{q}
\right)^-
\right), \, 
\min(p^-,q^-,1)
\right\}. 
\]
Since 
\[
\left\| 
f_1+f_2
\right\|_{A_{p(\cdot),q(\cdot)}^{s(\cdot)}}^{\sigma}  
\le \left\| 
f_1
\right\|_{A_{p(\cdot),q(\cdot)}^{s(\cdot)}}^{\sigma} 
+ 
\left\| 
f_2
\right\|_{A_{p(\cdot),q(\cdot)}^{s(\cdot)}}^{\sigma} 
\]
holds for any $f_1,f_2\in A_{p(\cdot),q(\cdot)}^{s(\cdot)}(\R^n)$ by Lemma $\ref{min triangle}$, 
we have $f\in A_{p(\cdot),q(\cdot)}^{s(\cdot)}(\R^n)$ and $(\ref{norm quark})$. 

{\bf Necessity. } 
Let $f\in A_{p(\cdot),q(\cdot)}^{s(\cdot)}(\R^n)$. Then we can write 
\begin{align}
f= &(2\pi)^{-\frac{n}{2}}\sum_{m\in\Z^n}\tau(D)f(m)\inversefourier\kappa(\cdot-m) \notag \\ 
&+(2\pi)^{-\frac{n}{2}}\sum_{\nu\in\N}
\left(
\sum_{m\in\Z^n}\varphi_{\nu}(D)f\left(\frac{m}{2^{\nu}}\right)\inversefourier\kappa(2^{\nu}\cdot-m) 
\right) \label{tenkai}
\end{align}
by Corollary $\ref{Fraizer-Jewerth corollary}$. 
For any $(\nu,m)\in\N_0\times\Z^n$, we define 
\begin{equation}
\Lambda_{\nu,m} = \begin{cases} 
\tau(D)f(m)  \ \ &(\nu=0) \notag \\ 
 \varphi_{\nu}(D)f\left(\frac{m}{2^{\nu}}\right) \ \ &(\nu\in\N). 
\end{cases}
 \label{5.64}
\end{equation}
Then we rewrite $(\ref{tenkai})$ to 
\begin{equation}
f\simeq_n\sum_{\nu\in\N_0}\sum_{m\in\Z^n}\Lambda_{\nu,m}\inversefourier\kappa(2^{\nu}\cdot-m). \label{5.65}
\end{equation}
We can assume that $\rho$ is a integer. 
By the Taylor expansion, we obtain  
\begin{align*}
\psi(2^{\nu+\rho}x-l)&\inversefourier\kappa(2^{\nu}x-m) \\ 
&= \sum_{\beta\in\N_0}
\frac{\partial^{\beta}\inversefourier\kappa(2^{-\rho}l-m)(2^{\nu}x-2^{-\rho}l)^{\beta}\psi(2^{\nu+\rho}x-l)}
{\beta!} \\ 
&= \sum_{\beta\in\N_0^n}\frac{2^{-\rho|\beta|}\partial^{\beta}\inversefourier\kappa(2^{-\rho}l-m)\psi^{\beta}(2^{\nu+\rho}x-l)}
{\beta!}. 
\end{align*} 
Since $\sum_{m\in\Z^n}\psi(x-m)=1$, 
we see that 
\begin{align}
\varphi_{\nu}&(D)f \simeq_n \sum_{m\in\Z^n}\sum_{l\in\Z^n}\sum_{\beta\in\N_0^n}
\frac{2^{-\rho|\beta|}}{\beta!}
\Lambda_{\nu,m}\partial^{\beta}\inversefourier\kappa(2^{-\rho}l-m)\psi^{\beta}(2^{\nu+\rho}x-l). \label{tenkai2}
\end{align}
Since we can regard  $(\ref{tenkai2})$ as converging in the topology of $L^{\infty}$, 
we can change the order of summation. Hence we can rewrite $(\ref{tenkai2})$ as 
\begin{align*}
\varphi_{\nu}(D)f 
\simeq_n 
\sum_{l\in\Z^n}\sum_{\beta\in\N_0^{n}}\sum_{m\in\Z^n}
\frac{2^{-\rho|\beta|}}{\beta!}
\Lambda_{\nu,m}\partial^{\beta}\inversefourier\kappa(2^{-\rho}l-m)(\beta\qu)_{\nu+\rho,l}(x). 
\end{align*}
Let 
\[
\lambda_{\nu+\rho,l}^{\beta}:= \frac{2^{-\rho|\beta|}}{\beta!}\sum_{m\in\Z^n}\Lambda_{\nu,m}\partial^{\beta}\inversefourier\kappa(2^{-\rho}l-m). 
\]
Then we have 
\begin{equation}
f=\sum_{\nu\in\N_0}\varphi_{\nu}(D)f\simeq_n 
\sum_{\nu\in\N_0}\sum_{l\in\Z^n}\sum_{\beta\in\N_0^n}\lambda_{\nu+\rho,l}^{\beta}(\beta\qu)_{\nu+\rho,l}. \label{5.66}
\end{equation}

Next we consider the $a_{p(\cdot),q(\cdot),\rho}^{s(\cdot)}$ quasi norm of coefficients. 
Let $l\in\Z^n$, $l_0$ be a lattice point of $[0,2^{\rho})^n$  
and $x\in Q_{\nu+\rho, 2^{\rho}l+l_0 }$. 
By $(\ref{kappa estimate})$, we obtain 
\begin{equation}
\left|
\lambda_{\nu+\rho,2^{\rho}l+l_0}^{\beta}
\right|
\lesssim 
2^{-\rho|\beta|}\sum_{m\in\Z^n}\langle l-m\rangle^{-N}\left| \Lambda_{\nu,m}\right| 
= c2^{-\rho|\beta|}\sum_{m\in\Z^n}\langle m\rangle^{-N}
\left| \Lambda_{\nu,m+l}\right|. \label{8-13-1} 
\end{equation}
For each $m\in\Z^n$, we define 
\[
\eta_0:=\min\left( \min(1, p^-,q^-),  \min(1, q^-)\min\left(1, \left( \frac{p}{q} \right)^- \right)  \right), \ \ 
\Lambda^m:=\left\{ \left| \Lambda_{\nu,m+l}\right| \right\}_{\nu\in\N_0,l\in\Z^n}. 
\]
By $\Z^n= 2^{\rho}\Z^n +[0,2^{\rho})^n$, $(\ref{8-13-1})$ and Lemma $\ref{min triangle}$, we see that 
\begin{align*}
\left\| \lambda^{\beta} \right\|_{a_{p(\cdot),q(\cdot)}^{s(\cdot)}} 
&\lesssim 2^{-\rho|\beta|}
\left\|
\sum_{m\in\Z^n}\langle m \rangle^{-N}\Lambda^m
\right\|_{a_{p(\cdot),q(\cdot)}^{s(\cdot)}} \lesssim 
2^{-\rho|\beta|}
\left(
\sum_{m\in\Z^n}\langle m\rangle^{-N\eta_0} 
\left\|
\Lambda^m
\right\|_{a_{p(\cdot),q(\cdot)}^{s(\cdot)}}^{\eta_0}
\right)^{\frac{1}{\eta_0}}. 
\end{align*}
Since we can take $N$ sufficiency large, 
we obtain  
\[
\left\| \lambda^{\beta} \right\|_{a_{p(\cdot),q(\cdot)}^{s(\cdot)}} 
\lesssim 
2^{-\rho|\beta|}
\left(
\sum_{m\in\Z^n}\langle m\rangle^{(\frac{M}{\alpha}-N)\eta_0} 
\left\|
\Lambda
\right\|_{a_{p(\cdot),q(\cdot)}^{s(\cdot)}}^{\eta_0}
\right)^{\frac{1}{\eta_0}} 
\lesssim 
2^{-\rho|\beta|}
\left\|
\Lambda
\right\|_{a_{p(\cdot),q(\cdot)}^{s(\cdot)}}
\]
by Lemma $\ref{heikouidou}$, 
where $\Lambda=\{\Lambda_{\nu,m}\}_{\nu\in\N_0,m\in\Z^n}$.  
Hence we have 
\[
\left\| 
\lambda
\right\|_{a_{p(\cdot),q(\cdot),\rho}^{s(\cdot)}}\lesssim 
\left\| 
\Lambda
\right\|_{a_{p(\cdot),q(\cdot)}^{s(\cdot)}}
\]
by $\rho>r$. 

Finally, we prove  $||\Lambda||_{a_{p(\cdot),q(\cdot)}^{s(\cdot)}}\lesssim ||f||_{A_{p(\cdot),q(\cdot)}^{s(\cdot)}}$. 
Let $\eta_1=\min(1, p^-, q^-)$ and $M$ be a sufficiency large. 
For any $y\in Q_{\nu, m}$, we have 
\[
\frac{1}{\left( 1+ 2^{\nu}|y-2^{-\nu}m| \right)^{\frac{2M}{\eta_1}}  } 
\left| 
\varphi_{\nu}(D)f\left( \frac{m}{2^{\nu}} \right) 
\right|
\lesssim \left(\eta_{\nu,M}\ast (\varphi_{\nu}(D)f)^{\frac{\eta_1}{2}}(y)\right)^{\frac{2}{\eta_1}}
\]
by Lemma $\ref{lemma:r}$. 
Hence we see that 
\begin{align*}
|\Lambda_{\nu, m}| 
&= \left| 
\varphi_{\nu}(D)f\left( \frac{m}{2^{\nu}} \right) 
\right| \notag \\ 
&= \left( 1+ 2^{\nu}|y-2^{-\nu}m| \right)^{\frac{2M}{\eta_1}} \cdot 
\frac{1}{\left( 1+ 2^{\nu}|y-2^{-\nu}m| \right)^{\frac{2M}{\eta_1}}  } 
\left| 
\varphi_{\nu}(D)f\left( \frac{m}{2^{\nu}} \right) 
\right|  \notag \\ 
&\lesssim  
(1+n)^{\frac{2M}{\eta_1} } \frac{1}{\left( 1+ 2^{\nu}|y-2^{-\nu}m| \right)^{\frac{2M}{\eta_1}}  } 
\left| 
\varphi_{\nu}(D)f\left( \frac{m}{2^{\nu}} \right) 
\right|  \notag \\ 
&\lesssim \left(\eta_{\nu,M}\ast (\varphi_{\nu}(D)f)^{\frac{\eta_1}{2}}(y)\right)^{\frac{2}{\eta_1}}. 
\end{align*}
Since we have 
\[
\left| \Lambda_{\nu,m} 
\right| 
\lesssim 
\inf_{y\in Q_{\nu,m}}\left(\eta_{\nu,M}\ast (\varphi_{\nu}(D)f)^{\frac{\eta_1}{2}}(y)
\right)^{\frac{2}{\eta_1}}, 
\]
we obtain $||\Lambda||_{a_{p(\cdot),q(\cdot)}^{s(\cdot)}}\lesssim ||f||_{A_{p(\cdot),q(\cdot)}^{s(\cdot)}}$. 
This proves the necessity of quarkonial decomposition. 
\end{proof}

\section{Application to Trace theory}
\label{sec trace}
Let $n\ge 2$. In this Section, we consider the Trace operator 
\begin{equation}
\tr : f(x',x_n) \longmapsto  f(x',0), \ \ x'\in\R^{n-1}, \ \ f\in\mathcal{S}(\R^n). \label{5.152}
\end{equation} 
We write $x=(x',x_n)\in\R^n$, where $x'\in\R^{n-1}$ and $x_n\in\R$. 
Furthermore,  we write 
$\tilde{p}(x')=p(x', 0)$, $\tilde{q}(x')=q(x',0)$ and $\tilde{s}(x')=s(x',0)$. 

\begin{theorem}
\label{theorem 5.4.4}
Assume that $p(\cdot),q(\cdot)\in C^{\log}(\R^n)\cap\mathcal{P}_0(\R^n)$. 

\begin{enumerate}
\item[{\rm (1)}] Let $s(\cdot)\in C^{\log}(\R^n)$ satisfy 
\begin{equation}
\mathop{{\rm{ess}}\,\,\rm{inf}}_{x\in\R^n} 
\left\{ s(\cdot)-\left[\frac{1}{p(\cdot)}+(n-1)\left( \frac{1}{\min(1, p(\cdot) )}-1 \right) \right] \right\}>0. \label{eq:10-6-1}
\end{equation} 
\begin{enumerate}
\item The operator $\tr$ can be extended as a surjective and continuous mapping from $B_{p(\cdot),q(\cdot)}^{s(\cdot)}(\R^n)$ 
to  $B_{ \tilde{p}(\cdot),\tilde{q}(\cdot)}^{ \tilde{s}(\cdot)-\frac{1}{\tilde{p}(\cdot)}}(\R^{n-1})$. 
\item The operator $\tr$ can be extended as a surjective and continuous mapping from $F_{p(\cdot),q(\cdot)}^{s(\cdot)}(\R^n)$ 
to  $F_{\tilde{p}(\cdot),\tilde{p}(\cdot)}^{\tilde{s}(\cdot)-\frac{1}{\tilde{p}(\cdot)}}(\R^{n-1})$.  
\end{enumerate} 
\item[{\rm (2)}] Let $s(\cdot)\in C^{\log}(\R^n)$ and $k\in\N_0$ satisfy 
\begin{equation}
\mathop{{\rm{ess}}\,\,\rm{inf}}_{x\in\R^n}\left\{ s(\cdot)-\left[k+\frac{1}{p(\cdot)}+(n-1)\left( \frac{1}{ \min(1, p(\cdot) )}-1\right)\right]\right\}>0.
\label{eq:10-6-2}
\end{equation}
\begin{enumerate}
\item  If $g_0\in B_{\tilde{p}(\cdot),\tilde{q}(\cdot)}^{\tilde{s}(\cdot)-\frac{1}{\tilde{p}(\cdot)}}(\R^{n-1})$, 
$g_1\in B_{\tilde{p}(\cdot),\tilde{q}(\cdot)}^{\tilde{s}(\cdot)-\frac{1}{\tilde{p}(\cdot)}-1}(\R^{n-1})$, $\cdots$, 
$g_k\in B_{\tilde{p}(\cdot),\tilde{q}(\cdot)}^{\tilde{s}(\cdot)-\frac{1}{\tilde{p}(\cdot)}-k}(\R^{n-1})$, then there exists a $f\in B_{p(\cdot),q(\cdot)}^{s(\cdot)}(\R^n)$ such that 
$\tr(f)=g_0$, $\tr(\partial_{x_n}f)=g_1$, $\cdots$, $\tr(\partial_{x_n}^kf)=g_k$. 
\item  If $g_0\in F_{\tilde{p}(\cdot),\tilde{p}(\cdot)}^{\tilde{s}(\cdot)-\frac{1}{\tilde{p}(\cdot)}}(\R^{n-1})$, 
$g_1\in F_{\tilde{p}(\cdot),\tilde{p}(\cdot)}^{\tilde{s}(\cdot)-\frac{1}{\tilde{p}(\cdot)}-1}(\R^{n-1})$, $\cdots$, 
$g_k\in F_{\tilde{p}(\cdot),\tilde{p}(\cdot)}^{\tilde{s}(\cdot)-\frac{1}{\tilde{p}(\cdot)}-k}(\R^{n-1})$, then there exists a $f\in F_{p(\cdot),q(\cdot)}^{s(\cdot)}(\R^n)$ such that 
$\tr(f)=g_0$, $\tr(\partial_{x_n}f)=g_1$, $\cdots$, $\tr(\partial_{x_n}^kf)=g_k$. 
\end{enumerate} 
\end{enumerate} 
\end{theorem}

\begin{remark}
If $p(\cdot)=p$, $q(\cdot)=q$ and $s(\cdot)=s$ are constant functions, 
then it is known that the assumption $s>\frac{1}{p}+(n-1)\left( \frac{1}{\min(1, p) }-1 \right)$ is optimal, see \cite{Triebel1}. 

As we mentioned in Introduction,  Diening, H\"ast\"o and Roudenko \cite{Diening} proved Theorem $\ref{theorem 5.4.4}$-$(1)$ for Triebel--Lizrokin spaces with variable exponents. 
In the case of Besov spaces with variable exponents, Moura, Neves and Schneider \cite{Moura} proved Theorem $\ref{theorem 5.4.4}$-$(1)$ 
for $2$-microlocal Besov spaces wtih variable exponents, but summability index $q$ was constant. 
\end{remark}

To prove Theorem $\ref{theorem 5.4.4}$, we need Lemma $\ref{lemma 7.1}$ and Lemma $\ref{lemma 5.4.5}$.  
\begin{lemma}[ {\cite[Lemma 7.1]{Diening} }]
\label{lemma 7.1} 
Let $p(\cdot),q(\cdot)\in C^{\log}(\R^n)\cap\mathcal{P}_0(\R^n)$, $s(\cdot)\in C^{\log}(\R^n)$, $\epsilon>0$ and let $\{ E_{\nu,m}\}_{(\nu,m)\in\N_0\times\Z^n}$ be 
a collection of sets such that $E_{\nu,m}\subset 3Q_{\nu,m}$ and $|E_{\nu,m}|\ge \epsilon |Q_{\nu,m}|$.  
Then 
\begin{equation}
\left\| \left\{ \lambda_{\nu,m}\right\}_{(\nu,m)\in\N_0\times\Z^n} \right\|_{f_{p(\cdot),q(\cdot)}^{s(\cdot)}} 
\sim\left\| \left\{ 2^{\nu s(\cdot)}\sum_{m\in\Z^n}\left| \lambda_{\nu,m}\right| \chi_{E_{\nu,m}} \right\}_{\nu=0}^{\infty} \right\|_{L^{p(\cdot)}(\ell^{q(\cdot)})} \label{trie}
\end{equation}
for any $\{ \lambda_{\nu,m}\}_{(\nu,m)\in\N_0\times\Z^n}\in f_{p(\cdot),q(\cdot)}^{s(\cdot)}$ 
and 
\begin{equation}
\left\| \left\{ \lambda_{\nu,m}\right\}_{(\nu,m)\in\N_0\times\Z^n} \right\|_{b_{p(\cdot),q(\cdot)}^{s(\cdot)}} 
\sim\left\| \left\{ 2^{\nu s(\cdot)}\sum_{m\in\Z^n}\left| \lambda_{\nu,m}\right| \chi_{E_{\nu,m}} \right\}_{\nu=0}^{\infty} \right\|_{\ell^{q(\cdot)}(L^{p(\cdot)})} \label{besov}
\end{equation}
for any $\{ \lambda_{\nu,m}\}_{(\nu,m)\in\N_0\times\Z^n}\in b_{p(\cdot),q(\cdot)}^{s(\cdot)}$.  
\end{lemma}
Triebel--Lizorkin case $(\ref{trie})$ is proved in \cite[Lemma 7.1]{Diening}. By using same argument of the proof of \cite[Lemma 7.1]{Diening}, 
we can prove Besov case $(\ref{besov})$. 

\begin{lemma}[ {\cite[Lemma 7.2]{Diening}} ]  
\label{lemma 7.2} 
Let $p_1(\cdot), p_2(\cdot), q_1(\cdot), q_2(\cdot)\in C^{\log}(\R^n)\cap\mathcal{P}_0(\R^n)$ and $s_1(\cdot), s_2(\cdot)\in C^{\log}(\R^n)$. 
Assume that $p_1(\cdot)=p_2(\cdot)$, $q_1(\cdot)=q_2(\cdot)$ and $s_1(\cdot)=s_2(\cdot)$ in the upper or lower half space. 
For double-index complex-valued sequence $\lambda=\{\lambda_{\nu,m}\}_{(\nu,m)\in\N_0\times\Z^n}$, 
\begin{equation}
\left\| 
\left\{
\delta_{m_n,0}\lambda_{\nu,m}
\right\}_{(\nu,m)\in\N_0\times\Z^n}
\right\|_{a_{p_1(\cdot),q_1(\cdot)}^{s_1(\cdot)}(\R^{n})}
\sim 
\left\| 
\left\{
\delta_{m_n,0}\lambda_{\nu,m}
\right\}_{(\nu,m)\in\N_0\times\Z^n}
\right\|_{a_{p_2(\cdot),q_2(\cdot)}^{s_2(\cdot)}(\R^{n})}, \label{7.2}
\end{equation}
where 
\[
\delta_{m_n,0} =\begin{cases} 1 \ \ &m_n=0, \\ 
                                     0 \ \ &m_n\neq 0. 
                                     \end{cases}
\]
\end{lemma}

\begin{proof}
Triebel--Lizorkin case is proved in \cite{Diening}. Hence we prove the Besov case by using similar argument in the proof of \cite[Lemma 7.2]{Diening}. 
We prove the case of  $p_1(\cdot)=p_2(\cdot)$, $q_1(\cdot)=q_2(\cdot)$ and $s_1(\cdot)=s_2(\cdot)$ in the lower half space because 
it is obvious that $(\ref{7.2})$ holds if  $p_1(\cdot)=p_2(\cdot)$, $q_1(\cdot)=q_2(\cdot)$ and $s_1(\cdot)=s_2(\cdot)$ in the upper half space. 

For $m=(m',0)\in\Z^n$, we put 
\[
E_{\nu,m} = \left\{(x',x_n)\in\R^n\,:\, (x',-x_n)\in Q_{\nu,m}, \, -\frac{3}{4}2^{-\nu}\le x_n \le -\frac{1}{2}2^{-\nu} \right\}; 
\]
for all other $m\in\Z^n$, we put $E_{\nu,m}=Q_{\nu,m}$. 
Since $E_{\nu,m}$ is supported in the lower space when $m_n=0$, by Lemma $\ref{lemma 7.1}$, we see that 
\begin{align*}
\left\| 
\left\{
\delta_{m_n,0}\lambda_{\nu,m}
\right\}_{(\nu,m)\in\N_0\times\Z^n}
\right\|_{b_{p_2(\cdot),q_2(\cdot)}^{s_2(\cdot)}(\R^{n})} 
&\sim \left\| 
\left\{
2^{\nu s_2(\cdot)}\sum_{m\in\Z^n} \delta_{m_n,0}\lambda_{\nu,m}E_{\nu,m}
\right\}_{\nu=0}^{\infty}
\right\|_{\ell^{q_2(\cdot)}(L^{p_2(\cdot)})} \\ 
&\sim 
\left\|
\left\{
2^{\nu s_1(\cdot)}\sum_{m\in\Z^n} \delta_{m_n,0}\lambda_{\nu,m}E_{\nu,m}
\right\}_{\nu=0}^{\infty}
\right\|_{\ell^{q_1(\cdot)}(L^{p_1(\cdot)})} \\ 
&\sim 
\left\| 
\left\{
\delta_{m_n,0}\lambda_{\nu,m}
\right\}_{(\nu,m)\in\N_0\times\Z^n}
\right\|_{b_{p_1(\cdot),q_1(\cdot)}^{s_1(\cdot)}(\R^{n})}. 
\end{align*}
This complete the proof. 
\end{proof}

\begin{corollary}[cf {\cite[Proposition 7.3]{Diening}}]
\label{cor proposition 7.3} 
Let $p_1(\cdot), p_2(\cdot), q_1(\cdot), q_2(\cdot)\in C^{\log}(\R^n)\cap\mathcal{P}_0(\R^n)$ and $s_1(\cdot), s_2(\cdot)\in C^{\log}(\R^n)$. 
Assume that $p_1(x)=p_2(x)$, $q_1(x)=q_2(x)$ and $s_1(x)=s_2(x)$ for all $x\in\R^{n-1}\times\{0\}$. 
For double-index complex-valued sequence $\lambda=\{\lambda_{\nu,m}\}_{(\nu,m)\in\N_0\times\Z^n}$, 
\begin{equation}
\left\| 
\left\{
\delta_{m_n,0}\lambda_{\nu,m}
\right\}_{(\nu,m)\in\N_0\times\Z^n}
\right\|_{a_{p_1(\cdot),q_1(\cdot)}^{s_1(\cdot)}(\R^{n})}
\sim 
\left\| 
\left\{
\delta_{m_n,0}\lambda_{\nu,m}
\right\}_{(\nu,m)\in\N_0\times\Z^n}
\right\|_{a_{p_2(\cdot),q_2(\cdot)}^{s_2(\cdot)}(\R^{n})}.  \label{7.3}
\end{equation}
\end{corollary} 

Triebel--Lizorkin case is proved in \cite[Proposition 7.3]{Diening}. By using same argument in the proof of \cite[Proposition 7.3]{Diening}, we can prove the Besov case. 

\begin{lemma}
\label{lemma 5.4.5}
Let $p(\cdot),q(\cdot)\in C^{\log}(\R^n)\cap\mathcal{P}_0(\R^n)$ and $s(\cdot)\in C^{\log}(\R^n)$. 
For double-index complex-valued sequence $\lambda=\{\lambda_{\nu,m}\}_{(\nu,m)\in\N_0\times\Z^n}$, 
we have 
\[
\left\| 
\left\{
\lambda_{\nu,(m',0)}
\right\}_{(\nu,m')\in\N_0\times\Z^{n-1}}
\right\|_{f_{\tilde{p}(\cdot),\tilde{p}(\cdot)}^{\tilde{s}(\cdot)-\frac{1}{\tilde{p}(\cdot)}}(\R^{n-1})}  
\sim
\left\| 
\left\{
\delta_{m_n,0}\lambda_{\nu,m}
\right\}_{(\nu,m)\in\N_0\times\Z^n}
\right\|_{f_{p(\cdot),q(\cdot)}^{s(\cdot)}(\R^{n})}. 
\]
\end{lemma}

\begin{proof} 
By Corollary $\ref{cor proposition 7.3}$, it suffices to consider the case $p(\cdot)$, $q(\cdot)$ and $s(\cdot)$ independent of the $n$-th coordinate for $|x_n|\le 2$. 
Let $\tilde{Q}_{\nu,m'}=Q_{\nu,m'}\times[2^{-\nu}, 2^{-\nu+1})$ for $\nu\in\N_0$ and $m'\in\Z^{n-1}$. Firstly, we prove 
\begin{align}
&\left\| \left\{ 2^{\nu(\tilde{s}(\cdot)-\frac{1}{\tilde{p}(\cdot)})} \sum_{m'\in\Z^{n-1}}\lambda_{\nu,(m',0)}\chi_{\nu, m'}\right\}_{\nu=0}^{\infty} \right\|_{L^{\tilde{p}(\cdot)}(\ell^{\tilde{p}(\cdot)})}  \notag \\ 
&\qquad \sim 
\left\| \left\{ 2^{\nu s(\cdot)} 
\sum_{m'\in\Z^{n-1}}
\lambda_{\nu,(m',0)} \chi_{\tilde{Q}_{\nu, m'}}  
\right\}_{\nu=0}^{\infty}
\right\|_{L^{p(\cdot)}(\ell^{q(\cdot)})}.  
\label{13-6-15-1}
\end{align}
Let $\lambda>0$. By  $\tilde{Q}_{\nu,m'}=Q_{\nu,m'}\times[2^{-\nu}, 2^{-\nu+1})$,  we obtain 
\begin{align}
&\int_{\R^{n-1}}  \left|  \frac{ \left\{\sum_{\nu\in\N_0}\left( 2^{\nu(s(x',0)-\frac{1}{p(x',0)}) } \sum_{m'\in\Z^{n-1}}\lambda_{\nu,(m',0)}\chi_{\nu, m'} \right)^{p(x',0)} \right\}^{\frac{1}{p(x',0)} } } {\lambda}    \right|^{p(x',0)} {\rm d}x'  \notag \\ 
&= \int_{\R^{n-1}} \sum_{\nu\in\N_0} 2^{\nu p(x',0)s(x',0)-\nu}\left| \frac{ \sum_{m'\in\Z^{n-1}} \lambda_{\nu,(m',0)}\chi_{\nu,m'}(x')}{\lambda} \right|^{p(x',0)} {\rm d}x' \notag \\ 
&= \int_{\R^{n-1}} \sum_{\nu\in\N_0} 2^{\nu p(x',0)s(x',0)} 
 \left(  \int_{2^{-\nu}}^{2^{-\nu+1}} \left| \frac{\sum_{m'\in\Z^{n-1}} \lambda_{\nu,(m',0)}\chi_{\nu,m'}(x')}{\lambda} \right|^{p(x',0)}  \chi_{[2^{-\nu}, 2^{-\nu+1}  ) } (x_n) {\rm d}x_n  \right) {\rm d}x' \notag \\ 
&= \int_{\R^{n-1}} \sum_{\nu\in\N_0} 2^{\nu p(x',0)s(x',0)}  \left(  \int_{2^{-\nu}}^{2^{-\nu+1}} \left| \frac{\sum_{m'\in\Z^{n-1}} \lambda_{\nu,(m',0)}\chi_{\nu, m'\times [2^{-\nu}, 2^{-\nu+1}  ) }(x',x_n)}{\lambda} \right|^{p(x',0)}  {\rm d}x_n  \right) {\rm d}x'  \notag \\ 
&= \int_{\R^{n-1}} \sum_{\nu\in\N_0} 2^{\nu p(x',0)s(x',0)} 
 \left(  \int_{2^{-\nu}}^{2^{-\nu+1}} \left| \frac{\sum_{m'\in\Z^{n-1}} \lambda_{\nu,(m',0)}\chi_{\tilde{Q}_{\nu,m'}}(x',x_n)}{\lambda} \right|^{p(x',0)}  {\rm d}x_n  \right) {\rm d}x'. \label{12-11-25-1}
\end{align}
By assumptions for $p(\cdot)$, $q(\cdot)$ and $s(\cdot)$, we see that  
\begin{align}
&\sum_{\nu\in\N_0} 2^{\nu p(x',0)s(x',0)}
\left(  \int_{2^{-\nu}}^{2^{-\nu+1}} \left| \frac{ \sum_{m'\in\Z^{n-1}} \lambda_{\nu,(m',0)}\chi_{\tilde{Q}_{\nu,m'}}(x',x_n)}{\lambda} \right|^{p(x',0)}  {\rm d}x_n  \right) \notag \\
&=\sum_{\nu\in\N_0} 
\left(  \int_{2^{-\nu}}^{2^{-\nu+1}} 2^{\nu p(x',0)s(x',0)}\left|\frac{\sum_{m'\in\Z^{n-1}} \lambda_{\nu,(m',0)}\chi_{\tilde{Q}_{\nu,m'}}(x',x_n)}{\lambda} \right|^{p(x',0)}  {\rm d}x_n  \right)   \notag \\
&=\sum_{\nu\in\N_0} 
\left(  \int_{2^{-\nu}}^{2^{-\nu+1}} 2^{\nu p(x',x_n)s(x',x_n)}\left| \frac{\sum_{m'\in\Z^{n-1}} \lambda_{\nu,(m',0)}\chi_{\tilde{Q}_{\nu,m'}}(x',x_n)}{\lambda} \right|^{p(x',x_n)}  {\rm d}x_n  \right) \notag   \\
&=\sum_{\nu\in\N_0} 
\left(  \int_{-\infty}^{\infty} 2^{\nu p(x',x_n)s(x',x_n)}\left| \frac{\sum_{m'\in\Z^{n-1}} \lambda_{\nu,(m',0)}\chi_{\tilde{Q}_{\nu,m'}}(x',x_n)}{\lambda} \right|^{p(x',x_n)}  {\rm d}x_n  \right)  \notag \\
&=
 \int_{-\infty}^{\infty} \sum_{\nu\in\N_0}  2^{\nu p(x',x_n)s(x',x_n)}\left|
\frac{\sum_{m'\in\Z^{n-1}} \lambda_{\nu,(m',0)}\chi_{\tilde{Q}_{\nu,m'}}(x',x_n)}{\lambda} \right|^{p(x',x_n)}  {\rm d}x_n.  \label{12-11-25-2}
\end{align}
By $(\ref{12-11-25-1})$ and $(\ref{12-11-25-2})$, we obtain 
\begin{align}
&\int_{\R^{n-1}}  \left|  \frac{ \left\{\sum_{\nu\in\N_0}\left( 2^{\nu(s(x',0)-\frac{1}{p(x',0)}) } \sum_{m'\in\Z^{n-1}}\lambda_{\nu,(m',0)}\chi_{\nu, m'} \right)^{p(x',0)} \right\}^{\frac{1}{p(x',0)} } } {\lambda}    \right|^{p(x',0)} {\rm d}x' \notag \\ 
&=\int_{\R^{n-1}} 
\left( \int_{-\infty}^{\infty} \sum_{\nu\in\N_0}  2^{\nu p(x',x_n)s(x',x_n)}  
\left| \frac{\sum_{m'\in\Z^{n-1}} \lambda_{\nu,(m',0)}\chi_{\tilde{Q}_{\nu,m'}}(x',x_n)}{\lambda} \right|^{p(x',x_n)}  {\rm d}x_n\right)  {\rm d}x'  \notag \\ 
&=\int_{\R^n} 
\left(  \sum_{\nu\in\N_0}  2^{\nu p(x)s(x)}  
\left| \frac{\sum_{m'\in\Z^{n-1}} \lambda_{\nu,(m',0)}\chi_{\tilde{Q}_{\nu,m'}}(x)}{\lambda} \right|^{p(x)}  \right)  {\rm d}x \notag \\ 
&=\int_{\R^n} 
\left\{ \left(  \sum_{\nu\in\N_0}  
\left| \frac{2^{\nu s(x)}  \sum_{m'\in\Z^{n-1}} \lambda_{\nu,(m',0)}\chi_{\tilde{Q}_{\nu,m'}}(x)}{\lambda} \right|^{p(x)}  \right)^{\frac{1}{p(x)} } \right\}^{p(x)}  {\rm d}x.  \label{12-11-25-3.1} 
\end{align}
For any $x_n\in\R$, positive integers $\nu$ satisfying $x_n\in [2^{-\nu}, 2^{-\nu+1})$ are at most three. 
This implies that, by $\tilde{Q}_{\nu,m'}=Q_{\nu,m'}\times[2^{-\nu}, 2^{-\nu+1})$,  
\[
 \left(  \sum_{\nu\in\N_0}  
\left| \frac{2^{\nu s(x)}  \sum_{m'\in\Z^{n-1}} \lambda_{\nu,(m',0)}\chi_{\tilde{Q}_{\nu,m'}}(x)}{\lambda} \right|^{p(x)}  \right)^{\frac{1}{p(x)} }
\]
consists of at most three non-zero members for any $x\in\R^n$. Hence we have 
\begin{align*}
&\left(  \sum_{\nu\in\N_0}   
\left| \frac{2^{\nu s(x)}  \sum_{m'\in\Z^{n-1}} \lambda_{\nu,(m',0)}\chi_{\tilde{Q}_{\nu,m'}}(x)}{\lambda} \right|^{p(x)}  \right)^{\frac{1}{p(x)} } \\ 
&\sim 
\left(  \sum_{\nu\in\N_0}  
\left| \frac{2^{\nu s(x)}  \sum_{m'\in\Z^{n-1}} \lambda_{\nu,(m',0)}\chi_{\tilde{Q}_{\nu,m'}}(x)}{\lambda} \right|^{q(x)}  \right)^{\frac{1}{q(x)} }. 
\end{align*}
By $(\ref{12-11-25-3.1} )$ and the equation as above, we see that 
\begin{align}
&\int_{\R^{n-1}}  \left|  \frac{ \left\{\sum_{\nu\in\N_0}\left( 2^{\nu(s(x',0)-\frac{1}{p(x',0)}) } \sum_{m'\in\Z^{n-1}}\lambda_{\nu,(m',0)}\chi_{\nu, m'} \right)^{p(x',0)} \right\}^{\frac{1}{p(x',0)} } } {\lambda}    \right|^{p(x',0)} {\rm d}x' \notag \\ 
&=\int_{\R^n} 
\left\{ \left(  \sum_{\nu\in\N_0}  
\left| \frac{2^{\nu s(x)}  \sum_{m'\in\Z^{n-1}} \lambda_{\nu,(m',0)}\chi_{\tilde{Q}_{\nu,m'}}(x)}{\lambda} \right|^{p(x)}  \right)^{\frac{1}{p(x)} } \right\}^{p(x)}  {\rm d}x \notag \\ 
&\sim 
  \int_{\R^n} 
\left\{  
\left(  \sum_{\nu\in\N_0}  
\left| \frac{2^{\nu s(x)}  \sum_{m'\in\Z^{n-1}} \lambda_{\nu,(m',0)}\chi_{\tilde{Q}_{\nu,m'}}(x)}{\lambda} \right|^{q(x)}  \right)^{\frac{1}{q(x)} }
 \right\}^{p(x)}  {\rm d}x \label{12-11-25-3}
\end{align}
holds for any $\lambda>0$. 
This implies that $(\ref{13-6-15-1})$ holds. 

Finally, we prove 
\begin{align}
&\left\| \left\{ 2^{\nu s(\cdot)} 
\sum_{m'\in\Z^{n-1}}
\lambda_{\nu,(m',0)} \chi_{\tilde{Q}_{\nu, m'}}  
\right\}_{\nu=0}^{\infty}
\right\|_{L^{p(\cdot)}(\ell^{q(\cdot)})} \notag \\ 
&\sim 
\left\| \left\{ 2^{\nu s(\cdot)} 
\sum_{m'\in\Z^{n-1}}
\lambda_{\nu,(m',0)} \chi_{\nu, m'}  
\right\}_{\nu=0}^{\infty}
\right\|_{L^{p(\cdot)}(\ell^{q(\cdot)})}.  
\label{13-6-15-2}
\end{align} 
Without loss of generality, we can assume
$\left\| 
\left\{
\delta_{m_n,0}\lambda_{\nu,m}
\right\}_{(\nu,m)\in\N_0\times\Z^n}
\right\|_{f_{p(\cdot),q(\cdot)}^{s(\cdot)}(\R^{n})}=1$ and assume $\lambda_{\nu,m}=0$ when $m_n\neq 0$. 
Let $\alpha:=\frac{\min(p^-,q^-,1)}{2}$ and let 
\begin{equation}
\lambda':=\{\lambda_{\nu,(m',0)}\}_{(\nu,m')\in\N_0\times\Z^{n-1}}.  \label{5.154}
\end{equation} 
Let $\tilde{Q}_{\nu,m'}=Q_{\nu,m'}\times[2^{-\nu}, 2^{-\nu+1})$ for $\nu\in\N_0$ and $m'\in\Z^{n-1}$. 
If $x\in\tilde{Q}_{\nu,m'}$ and $y\in Q_{\nu,(m',0)}$, then 
we that $|x-y|\le 2\sqrt{n}2^{-\nu}$. 
Therefore, we obtain 
\[
1\le\left(\frac{2\sqrt{n}+1}{1+2^{\nu}|x-y|}\right)^M
\]
for any $M>0$. Let $M>2\max(C_{\log}(s),n)$. 
Hence, we see that 
\begin{align*}
|\lambda_{\nu,m'}|^{\alpha}\chi_{\tilde{Q}_{\nu,m'}}(x)
&\le \frac{\chi_{\tilde{Q}_{\nu,m'} }(x) }{|Q_{\nu,(m',0)}|}\int_{Q_{\nu, (m',0)}}  | \lambda_{\nu,m'}|^{\alpha}\chi_{\nu,(m',0)}(y){\rm d}y \\ 
&\lesssim \int_{\R^n} \frac{2^{\nu n}}{(1+2^{\nu}|x-y|)^M}\left( |\lambda_{\nu,m'}|^{\alpha}\chi_{\nu,(m',0)}(y)\right){\rm d}y \\ 
&\lesssim \int_{\R^n} \frac{2^{\nu n}}{(1+2^{\nu}|x-y|)^M}\left(\sum_{m'\in\Z^{n-1}}|\lambda_{\nu,m'}|^{\alpha}\chi_{\nu,(m',0)}(y)\right){\rm d}y \\ 
&= \left(\eta_{\nu,M}\ast \left(\sum_{m'\in\Z^{n-1}}|\lambda_{\nu,m'}|\chi_{\nu,(m',0)}(\cdot)\right)^{\alpha}  \right)(x). 
\end{align*}
By Lemma $\ref{lemma:2}$, we obtain  
\begin{equation}
\left| 2^{\nu s(x)}\sum_{m'\in\Z^{n-1}}\lambda_{\nu,m'}\chi_{\tilde{Q}_{\nu,m'}(x)} \right| 
\lesssim \left(\eta_{\nu,\frac{M}{2}}\ast \left(2^{\nu s(\cdot)}\sum_{m'\in\Z^{n-1}}|\lambda_{\nu,m'}|\chi_{\nu,(m',0)}(\cdot)\right)^{\alpha}  \right)^{\frac{1}{\alpha}}(x). 
\label{5.156}
\end{equation}

By $(\ref{5.156})$ and Theorem $\ref{thm:max2}$, we see that 
\begin{align*}
&\int_{\R^n} \left(\sum_{\nu\in\N_0} \left|2^{\nu s(x)}\sum_{m'\in\Z^{n-1}} \lambda_{\nu,(m',0)}\chi_{\tilde{Q}_{\nu,m'}(x)}\right|^{q(x)}\right)^{\frac{p(x)}{q(x)}} {\rm d}x \\ 
&\lesssim_{p,q} 
\int_{\R^n} \left(\sum_{\nu\in\N_0} 
\left|\left(\eta_{\nu,\frac{M}{2}}\ast \left(2^{\nu s(\cdot)}\sum_{m'\in\Z^{n-1}}|\lambda_{\nu,m'}|\chi_{\nu,(m',0)}(\cdot)\right)^{\alpha}  
\right)^{\frac{1}{\alpha}}(x) \right|^{q(x)}\right)^{\frac{p(x)}{q(x)}} {\rm d}x \\ 
&\lesssim 1. 
\end{align*}
This implies that 
\begin{align}
&\left\| \left\{ 2^{\nu s(\cdot)} 
\sum_{m'\in\Z^{n-1}}
\lambda_{\nu,(m',0)} \chi_{\tilde{Q}_{\nu, m'}}  
\right\}_{\nu=0}^{\infty}
\right\|_{L^{p(\cdot)}(\ell^{q(\cdot)})} \notag \\ 
&\lesssim  
\left\| \left\{ 2^{\nu s(\cdot)} 
\sum_{m'\in\Z^{n-1}}
\lambda_{\nu,(m',0)} \chi_{\nu, m'}  
\right\}_{\nu=0}^{\infty}
\right\|_{L^{p(\cdot)}(\ell^{q(\cdot)})}. \label{2013-6-16-1}  
\end{align} 

On the other hand, without loss of generality, we can assume that 
\[
\left\| \left\{ 2^{\nu s(\cdot)} 
\sum_{m'\in\Z^{n-1}}
\lambda_{\nu,(m',0)} \chi_{\nu, m'}  
\right\}_{\nu=0}^{\infty}
\right\|_{L^{p(\cdot)}(\ell^{q(\cdot)})} =1. 
\]
By using same argument, we have 
\begin{align*}
|\lambda_{\nu,m'}|^{\alpha} \chi_{\nu,(m',0)}(x)
&\le \frac{\chi_{\nu,(m',0)}(x) }{ |\tilde{Q}_{\nu,m'}|}  \int_{ \tilde{Q}_{\nu,m'} }  | \lambda_{\nu,m'}|^{\alpha} \chi_{\tilde{Q}_{\nu,m'}}(y){\rm d}y \\ 
&\lesssim \int_{\R^n} \frac{2^{\nu n}}{(1+2^{\nu}|x-y|)^M}\left( |\lambda_{\nu,m'}|^{\alpha} \chi_{\tilde{Q}_{\nu,m'}}(y)\right){\rm d}y \\ 
&\lesssim \int_{\R^n} \frac{2^{\nu n}}{(1+2^{\nu}|x-y|)^M}\left(\sum_{m'\in\Z^{n-1}}|\lambda_{\nu,m'}|^{\alpha} \chi_{\tilde{Q}_{\nu,m'}}(y)\right){\rm d}y \\ 
&= \left(\eta_{\nu,M}\ast \left(\sum_{m'\in\Z^{n-1}}|\lambda_{\nu,m'}| \chi_{\tilde{Q}_{\nu,m'}}(\cdot)\right)^{\alpha}  \right)(x). 
\end{align*}
Hence we see that 
\begin{equation}
\left| 2^{\nu s(x)}\sum_{m'\in\Z^{n-1}}\lambda_{\nu,m'}  \chi_{\nu,(m',0)}(x) \right| 
\lesssim \left(\eta_{\nu,\frac{M}{2}}\ast \left(2^{\nu s(\cdot)}\sum_{m'\in\Z^{n-1}}|\lambda_{\nu,m'}|  \chi_{\tilde{Q}_{\nu,m'}}(\cdot)\right)^{\alpha}  \right)^{\frac{1}{\alpha}}(x) 
\label{5.157}
\end{equation}
By $(\ref{5.157})$ and Theorem $\ref{thm:max2}$, we see that 
\begin{align*}
&\int_{\R^n} \left(\sum_{\nu\in\N_0} \left|2^{\nu s(x)}\sum_{m'\in\Z^{n-1}} \lambda_{\nu,(m',0)}   \chi_{\nu,(m',0)}(x) \right|^{q(x)}\right)^{\frac{p(x)}{q(x)}} {\rm d}x \\ 
&\lesssim_{p,q} 
\int_{\R^n} \left(\sum_{\nu\in\N_0} 
\left|\left(\eta_{\nu,\frac{M}{2}}\ast \left(2^{\nu s(\cdot)}\sum_{m'\in\Z^{n-1}}|\lambda_{\nu,m'}|  \chi_{\tilde{Q}_{\nu,m'}}(\cdot)\right)^{\alpha}  
\right)^{\frac{1}{\alpha}}(x) \right|^{q(x)}\right)^{\frac{p(x)}{q(x)}} {\rm d}x \\ 
&\lesssim 1. 
\end{align*}
This implies that 
\begin{align}
&\left\| \left\{ 2^{\nu s(\cdot)} 
\sum_{m'\in\Z^{n-1}}
\lambda_{\nu,(m',0)} \chi_{\tilde{Q}_{\nu, m'}}  
\right\}_{\nu=0}^{\infty}
\right\|_{L^{p(\cdot)}(\ell^{q(\cdot)})}  \notag \\ 
&\gtrsim  
\left\| \left\{ 2^{\nu s(\cdot)} 
\sum_{m'\in\Z^{n-1}}
\lambda_{\nu,(m',0)} \chi_{\nu, m'}  
\right\}_{\nu=0}^{\infty}
\right\|_{L^{p(\cdot)}(\ell^{q(\cdot)})}. \label{2013-6-16-2}  
\end{align} 

By $(\ref{2013-6-16-1}  )$ and $(\ref{2013-6-16-2}  )$, we have $(\ref{13-6-15-2})$. 
Therefore, Lemma $\ref{lemma 5.4.5}$ holds by $(\ref{13-6-15-1})$ and $(\ref{13-6-15-2})$. 
\end{proof}

\begin{lemma}
\label{lemma-saikin}
Let $p(\cdot),q(\cdot)\in C^{\log}(\R^n)\cap\mathcal{P}_0(\R^n)$. 
For double-index complex-valued sequence $\lambda=\{\lambda_{\nu,m}\}_{(\nu,m)\in\N_0\times\Z^n}$, 
we have 
\[
\left\| 
\left\{
\lambda_{\nu,(m',0)}
\right\}_{(\nu,m')\in\N_0\times\Z^{n-1}}
\right\|_{b_{\tilde{p}(\cdot),\tilde{q}(\cdot)}^{\tilde{s}(\cdot)-\frac{1}{\tilde{p}(\cdot)}}(\R^{n-1})}  
\sim 
\left\| 
\left\{
\delta_{m_n,0}\lambda_{\nu,m}
\right\}_{(\nu,m)\in\N_0\times\Z^n}
\right\|_{b_{p(\cdot),q(\cdot)}^{s(\cdot)}(\R^{n})}. 
\]
\end{lemma}

\begin{proof}
By Corollary $\ref{cor proposition 7.3}$, it suffices to consider the case $p(\cdot)$, $q(\cdot)$ and $s(\cdot)$ independent of the $n$-th coordinate for $|x_n|\le 2$. 
Let $\tilde{Q}_{\nu,m'}=Q_{\nu,m'}\times[2^{-\nu}, 2^{-\nu+1})$ for $\nu\in\N_0$ and $m'\in\Z^{n-1}$. Firstly, we prove 
\begin{align}
&\left\| \left\{ 2^{\nu(\tilde{s}(\cdot)-\frac{1}{\tilde{p}(\cdot)})} \sum_{m'\in\Z^{n-1}}\lambda_{\nu,(m',0)}\chi_{\nu, m'}\right\}_{\nu=0}^{\infty} \right\|_{\ell^{\tilde{q}(\cdot)}(L^{\tilde{p}(\cdot)})}  \notag \\ 
&\sim 
\left\| \left\{ 2^{\nu s(\cdot)} 
\sum_{m'\in\Z^{n-1}}
\lambda_{\nu,(m',0)} \chi_{\tilde{Q}_{\nu, m'}}  
\right\}_{\nu=0}^{\infty}
\right\|_{\ell^{q(\cdot)}(L^{p(\cdot)})}.  
\label{13-6-4-1}
\end{align}
Let $\lambda>0$ and $\mu>0$. 
Then we recall that 
\begin{align*}
&\left\| \left\{ 2^{\nu(\tilde{s}(\cdot)-\frac{1}{\tilde{p}(\cdot)})} \sum_{m'\in\Z^{n-1}}\lambda_{\nu,(m',0)}\chi_{\nu, m'}\right\}_{\nu=0}^{\infty} \right\|_{\ell^{\tilde{q}(\cdot)}(L^{\tilde{p}(\cdot)})} \\ 
&=
\inf\left\{ 
\lambda>0\,:\, \sum_{\nu=0}^{\infty}
\left\| 
\left(
\frac{
 2^{\nu(\tilde{s}(\cdot)-\frac{1}{\tilde{p}(\cdot)})} \sum_{m'\in\Z^{n-1}}\lambda_{\nu,(m',0)}\chi_{\nu, m'}
}
{\lambda}
\right)^{\tilde{q}(\cdot)}
\right\|_{L^{\frac{\tilde{p}(\cdot)}{\tilde{q}(\cdot)}}}
 \right\}
\end{align*}
and 
\begin{align*}
&\left\| 
\left(
\frac{
 2^{\nu(\tilde{s}(\cdot)-\frac{1}{\tilde{p}(\cdot)})} \sum_{m'\in\Z^{n-1}}\lambda_{\nu,(m',0)}\chi_{\nu, m'}
}
{\lambda}
\right)^{\tilde{q}(\cdot)}
\right\|_{L^{\frac{\tilde{p}(\cdot)}{\tilde{q}(\cdot)}}} \\ 
&= 
\inf 
\left\{
\mu>0\,:\, 
\int_{\R^{n-1}}
\left\{
\frac{
\left(
 2^{\nu(s(x',0)-\frac{1}{p(x',0)})} \sum_{m'\in\Z^{n-1}}
\frac{\lambda_{\nu,(m',0)}}{\lambda}\chi_{\nu, m'}
\right)^{q(x',0)}
}
{\mu}
\right\}^{\frac{p(x',0)}{q(x',0)} } 
{\rm d}x'
\right\}
\end{align*}
for each $\nu\in \N_0$. 
By $\displaystyle 2^{-\nu}=\int_{2^{-\nu}}^{2^{-\nu+1}}\chi_{[2^{-\nu}, 2^{-\nu+1})}\,{\rm d}x $, 
we see that
\begin{align*}
&\int_{\R^{n-1}}
\left\{
\frac{
\left(
 2^{\nu(s(x',0)-\frac{1}{p(x',0)})} \sum_{m'\in\Z^{n-1}}
\frac{\lambda_{\nu,(m',0)}}{\lambda}\chi_{\nu, m'}
\right)^{q(x',0)}
}
{\mu}
\right\}^{\frac{p(x',0)}{q(x',0)} } 
{\rm d}x' \\ 
&= 
\int_{\R^{n-1}}
\frac{
\left(
 2^{\nu(s(x',0)-\frac{1}{p(x',0)})} \sum_{m'\in\Z^{n-1}}
\frac{\lambda_{\nu,(m',0)}}{\lambda}\chi_{\nu, m'}
\right)^{p(x',0)}
}
{\mu^{\frac{p(x',0)}{q(x',0)}}}
{\rm d}x' \\ 
&= 
\int_{\R^{n-1}}
\frac{
 2^{\nu(s(x',0)q(x',0)-1)}
\left( \sum_{m'\in\Z^{n-1}}
\frac{\lambda_{\nu,(m',0)}}{\lambda}\chi_{\nu, m'}
\right)^{p(x',0)}
}
{\mu^{\frac{p(x',0)}{q(x',0)}}}
{\rm d}x' \\ 
&= 
\int_{\R^{n-1}}
\left\{ 
\int_{2^{-\nu}}^{2^{-\nu+1}}
\frac{
 2^{\nu s(x',0)q(x',0)}
\left( \sum_{m'\in\Z^{n-1}}
\frac{\lambda_{\nu,(m',0)}}{\lambda}\chi_{\nu, m'}
\right)^{p(x',0)} \chi_{[2^{-\nu}, 2^{-\nu+1})}(x_n)
}
{\mu^{\frac{p(x',0)}{q(x',0)}}}\,{\rm d}x_n
\right\}
{\rm d}x' \\ 
&= 
\int_{\R^{n-1}}
\left\{ 
\int_{2^{-\nu}}^{2^{-\nu+1}}
\frac{
 2^{\nu s(x',x_n)q(x',x_n)}
\left( \sum_{m'\in\Z^{n-1}}
\frac{\lambda_{\nu,(m',0)}}{\lambda}  \chi_{\tilde{Q}_{\nu, m'}} 
\right)^{p(x',x_n)} 
}
{\mu^{\frac{p(x',x_n)}{q(x',x_n)}}} \,{\rm d}x_n
\right\}
{\rm d}x' \\ 
&= 
\int_{\R^{n}}
\frac{
 2^{\nu s(x)q(x)}
\left( \sum_{m'\in\Z^{n-1}}
\frac{\lambda_{\nu,(m',0)}}{\lambda}  \chi_{\tilde{Q}_{\nu, m'}} 
\right)^{p(x)} 
}
{\mu^{\frac{p(x)}{q(x)}}} 
{\rm d}x\\ 
&= 
\int_{\R^{n}}
\left(
\frac{
 2^{\nu s(x)q(x)}
\left( \sum_{m'\in\Z^{n-1}}
\frac{\lambda_{\nu,(m',0)}}{\lambda}  \chi_{\tilde{Q}_{\nu, m'}} 
\right)^{q(x)} 
}
{\mu }
\right)^{\frac{p(x)}{q(x)}} 
{\rm d}x
\end{align*}
holds for any $\nu\in\N_0$. 
This implies that 
\[
\left\| 
\left(
\frac{
 2^{\nu(\tilde{s}(\cdot)-\frac{1}{\tilde{p}(\cdot)})} \sum_{m'\in\Z^{n-1}}\lambda_{\nu,(m',0)}\chi_{\nu, m'}
}
{\lambda}
\right)^{\tilde{q}(\cdot)}
\right\|_{L^{\frac{\tilde{p}(\cdot)}{\tilde{q}(\cdot)}}}
= 
\left\| 
\left(
\frac{
 2^{\nu s(\cdot)} \sum_{m'\in\Z^{n-1}}
\lambda_{\nu,(m',0) }\chi_{\tilde{Q}_{\nu, m'}} 
}
{\lambda}
\right)^{q(\cdot)} 
\right\|_{L^{\frac{p(\cdot)}{q(\cdot)}}}
\]
holds for any $\nu\in \N_0$. 
Therefore, we obtain $(\ref{13-6-4-1})$. 

Secondly, we prove 
\begin{align}
&\left\| \left\{ 2^{\nu s(\cdot)} 
\sum_{m'\in\Z^{n-1}}
\lambda_{\nu,(m',0)} \chi_{\tilde{Q}_{\nu, m'}}  
\right\}_{\nu=0}^{\infty}
\right\|_{\ell^{q(\cdot)}(L^{p(\cdot)})}  \notag \\ 
&\sim
\left\| \left\{ 2^{\nu s(\cdot)} 
\sum_{m'\in\Z^{n-1}}
\lambda_{\nu,(m',0)} \chi_{\nu, m'}  
\right\}_{\nu=0}^{\infty}
\right\|_{\ell^{q(\cdot)}(L^{p(\cdot)})}. \label{13-6-4-2}
\end{align}
By using the same argument of the proof of Lemma $\ref{lemma 5.4.5}$, 
we have $(\ref{5.156})$ and $(\ref{5.157})$. 
Hence we obtain $(\ref{13-6-4-2})$.  

Therefore, Lemma $\ref{lemma-saikin}$ holds by $(\ref{13-6-4-1})$ and $(\ref{13-6-4-2})$.

\end{proof}

Now, we can prove Theorem $\ref{theorem 5.4.4}$. 

\begin{proof}[Proof of Theorem $\ref{theorem 5.4.4}$]
Firstly, we prove that $\tr$ is well-defined. 
We apply Theorem $\ref{quark decomposition}$ as $f\in A_{p(\cdot),q(\cdot)}^{s(\cdot)}(\R^n)$, 
we have 
\[
f=\sum_{\beta\in\N_0^n}\sum_{\nu\in\N_0}\sum_{m\in\Z^n}\lambda_{\nu,m}^{\beta}(\beta\qu)_{\nu,m}. 
\]
Here we can take $\psi$ in Definition $\ref{def quark}$ such that $\psi(x)=\mu(x_1)\mu(x_2)\cdots\mu(x_n)$, where $\mu$ is a 1-dimensional smooth function 
such that $\spt(\mu)\subset (-1,1)$.   
By using the quarkonial decomposition, we extend $\tr$ so that 
\[
\tr f=\sum_{\beta\in\N_0^n}\sum_{\nu\in\N_0}\sum_{m\in\Z^n}\lambda_{\nu,m}^{\beta}(\beta\qu)_{\nu,m}(\cdot,0). 
\]
By the support of $\mu$ and the definition of quarks, 
$\beta_n$ and $m_n$ are $0$. 
Theorefore, we have 
\begin{equation}
\tr f=\sum_{\beta\in\N_0^n,\beta_n=0}\sum_{\nu\in\N_0}\sum_{m\in\Z^n, m_n=0}\lambda_{\nu,m}^{\beta}(\beta\qu)_{\nu,m}(\cdot,0). \label{5.158}
\end{equation} 
Let $\lambda^{\beta'}=\{\lambda_{\nu,(m',0)}^{(\beta',0)}\}$. 
Then we see that extended $\tr$ 
is well-defined because that the convergence of $(\ref{5.158})$ is uniformly if $f\in\mathcal{S}(\R^n)$ 
and that 
\[
||\lambda^{\beta'}||_{f_{\tilde{p}(\cdot),\tilde{p}(\cdot)}^{\tilde{s}(\cdot)-\frac{1}{\tilde{p}(\cdot)}}(\R^{n-1})} 
\lesssim ||\lambda^{\beta}||_{f_{p(\cdot),q(\cdot)}^{s(\cdot)}(\R^n)}
\] and 
\[
||\lambda^{\beta'}||_{b_{\tilde{p}(\cdot),\tilde{q}(\cdot)}^{\tilde{s}(\cdot)-\frac{1}{\tilde{p}(\cdot)}}(\R^{n-1})} 
\lesssim ||\lambda^{\beta}||_{b_{p(\cdot),q(\cdot)}^{s(\cdot)}(\R^n)}
\] hold by Lemma $\ref{lemma 5.4.5}$. 

Since we assume $(\ref{rho quark})$, 
we can take $\epsilon>0$ such that 
$0<\epsilon<\rho-r$. 
For any $\nu\in\N_0$ and $m\in\Z^n$, there exists $d>0$ such that 
$\spt (\beta\qu)_{\nu,m}\subset dQ_{\nu,m}$. 
The condition 
$(\ref{eq:10-6-1})$ 
implies that smooth atoms in $F_{\tilde{p}(\cdot),\tilde{p}(\cdot)}^{\tilde{s}(\cdot)-\frac{1}{\tilde{p}(\cdot)} }(\R^{n-1})$ 
are not required to satisfy any moment conditions 
and that smooth atoms in $B_{\tilde{p}(\cdot),\tilde{q}(\cdot)}^{\tilde{s}(\cdot)-\frac{1}{\tilde{p}(\cdot)} }(\R^{n-1})$ are not required to satisfy any moment conditions . 
Hence we can regard $2^{-(r+\epsilon)|\beta|}(\beta\qu)_{\nu,m}$ as smooth atoms in  $F_{\tilde{p}(\cdot),\tilde{p}(\cdot)}^{\tilde{s}(\cdot)-\frac{1}{\tilde{p}(\cdot)} }(\R^{n-1})$ 
and $B_{\tilde{p}(\cdot),\tilde{q}(\cdot)}^{\tilde{s}(\cdot)-\frac{1}{\tilde{p}(\cdot)} }(\R^{n-1})$.  

Let $\alpha=\min(p^-,q^-,1)$. Then we see that 
\begin{align*}
\left\| \tr f \right\|_{F_{\tilde{p}(\cdot),\tilde{p}(\cdot)}^{\tilde{s}(\cdot)-\frac{1}{\tilde{p}(\cdot)}}}^{\alpha}  
&\le \sum_{\beta\in\N_0^n, \beta_n=0} \left\| \sum_{\nu\in\N_0}\sum_{m\in\Z^n, m_n=0}\lambda_{\nu,m}^{\beta}(\beta\qu)_{\nu,m}(\cdot,0) \right\|_{F_{\tilde{p}(\cdot),\tilde{p}(\cdot)}^{\tilde{s}(\cdot)-\frac{1}{\tilde{p}(\cdot)}}}^{\alpha} \\ 
&\le  \sum_{\beta'}\left\| 
\lambda^{\beta'} 
\right\|_{f_{\tilde{p}(\cdot),\tilde{p}(\cdot)}^{\tilde{s}(\cdot)-\frac{1}{\tilde{p}(\cdot)}}}^{\alpha}  \\ 
&\lesssim  \sum_{\beta\in\N_0} 
\left\| 
\lambda^{\beta}
\right\|_{f_{p(\cdot),q(\cdot)}^{s(\cdot)}}^{\alpha}  \\ 
&\lesssim  \sum_{\beta\in\N_0} 2^{-\rho\alpha |\beta|}
\left\| 
\lambda
\right\|_{f_{p(\cdot),q(\cdot),\rho}^{s(\cdot)}}^{\alpha} \\ 
&\lesssim  \sum_{\beta\in\N_0} 2^{-\rho\alpha |\beta|}
\left\| 
\lambda
\right\|_{f_{p(\cdot),q(\cdot),\rho}^{s(\cdot)}}^{\alpha} \\ 
&\lesssim \left\| f \right\|_{F_{p(\cdot),q(\cdot)}^{s(\cdot)}}^{\alpha}
\end{align*}
by Lemma $\ref{min triangle}$. 
Therefore, we have $\displaystyle \left\| \tr f \right\|_{F_{\tilde{p}(\cdot),\tilde{p}(\cdot)}^{\tilde{s}(\cdot)-\frac{1}{\tilde{p}(\cdot)}}} \lesssim \left\| f \right\|_{F_{p(\cdot),q(\cdot)}^{s(\cdot)}} $. 
By using same calculation, we also have 
 $\displaystyle \left\| \tr f \right\|_{B_{\tilde{p}(\cdot),\tilde{q}(\cdot)}^{\tilde{s}(\cdot)-\frac{1}{\tilde{p}(\cdot)}}} \lesssim \left\| f \right\|_{B_{p(\cdot),q(\cdot)}^{s(\cdot)}} $. 
Surjective follows from second assertion with $k=0$.

Finally, we prove the second assertion for Triebel--Lizorkin spaces. 
We fix $j=0,1,2,\cdots,k$. Let $g_j\in F_{\tilde{p}(\cdot),\tilde{p}(\cdot)}^{\tilde{s}(\cdot)-j-\frac{1}{\tilde{p}(\cdot)}}(\R^{n-1})$. 
We can write 
\begin{equation}
g_j= \sum_{\beta'\in\N_0^{n-1}}\sum_{\nu\in\N_0}\sum_{m'\in\Z^{n-1}}\lambda^{\beta'}_{\nu,m'}(\beta'\qu)_{\nu,m'} \label{5.159}
\end{equation} 
by Theorem $\ref{quark decomposition}$. 
Let $L\gg 1$ and  
\[
v_j=\sum_{\beta'\in\N_0^{n-1}}\sum_{\nu\in\N_0}\sum_{m'\in\Z^{n-1}} 
\frac{\lambda^{\beta'}_{\nu,m'}}{(2L+j)!2^{\nu(2L+j)}} \left((\beta',2L+j) \qu\right)_{\nu,(m',0)} 
\]
Then we see that $v_j\in F_{p(\cdot),q(\cdot)}^{s(\cdot)+2L}(\R^n)$ by Theorem $\ref{quark decomposition}$ and Lemma $\ref{lemma 5.4.5}$. 
We define $h_j=\partial_{x_n}^{2L} v_j\in F_{p(\cdot),q(\cdot)}^{s(\cdot)}(\R^n)$ such that 
\begin{equation}
h_j :=\sum_{\beta'\in\N_0^{n-1}}\sum_{\nu\in\N_0}\sum_{m'\in\Z^{n-1}} 
\frac{\lambda^{\beta'}_{\nu,m'}}{(2L+j)!} \partial_{x_n}^{2L} [ x_n^{2L+j} \left((\beta',0) \qu\right)_{\nu,(m',0)} ]. \label{5.160}
\end{equation}
Then we see that the right hand side of 
\begin{equation}
\partial_{x_n}^l h_j =\frac{1}{(2L+j)!} \sum_{\beta'\in\N_0^{n-1}}\sum_{\nu\in\N_0}\sum_{m'\in\Z^{n-1}} 
\lambda^{\beta'}_{\nu,m'} \partial_{x_n}^{2L+l} [ x_n^{2L+j} \left((\beta',0) \qu\right)_{\nu,(m',0)} ] \label{5.161}
\end{equation}
converges in the sense of $ F_{p(\cdot),q(\cdot)}^{s(\cdot)-l}(\R^n)$ for any $l=0,1,\cdots, k$. 
Now we consider the 
\begin{equation}
\tr \left[
\partial_{x_n}^l h_j \right]=\frac{1}{(2L+j)!} \tr\left[ \sum_{\beta'\in\N_0^{n-1}}\sum_{\nu\in\N_0}\sum_{m'\in\Z^{n-1}} 
\lambda^{\beta'}_{\nu,m'} \partial_{x_n}^{2L+l} [ x_n^{2L+j} \left((\beta',0) \qu\right)_{\nu,(m',0)} ] \right]. \label{2014-2-9-1}
\end{equation}
We can change the order $\tr$ and summation in $(\ref{2014-2-9-1})$ because 
$\tr$ is continuous on $F_{p(\cdot),q(\cdot)}^{s(\cdot)-l}(\R^n)$. 
Hence we have 
\begin{equation}
\tr \left[
\partial_{x_n}^l h_j \right]=\frac{1}{(2L+j)!} \sum_{\beta'\in\N_0^{n-1}}\sum_{\nu\in\N_0}\sum_{m'\in\Z^{n-1}} 
\lambda^{\beta'}_{\nu,m'} \tr \left[ \partial_{x_n}^{2L+l} [ x_n^{2L+j} \left((\beta',0) \qu\right)_{\nu,(m',0)} ]\right]. \label{2014-2-9-2}
\end{equation}
Let 
\[
\delta_{j,l} =\begin{cases} 1 \ \ &j=l, \\ 
                                     0 \ \ &j\neq l. 
                                     \end{cases}
\]
Recall that the definition of $\psi(x)$ and quarks, we have 
$\displaystyle \sum_{m\in\Z^n}\mu(x_n-m)=1$. 
Furthermore, by the support of $\mu$, it is easy to see that $\displaystyle \mu^{(k)}(0)=0$ for any $k\in\N$. 
Hence we have 
\begin{equation}
\tr\left[ 
\partial_{x_n}^{2L+l}[ x_n^{2L+j} \left( (\beta',0)\qu \right)_{\nu,m'}(x') ] 
 \right]
 =\delta_{j,l}(2L+j)!\left( (\beta',0)\qu \right)_{\nu,m'}(x'). \label{5.162}
\end{equation}
Therefore, we obtain 
\begin{align*}
\tr[\partial_{x_n}^l h_j] &= 
\frac{1}{(2L+j)!} \sum_{\beta'\in\N_0^{n-1}}\sum_{\nu\in\N_0}\sum_{m'\in\Z^{n-1}} 
\lambda^{\beta'}_{\nu,m'} \tr \partial_{x_n}^{2L+l} [ x_n^{2L+j} \left((\beta',0) \qu\right)_{\nu,(m',0)} ]  \\ 
&= \delta_{j,l}g_j 
\end{align*} 
by $(\ref{2014-2-9-2})$ and $(\ref{5.162})$. 
This implies  that $f=\sum_{j=0}^{k}h_j$ satisfies the second assertion of Triebel--Lizorkin case. 
We can also prove the Besov case by using similar argument as above. 
\end{proof}

\section{Trace theorem for upper half space $\R^n_+$}

We will extend Theorem $\ref{theorem 5.4.4}$ to Besov spaces and Triebel--Lizorkn spaces with variable exponents on 
 $\R^n_+=\{ (x',x_n)\in\R^n\,:\,x_n>0 \}$. 
To do this, we define spaces $B_{p(\cdot),q(\cdot)}^{s(\cdot)}(\R^n_+)$ and $F_{p(\cdot),q(\cdot)}^{s(\cdot)}(\R^n_+)$. 
\begin{definition}
\label{definition 6.1.1}
Let $p(\cdot),q(\cdot)\in C^{\log}(\R^n)\cap\mathcal{P}_0(\R^n)$ and $s(\cdot)\in C^{\log}(\R^n)$. 

Besov spaces with variable exponents on upper half plane $B_{p(\cdot),q(\cdot)}^{s(\cdot)}(\jissu^n_+)$ is the collection of $f\in\mathcal{D}'(\jissu^n_+)$ such that  
there exists a $g\in  B_{p(\cdot),q(\cdot)}^{s(\cdot)}(\R^n)$ satisfying $f=g|_{\R_+}$, where 
$g|_{\R^n_+}$ denotes the restriction of $g\in \mathcal{S}'(\R^n)$ to $\R^n_+$ in the sense of $\mathcal{D}'(\R^n_+)$. 
The space $B_{p(\cdot),q(\cdot)}^{s(\cdot)}(\jissu^n_+)$ becomes a normed space equipped with the norm  
\[
||f||_{B_{p(\cdot),q(\cdot)}^{s(\cdot)}(\R^n_+)} :=  
\inf \left\{ ||g||_{B_{p(\cdot),q(\cdot)}^{s(\cdot)}}\,:\, g\in B_{p(\cdot),q(\cdot)}^{s(\cdot)}(\R^n),\,f=g|_{\R^n_+}\right\}. 
\]

Triebel--Lizorkin spaces with variable exponents on upper half plane $F_{p(\cdot),q(\cdot)}^{s(\cdot)}(\jissu^n_+)$ is the collection of $f\in\mathcal{D}'(\jissu^n_+)$ such that  
there exists a $g\in  F_{p(\cdot),q(\cdot)}^{s(\cdot)}(\R^n)$ satisfying $f=g|_{\R_+}$. The space $F_{p(\cdot),q(\cdot)}^{s(\cdot)}(\jissu^n_+)$ becomes a normed space equipped with 
the norm 
\[
||f||_{F_{p(\cdot),q(\cdot)}^{s(\cdot)}(\R^n_+)} :=  
\inf \left\{ ||g||_{F_{p(\cdot),q(\cdot)}^{s(\cdot)}}\,:\, g\in F_{p(\cdot),q(\cdot)}^{s(\cdot)}(\R^n),\,f=g|_{\R^n_+}\right\}. 
\]

\end{definition}
Then we have following theorem as well as $A_{p(\cdot),q(\cdot)}^{s(\cdot)}(\R^n)$ cases. 

\begin{theorem}
\label{upper half space completeness}
Let $p(\cdot), q(\cdot)\in C^{\log}(\R^n)\cap \mathcal{P}_0(\R^n)$ and $s(\cdot)\in C^{\log}(\R^n)$. 
Then 

{\rm (1)} $A_{p(\cdot),q(\cdot)}^{s(\cdot)}(\R^n_+)$ is a quasi-Banach space.  

{\rm (2)} For any $f,g\in F_{p(\cdot),q(\cdot)}^{s(\cdot)}(\R^n_+)$, we have 
\[
\left\| f+g \right\|_{F_{p(\cdot),q(\cdot)}^{s(\cdot)}(\R^n_+)}^{\min(p^-,q^-,1)} 
\le \left\| f\right\|_{F_{p(\cdot),q(\cdot)}^{s(\cdot)}(\R^n_+)}^{\min(p^-,q^-,1)} 
+ \left\| g \right\|_{F_{p(\cdot),q(\cdot)}^{s(\cdot)}(\R^n_+)}^{\min(p^-,q^-,1)}. 
\]  

{\rm (3)} Let 
\[
\alpha=
\min
\left(
q^-,\,1
\right)
\min
\left(1,\,
\left(
\frac{p}{q}
\right)^-
\right). 
\]
Then, for any $f,g\in B_{p(\cdot),q(\cdot)}^{s(\cdot)}(\R^n_+)$, we have 
\[
\left\| f+g \right\|_{B_{p(\cdot),q(\cdot)}^{s(\cdot)}(\R^n_+)}^{\alpha} 
\le \left\| f\right\|_{B_{p(\cdot),q(\cdot)}^{s(\cdot)}(\R^n_+)}^{\alpha} 
+ \left\| g \right\|_{B_{p(\cdot),q(\cdot)}^{s(\cdot)}(\R^n_+)}^{\alpha}. 
\]  
\end{theorem}

\subsection{ Extension of Franke and Runst's lift operator}

In this subsection, we extend the lifting operator introduced by Franke and Runst \cite{Franke-Runst}. 
We construct a collection of operator $\{J_{\sigma}\}_{\sigma\in\R}$ such that 
following three conditions. 
\begin{enumerate}
\item $J_{\sigma}$ is an isomorphism between $A_{p(\cdot),q(\cdot)}^{s(\cdot)}(\R^n)$ and $A_{p(\cdot),q(\cdot)}^{s(\cdot)-\sigma}(\R^n)$. 
\item $J_{\sigma}$ is an inverse map of $J_{-\sigma}$. 
\item Let $f\in\mathcal{S}'(\R^n)$. If 
\[
\spt f \subset \R^{n-1}\times (-\infty,0], 
\]
then 
\[
\spt J_{\sigma}f \subset \R^{n-1}\times (-\infty,0]. 
\]
\end{enumerate}

Let $\eta\in\mathcal{S}(\R^n)$ be a positive function which satisfy $\spt\eta\subseteq (-2,-1)$ and $\int_{\R^n}\eta(x){\rm d}x=2$. 
For any $0<\epsilon\ll 1$, we define a holomorphic function $\psi_{\epsilon}$ on $\C$ such that 
\[
\psi_{\epsilon}(x):= \int_{-\infty}^0\eta(t)e^{-i\epsilon tz}{\rm d}t-iz. 
\]

Let $\uH=\{z\in\C\,:\,{\rm Im}(z)>0\}$, $\overline{\uH}=\{z\in\C\,:\,{\rm Im}(z)\ge0\}$ and 
$\Omega=\{z\in\C\,:\,|z|>4, {\rm Re}(z)>0\}$. 
If $z\in\C$ satisfy $|z|>4$ and ${\rm Im}(z)>0$, then $-iz\in\Omega$. Hence we see 
\begin{equation}
{\rm dist}(\psi_{\epsilon}(z),\Omega)\le |\psi_{\epsilon}(z)+iz| 
= \left| \int_{-\infty}^0\eta(t)e^{-i\epsilon tz}{\rm d}t  \right|<2. 
\label{6.7}
\end{equation} 
If $z\in\C$ satisfy $|z|\le 4$ and ${\rm Im}(z)\ge 0$, then we have ${\rm Re}(\psi_{0}(z))=2+{\rm Im}(z)$. 
Therefore, for any $0<\epsilon\ll 1 $, we obtain 
\[
{\rm Re}(\psi_{\epsilon}(z))=\int_{-\infty}^0\eta(t)e^{\epsilon{\rm Im}(z)}\cos(\epsilon t{\rm Re}(z)){\rm d}t+{\rm Im}(z)\ge \frac{3}{2}. 
\]
If $\epsilon>0$ is a sufficiency small real number, we see that $\psi_{\epsilon}$ is a mapping from $\overline{\uH}$ to 
\begin{equation}
\Omega_0:=\{z\in\C\,:\, {\rm Re}(z)>1\} \cup \{ z\in\C\,:\,|{\rm Im}(z)|>1\}. \label{6.8}
\end{equation}
We fix such a sufficiency small real number $\epsilon>0$. 
We select a branch-cut of $\log$ on simply-connected region $\C\setminus (-\infty,0]$ such that 
$\log 1=0$. 
Then we define $a^z=\exp(a\log z)$ for $z\in\C\setminus(-\infty,0]$. 
Furthermore, for any $\sigma\in\R$, we define $\varphi^{(\sigma)}$ $:$ $\R^{n-1}\times\overline{\uH}\rightarrow\C$ such that 
\begin{equation}
\varphi^{(\sigma)}(x',z_n):= \left(\langle x'\rangle \psi_{\epsilon}\left( \frac{z_n}{\langle x' \rangle} \right)  \right)^{\sigma}, \ \ z\in\overline{\uH}. 
\label{6.9}
\end{equation} 
We put $\varphi:=\varphi^{(1)}$. Then we have following lemma. 

\begin{lemma}
\label{lemma 6.1.4} 
For any $\alpha\in\N_0^{n}$, 
\begin{equation}
\left| \partial^{\alpha}\varphi(x',z_n) \right| \lesssim_{\alpha}\left( \langle x'\rangle+|z_n|\right)^{1-|\alpha|}, 
\ \ (x',z_n)\in\R^{n-1}\times\overline{\uH} \label{6.10}
\end{equation}
holds. Furthermore, for any $(x',z_n)\in\R^{n-1}\times\overline{\uH}$, we have 
\begin{equation}
\langle x' \rangle+|z_n|\sim |\varphi(x',z_n)|. \label{6.11}
\end{equation}
\end{lemma}

\begin{proof}
Firstly, we will prove $(\ref{6.11})$. 
If $|z_n|>4\langle x' \rangle$, then we obtain 
\[
\left| \psi_{\epsilon}\left( \frac{|z_n|}{\langle x' \rangle} \right)  +i\frac{z_n}{\langle x' \rangle} \right| <2 
<\frac{|z_n|}{2\langle x' \rangle}
\]
by $(\ref{6.7})$. 
This implies 
\[
\frac{|z_n|}{2\langle x' \rangle}<
\left| \psi_{\epsilon}\left( \frac{|z_n|}{\langle x' \rangle} \right)  \right|  
<\frac{3|z_n|}{2\langle x' \rangle}. 
\]

If $|z_n|\le4\langle x' \rangle$, 
then we see that 
\[
\left| \psi_{\epsilon}\left( \frac{|z_n|}{\langle x' \rangle} \right)  \right| 
\le \int_{-\infty}^0 \eta(t){\rm d}t+\frac{|z_n|}{\langle x' \rangle} \le 2+4=6. 
\]
Hence $(\ref{6.11})$ holds. 

Finally, we will prove $(\ref{6.10})$. 
By the definition of $\varphi(x',z_n)$, we have 
\begin{equation}
\varphi(x',z_n)=\langle x' \rangle\int_{-\infty}^0 \eta(t)\exp\left(-it\epsilon\frac{z_n}{\langle x'\rangle} \right){\rm d}t-iz_n. 
\label{6.11'}
\end{equation}
Here, it is easy to see that the second term of the right hand side of $(\ref{6.11'})$ satisfies $(\ref{6.10})$.  
Hence, we estimate the first term of the right hand side of $(\ref{6.11'})$. 
By Libniz's formula, we see that 
\begin{align*}
\left| \partial^{\alpha} \left( \langle x' \rangle \int_{-\infty}^0\eta(t)\exp\left( -it\epsilon\frac{z_n}{\langle x' \rangle}\right) 
 {\rm d}t \right)   \right| 
&\lesssim \sum_{\gamma\le \alpha}\left| \partial^{\alpha-\gamma}\langle x' \rangle \partial^{\gamma} 
\left( \int_{-\infty}^0\eta(t)\exp\left( -it\epsilon\frac{z_n}{\langle x' \rangle} \right) {\rm d}t\right) \right| \\ 
&\lesssim \sum_{\gamma\le \alpha} \langle x'\rangle^{1-|\alpha|+|\gamma|} \left| \partial^{\gamma} 
\left( \int_{-\infty}^0\eta(t)\exp\left( -it\epsilon\frac{z_n}{\langle x' \rangle} \right) {\rm d}t\right) \right|. 
\end{align*}
Since 
\[
\int_{-\infty}^0\eta(t)\exp\left(-it\epsilon\frac{z_n}{\langle x' \rangle}\right){\rm d}t 
=\sqrt{2\pi} \fourier\left[ \eta\exp\left(t\epsilon\frac{{\rm Im}(z_n)}{\langle x'\rangle} \right) \right]
\left( t\epsilon\frac{{\rm Re}(z_n)}{\langle x'\rangle}\right), 
\]
we obtain 
\begin{align*}
\left| \partial^{\alpha} \left( \langle x' \rangle \int_{-\infty}^0\eta(t)\exp\left( -it\epsilon\frac{z_n}{\langle x' \rangle}\right) 
 {\rm d}t \right)   \right| 
\lesssim \langle x'\rangle^{1-|\alpha|}\left(1+\frac{|z_n|}{\langle x'\rangle}\right)^{1-|\alpha|} = \left(\langle x'\rangle+|z_n|\right)^{1-|\alpha|}. 
\end{align*}
\end{proof}

We also use same symbol $\varphi^{(\sigma)}$ for $\varphi^{(\sigma)}|_{\R^{n-1}\times\R}$. 
Then by Theorem $\ref{thm:5}$ and Lemma $\ref{lemma 6.1.4} $, we have following Proposition. 
\begin{proposition}
\label{proposition 6.1.5} 
Let $p(\cdot),q(\cdot)\in C^{\log}(\R^n)\cap \mathcal{P}_0(\R^n)$ and $s(\cdot)\in C^{\log}(\R^n)$. 
Then, for any $\sigma\in\R$, we have following properties. 

{\rm (1) } $J_{\sigma}:=\varphi^{(\sigma)}(D)$ is a linear isomorphism between $A_{p(\cdot),q(\cdot)}^{s(\cdot)}(\R^n)$ and  $A_{p(\cdot),q(\cdot)}^{s(\cdot)-\sigma}(\R^n)$. 

{\rm (2) } $J_{-\sigma}$ is the inverse operator of $J_{\sigma}$. 

{\rm (3) } For any $f\in A_{p(\cdot),q(\cdot)}^{s(\cdot)}(\R^n)$, we have 
\[
\left\| J_{\sigma}f\right\|_{ A_{p(\cdot),q(\cdot)}^{s(\cdot)-\sigma}(\R^n) } \sim \left\| f \right\|_{A_{p(\cdot),q(\cdot)}^{s(\cdot)}(\R^n)}. 
\] 
\end{proposition}

To consider the support of $J_{\sigma}f$, we use following the Paley--Wiener theorem. 

\begin{lemma}[Paley--Wiener theorem]
\label{lemma 6.1.6} 
Let $\varphi\in\mathcal{S}(\R^n)$ and $\epsilon>0$. Then, $\spt \varphi\subset \R^{n-1}\times [\epsilon,\infty)$ if and only if $\inversefourier\varphi$ can be extended 
to a continuous function $\Psi$ $:$ $\R^{n-1}\times\overline{\uH} \to \C$ satisfying 
\begin{enumerate}
\item $\Psi(\xi',\cdot)$ is a holomorphic function on $\uH$, 
\item For each $N\in\N$, 
\begin{equation}
\left| \Psi(\xi', \xi_n+i\zeta_n )\right|\lesssim_N \langle \xi \rangle^{-N}(1+\zeta_n)^{-N}\exp(-\epsilon\zeta_n) \label{6.12}
\end{equation}
holds for any $(\xi',\xi_n+i\zeta_n)\in\R^{n-1}\times \overline{\uH}$. 
\end{enumerate}
\end{lemma}

By using the Paley--Wiener theorem, we obtain a result about the support of $J_{\sigma}f$. 
\begin{proposition}
\label{proposition 6.1.7} 
Let $f\in\mathcal{S}'(\R^n)$. If $\spt f\subset \R^{n-1}\times(-\infty,0]$, then  $\spt J_{\sigma}f\subset \R^{n-1}\times(-\infty,0]$. 
\end{proposition}

\begin{proof}
We take a test function $\psi\in\mathcal{D}(\R^{n-1}\times (0,\infty))$. Since $\psi$ has a compact support, we see that there exists an $\epsilon>0$ such that 
$\spt \psi\subset \R^{n-1}\times [\epsilon,\infty)$. Then we have 
\begin{equation}
\langle J_{\sigma}f,\psi\rangle = \langle f,\fourier[\varphi^{(\sigma)}\inversefourier \psi]\rangle. \label{6.13}
\end{equation}
We can see that $\inversefourier\psi$ satisfies $(\ref{6.12})$ and that $\varphi^{(\sigma)}\inversefourier \psi$ also satisfies $(\ref{6.12})$ by Lemma $\ref{lemma 6.1.4}$. 
Hence, we obtain $\spt \fourier[\varphi^{(\sigma)}\inversefourier\varphi] \subset \R^{n-1}\times [\epsilon,\infty)$. 
Therefore, we have $\langle J_{\sigma}f,\psi\rangle=0$ because $\spt f\subset \R^{n-1}\times(-\infty, 0]$. 
\end{proof}

\begin{theorem}
\label{theorem 6.1.8} 
Let $p(\cdot),q(\cdot)\in C^{\log}(\R^n)\cap\mathcal{P}_0(\R^n)$, $s(\cdot)\in C^{\log}(\R^n)$ and $\sigma\in\R$. 

{\rm (1) } Let $f\in A_{p(\cdot),q(\cdot)}^{s(\cdot)}(\R^n_+)$. Then $J_{\sigma}f:= (J_{\sigma}g)|_{\R^n_+}$ does not depend on $g\in A_{p(\cdot),q(\cdot)}^{s(\cdot)}(\R^n)$ satisfying 
$f=g|_{\R^n_+}$.  

{\rm (2) } $J_{\sigma}$ is an isomorphism between $ A_{p(\cdot),q(\cdot)}^{s(\cdot)}(\R^n_+)$ and $ A_{p(\cdot),q(\cdot)}^{s(\cdot)-\sigma}(\R^n_+)$. 
Furthermore,  $J_{-\sigma}$ is the inverse of $J_{\sigma}$.  
\end{theorem}

\begin{proof}
We will prove ${\rm (1)}$. Let $g_1, g_2\in   A_{p(\cdot),q(\cdot)}^{s(\cdot)}(\R^n)$ satisfy $f=g_1|_{\R^n_+}=g_2|_{\R^n_+}$. 
Then we have 
\[
(J_{\sigma}(g_1-g_2))|_{\R^n_+}=0
\]
by $(g_1-g_2)|_{\R^n_+}=0$ and Proposition $\ref{proposition 6.1.7}$. 
Therefore, we obtain 
\begin{equation}
(J_{\sigma}g_1)|_{\R^n_+} = (J_{\sigma}g_2)|_{\R^n_+} \label{6.14}
\end{equation}
because 
\[
(J_{\sigma}g_1)|_{\R^n_+} - (J_{\sigma}g_2)|_{\R^n_+} = (J_{\sigma}(g_1-g_2))|_{\R^n_+}=0. 
\]
$(\ref{6.14})$ means $J_{\sigma}f$ does not depend on $g\in A_{p(\cdot),q(\cdot)}^{s(\cdot)}(\R^n)$ satisfying $f=g|_{\R^n_+}$. 
${\rm (2)}$ follows from the properties of $J_{\sigma}$ as an operator on $ A_{p(\cdot),q(\cdot)}^{s(\cdot)}(\R^n)$. 
\end{proof}

\subsection{An extension operator for $A_{p(\cdot),q(\cdot)}^{s(\cdot)}(\R^n_+)$} 

\begin{theorem}
\label{theorem 6.1.9} 
Let $N\in\N$, $p(\cdot),q(\cdot)\in C^{\log}(\R^n)\cap \mathcal{P}_0(\R^n)$ and $s(\cdot)\in C^{\log}(\R^n)$. 
Then there exists an operator ${\rm Ext}_N$ which is so called extension operator: 
\begin{equation}
{\rm Ext}_N\,:\,\bigcup_{\substack{p(\cdot),q(\cdot)\,:\, N^{-1}\le p^-,q^- \\ s(\cdot)\,:\, s^+\le N}}  A_{p(\cdot),q(\cdot)}^{s(\cdot)}(\R^n_+)  
\longrightarrow 
\bigcup_{\substack{p(\cdot),q(\cdot)\,:\, N^{-1}\le p^-,q^- \\ s(\cdot)\,:\, s^+\le N}}   A_{p(\cdot),q(\cdot)}^{s(\cdot)}(\R^n), 
\label{6.15}
\end{equation} 
satisfying the following conditions. 

{\rm (1) } ${\rm Ext}_N|_{ A_{p(\cdot),q(\cdot)}^{s(\cdot)}(\R^n_+)  }$ is continuous. 

{\rm (2) } For any $f\in  A_{p(\cdot),q(\cdot)}^{s(\cdot)}(\R^n_+)  $, $({\rm Ext}_N f)|_{ \R^n_+ }=f$. 
\end{theorem}

\begin{proof} 
Let $p(\cdot),q(\cdot)\in C^{\log}(\R^n)\cap \mathcal{P}_0(\R^n)$  satisfy $N^{-1} \le p^-, q^- $ and $s(\cdot)\in C^{\log}(\R^n)$ satisfy $||s||_{\infty}<N$. 

{\bf Step 1. }
Let $M\in\N$ be a sufficiency large. Let $\delta_{0,l}$ be the Kronecker delta function. We define $\lambda_1,\lambda_2,\cdots,\lambda_M\in\R$ uniquely such that 
the following simultaneous equations  
\begin{equation}
\sum_{j=0}^M (-j)^l\lambda_j =\delta_{0,l}, \ \ l=0,1,\cdots, M-1 \label{6.16}
\end{equation}
hold. 
Since a discriminant $D$ of this simultaneous equation is some positive constant times Vandermonde determinant ${\rm det}\{j^i\}_{i,j=1,\cdots, M}$, 
we have $D\neq 0$. Hence $\lambda_1,\lambda_2,\cdots, \lambda_M$ can be determined uniquely. 

For a function $f$ $:$ $\R^{n-1}\times[0,\infty) \longrightarrow \C$, we define 
\begin{equation}
f^*(x) := \begin{cases} f(x) \ \ &\text{(if $x_n\ge 0$)}, \\ 
                               \sum_{j=1}^M\lambda_j f(x', -jx_n) \ \ &\text{(if $x_n \le 0$)}. 
                               \end{cases} 
                               \label{6.17}
\end{equation}
Let $f\in\mathcal{B}^M(\R^n)$  be defined a neighborhood of $\R^{n-1}\times[0,\infty)$. 
Then $f^*\in\mathcal{B}(\R^n)$ because we defined $f^*$ $:$ $\R^{n} \longrightarrow \C$ such that the differential coefficient of $f^*$ at boundary coincides with $f$.  

{\bf Step 2. } 
We consider the case that $s(\cdot)$ satisfy $0> \left( \frac{n}{p(\cdot)} -s(\cdot) \right)^+$. 
In this case, we have $A_{p(\cdot),q(\cdot)}^{s(\cdot)}(\R^n) \hookrightarrow {\rm BUC}(\R^n)$ by Proposition $\ref{proposition 3.1.33}$. 

Let $M\in\N_0$ as in Step 1 be sufficiently large enough to $M \gg (n+1)N$. 
For $f\in A_{p(\cdot),q(\cdot)}^{s(\cdot)}(\R^n_+)$, there exist $g\in A_{p(\cdot),q(\cdot)}^{s(\cdot)}(\R^n)$ such that 
$f=g|_{\R^n_+}$ and 
\[
\| g\|_{A_{p(\cdot),q(\cdot)}^{s(\cdot)}(\R^n)} \le 2 \| f\|_{A_{p(\cdot),q(\cdot)}^{s(\cdot)}(\R^n_+)}. 
\]
Let $\rho,r$ as in Theorem $\ref{quark decomposition}$. Then, by Theorem $\ref{quark decomposition}$, we can express $g$ as  
\begin{equation}
g=\sum_{\beta\in\N_0^n}\sum_{\nu\in\N_0}\sum_{m\in\Z^n}\lambda_{\nu,m}^{\beta}(\beta\qu)_{\nu,m}. \label{6.18}
\end{equation}
We have $\| \lambda \|_{a_{p(\cdot)q(\cdot),\rho}^{s(\cdot)}} \lesssim \| g\|_{A_{p(\cdot),q(\cdot)}^{s(\cdot)}(\R^n)} \lesssim \| f\|_{A_{p(\cdot),q(\cdot)}^{s(\cdot)}(\R^n_+)}$, 
where $\lambda=\{\lambda_{\nu,m}^{\beta}\}_{\nu\in\N_0,m\in\Z^n, \beta\in\N_0^n}$. 
Then we define $g^*$ so that 
\begin{equation}
g^*:=\sum_{\beta\in\N_0^n}\sum_{\nu\in\N_0}\sum_{m\in\Z^n}\lambda_{\nu,m}^{\beta}(\beta\qu)_{\nu,m}^*. \label{6.19}
\end{equation} 
Let $h\in  A_{p(\cdot),q(\cdot)}^{s(\cdot)}(\R^n)$ satisfy $f=h|_{\R^n_+}$. Then, by Theorem $\ref{quark decomposition}$, we can also express $h$ as 
\begin{equation}
h=\sum_{\beta\in\N_0^n}\sum_{\nu\in\N_0}\sum_{m\in\Z^n}\rho_{\nu,m}^{\beta}(\beta\qu)_{\nu,m}. \label{6.18-6/28}
\end{equation}
We also define $h^*$ so that 
\begin{equation}
h^*=\sum_{\beta\in\N_0^n}\sum_{\nu\in\N_0}\sum_{m\in\Z^n}\rho_{\nu,m}^{\beta}(\beta\qu)_{\nu,m}^*. \label{6.18-6/28-2}
\end{equation}
Then, we prove $g^*=h^*$ in the sense of $\mathcal{S}'(\R^n)$. That is, we prove that $g^*$ depend only on $f$.  
Since  $A_{p(\cdot),q(\cdot)}^{s(\cdot)}(\R^n) \hookrightarrow {\rm BUC}(\R^n)$  and  $0> \left( \frac{n}{p(\cdot)} -s(\cdot) \right)^+$, 
$g^*$ and $h^*$ are uniformly continuous functions. 
Hence, it is sufficient to prove that $g^*(x)=h^*(x)$ for any $x\in\R^n$.  
Since $g$ and $h$ are continuous functions and 
$g|_{\R^n_+}=h|_{\R^n_+}$ in the sense of $\mathcal{D}'(\R^n)$, 
$g(x)=h(x)$ for any $x=(x',x_n)\in\R^{n-1}\times\R$, where $x_n>0$. 
By the continuity of $g(x)$ and $h(x)$, 
we have $g(x)=h(x)$ for any $x=(x',x_n)\in\R^{n-1}\times\R$, where $x_n\ge 0$. 
This implies that $g^*(x)=h^*(x)$ for any $x\in\R^n_+$.  
By the definition of $g^*(x)$ and $h^*(x)$, we have $g^*(x)=h^*(x)$ for any $x\in\R^n$. 

Since $g^*$ depend only on $f$, 
we can consider 
\begin{equation}
{\rm Ext}_Nf:=\sum_{\beta\in\N_0^n}\sum_{\nu\in\N_0}\sum_{m\in\Z^n}\lambda_{\nu,m}^{\beta}(\beta\qu)_{\nu,m}^* \label{6.20}
\end{equation}
and 
\begin{equation}
{\rm Ext}_N^{\beta}f:=\sum_{\nu\in\N_0}\sum_{m\in\Z^n}\lambda_{\nu,m}^{\beta}(\beta\qu)_{\nu,m}^*. \label{6.21}
\end{equation}
Let $\lambda^{\beta}=\{\lambda_{\nu,m}^{\beta}\}_{\nu\in\N_0,m\in\Z^n}$ and $\epsilon>0$ be sufficiently small such that $0<\epsilon<\rho-r$. 
The right hand side of $(\ref{6.21})$ is not a quarkonial decomposition. However we can regard  $2^{-(r+\epsilon)|\beta|}{\rm Ext}_N^{\beta}f$ as an atomic decomposition by the family of smooth atoms with no moment condition. 
Hence we have 
\[
\left\| {\rm Ext}_N^{\beta}f \right\|_{A_{p(\cdot),q(\cdot)}^{s(\cdot)}(\R^n)} \lesssim 2^{(\rho+\epsilon)|\beta|}  \left\| \lambda^{\beta} \right\|_{a_{p(\cdot),q(\cdot)}^{s(\cdot)}} 
\lesssim \frac{   \left\| \lambda \right\|_{a_{p(\cdot),q(\cdot),\rho}^{s(\cdot)}}  }{2^{\delta|\beta|}} \lesssim \frac{ \| f\|_{A_{p(\cdot),q(\cdot)}^{s(\cdot)}(\R^n_+)}  }{2^{\delta|\beta|}}. 
\]
Therefore, we obtain 
\begin{align*}
\left\| {\rm Ext}_Nf \right\|_{A_{p(\cdot),q(\cdot)}^{s(\cdot)}(\R^n)} 
\lesssim  \| f\|_{A_{p(\cdot),q(\cdot)}^{s(\cdot)}(\R^n_+)} 
\end{align*}
by Lemma $\ref{min triangle}$. 
This implies that ${\rm Ext}_N$ is a continuous mapping and has desired properties. 

{\bf Step 3. } 
In this step, we reduce the condition  $0> \left( \frac{n}{p(\cdot)} -s(\cdot) \right)^+$. 
We take a $\sigma\in\R$ such that  $0> \left( \frac{n}{p(\cdot)} -(s(\cdot)+\sigma) \right)^+$ and 
$nN<-N+\sigma<N+\sigma$. 
Let $L\in\N$ be sufficiently large enough to satisfy $L\gg 1$ and  $N+\sigma\le L$. 
We can construct ${\rm Ext}_L$ $:$ $A_{p(\cdot),q(\cdot)}^{s(\cdot)+\sigma}(\R^n_+) \longrightarrow A_{p(\cdot),q(\cdot)}^{s(\cdot)+\sigma}(\R^n)$ by using the same argument in Step 2. 
We rewrite ${\rm Ext}_L$ to $E^*$. Then, by Step 2, we see that $E^*|_{A_{p(\cdot),q(\cdot)}^{s(\cdot)+\sigma}(\R^n_+)}$ is a continuous mapping from 
$A_{p(\cdot),q(\cdot)}^{s(\cdot)+\sigma}(\R^n_+)$ to $A_{p(\cdot),q(\cdot)}^{s(\cdot)+\sigma}(\R^n)$ and that $E^* f|_{\R^n_+} = f$ holds for any 
$f\in A_{p(\cdot),q(\cdot)}^{s(\cdot)+\sigma}(\R^n_+)$. 
Then, $J_{-\sigma}$ is the continuous mapping from $A_{p(\cdot),q(\cdot)}^{s(\cdot)}(\R^n_+)$ to $A_{p(\cdot),q(\cdot)}^{s(\cdot)+\sigma}(\R^n_+)$ and 
$J_{\sigma}$ is also the continuous mapping from $A_{p(\cdot),q(\cdot)}^{s(\cdot)+\sigma}(\R^n)$ to $A_{p(\cdot),q(\cdot)}^{s(\cdot)}(\R^n)$. 
Therefore, following composite mapping 
\begin{equation}
{\rm Ext}_N := J_{\sigma}\circ E^* \circ J_{-\sigma}\,:\, A_{p(\cdot),q(\cdot)}^{s(\cdot)}(\R^n_+) \rightarrow A_{p(\cdot),q(\cdot)}^{s(\cdot)}(\R^n) \label{6.22}
\end{equation}
make a sense. 

We will prove ${\rm Ext}_Nf|_{\R^n_+}=f$. 
Let $\phi\in\mathcal{D}(\R^n_+)$ and $g\in A_{p(\cdot),q(\cdot)}^{s(\cdot)}(\R^n) $ such that $f=g|_{\R^n_+}$. 
Let $E$ $:$ $C_c^{\infty}(\R^n_+) \rightarrow C^{\infty}(\R^n)$ be a zero extension operator. 
Then we see that 
\begin{align*}
\langle {\rm Ext}_Nf|_{\R^n_+}, \phi\rangle &= \langle {\rm Ext}_Nf, E\phi\rangle \\ 
&= \langle E^*J_{\sigma}f, \fourier\left[ \varphi^{(\sigma)}\inversefourier E\phi\right]\rangle \\ 
&= \langle J_{-\sigma}f, \fourier\left[ \varphi^{(\sigma)}\inversefourier E\phi\right] |_{\R^n_+} \rangle 
\end{align*}
by the properties of $E^*$. 
Hence we obtain
\[
\langle {\rm Ext}_Nf|_{\R^n_+}, \phi\rangle =  \langle J_{-\sigma}g, \fourier\left[ \varphi^{(\sigma)}\inversefourier E\phi\right] \rangle  
= \langle g, E\phi\rangle =\langle f, \phi\rangle
\]
by the definition of the extension operator ${\rm Ext}_N$. Therefore, we have ${\rm Ext}_Nf|_{\R^n_+}=f$ for any $f\in  A_{p(\cdot),q(\cdot)}^{s(\cdot)}(\R^n_+) $. 
\end{proof} 

\subsection{Trace operator for $ A_{p(\cdot),q(\cdot)}^{s(\cdot)}(\R^n_+)$}

We extend Theorem $\ref{theorem 5.4.4}$ to  $ A_{p(\cdot),q(\cdot)}^{s(\cdot)}(\R^n_+)$. 
 We consider the Trace operator 
\[
\utr : f(x',x_n)\in\mathcal{S}(\R^n_+)\longmapsto f(x',0)\in\mathcal{S}(\R^{n-1}). 
\]

\begin{theorem}
\label{theorem 6.1.10}
Assume that $p(\cdot),q(\cdot)\in C^{\log}(\R^n)\cap\mathcal{P}_0(\R^n)$. 

\begin{enumerate}
\item[{\rm (1)}] Let $s(\cdot)\in C^{\log}(\R^n)$ satisfy 
\[
\mathop{{\rm{ess}}\,\,\rm{inf}}_{x\in\R^n} \left\{ s(\cdot)-\left[\frac{1}{p(\cdot)}+(n-1)\left( \frac{1}{\min(1, p(\cdot) )}-1 \right)\right]\right\}>0. 
\]
\begin{enumerate}
\item The operator $\utr$ can be extended as a surjective and continuous mapping from $B_{p(\cdot),q(\cdot)}^{s(\cdot)}(\R^n_+)$ 
to  $B_{\tilde{p}(\cdot),\tilde{q}(\cdot)}^{\tilde{s}(\cdot)-\frac{1}{\tilde{p}(\cdot)}}(\R^{n-1})$. 
\item The operator $\utr$ can be extended as a surjective and continuous mapping from $F_{p(\cdot),q(\cdot)}^{s(\cdot)}(\R^n_+)$ 
to  $F_{\tilde{p}(\cdot),\tilde{p}(\cdot)}^{\tilde{s}(\cdot)-\frac{1}{\tilde{p}(\cdot)}}(\R^{n-1})$.  
\end{enumerate} 
\item[{\rm (2)}]  Let $s(\cdot)\in C^{\log}(\R^n)$ and $k\in\N_0$ satisfy 
\[
\mathop{{\rm{ess}}\,\,\rm{inf}}_{x\in\R^n} \left\{ s(\cdot)-\left[k+\frac{1}{p(\cdot)}+(n-1)\left( \frac{1}{\min(1, p(\cdot) )}-1 \right)\right]\right\}>0. 
\]
\begin{enumerate}
\item  If $g_0\in B_{\tilde{p}(\cdot),\tilde{q}(\cdot)}^{\tilde{s}(\cdot)-\frac{1}{\tilde{p}(\cdot)}}(\R^{n-1})$, 
$g_1\in B_{\tilde{p}(\cdot),\tilde{q}(\cdot)}^{\tilde{s}(\cdot)-\frac{1}{\tilde{p}(\cdot)}-1}(\R^{n-1})$, $\cdots$, 
$g_k\in B_{\tilde{p}(\cdot),\tilde{q}(\cdot)}^{\tilde{s}(\cdot)-\frac{1}{\tilde{p}(\cdot)}-k}(\R^{n-1})$, then there exists a $f\in B_{p(\cdot),q(\cdot)}^{s(\cdot)}(\R^n_+)$ such that 
$\utr(f)=g_0$, $\utr(\partial_{x_n}f)=g_1$, $\cdots$, $\utr(\partial_{x_n}^kf)=g_k$. 
\item  If $g_0\in F_{\tilde{p}(\cdot),\tilde{p}(\cdot)}^{\tilde{s}(\cdot)-\frac{1}{\tilde{p}(\cdot)}}(\R^{n-1})$, 
$g_1\in F_{\tilde{p}(\cdot),\tilde{p}(\cdot)}^{\tilde{s}(\cdot)-\frac{1}{\tilde{p}(\cdot)}-1}(\R^{n-1})$, $\cdots$, 
$g_k\in F_{\tilde{p}(\cdot),\tilde{p}(\cdot)}^{\tilde{s}(\cdot)-\frac{1}{\tilde{p}(\cdot)}-k}(\R^{n-1})$, then there exists a $f\in F_{p(\cdot),q(\cdot)}^{s(\cdot)}(\R^n_+)$ such that 
$\utr(f)=g_0$, $\utr(\partial_{x_n}f)=g_1$, $\cdots$, $\utr(\partial_{x_n}^kf)=g_k$. 
\end{enumerate} 
\end{enumerate} 
\end{theorem}

\begin{proof}
Let $N$ be a sufficiency large. For $f\in A_{p(\cdot),q(\cdot)}^{s(\cdot)}(\R^n_+)$, we define $\utr f:= \tr[{\rm Ext}_N f]$, where 
${\rm Ext}_N$ is the extension operator as in Theorem $\ref{theorem 6.1.9}$. 
Firstly, we will prove that 
\begin{equation}
\utr [f|_{\R^n_+}] = \lim_{\epsilon\downarrow 0}\tr {\rm Ext}_N[f(\cdot', \cdot_n+\epsilon)|_{\R^n_+}]. \label{6.23} 
\end{equation}
By Theorem $\ref{quark decomposition}$, we can expansion ${\rm Ext}_Nf\in A_{p(\cdot),q(\cdot)}^{s(\cdot)}(\R^n)$ so that 
\[
{\rm Ext}_N f =\sum_{\beta\in\N_0^n}\sum_{\nu\in\N_0}\sum_{m\in\Z^n}\lambda_{\nu,m}^{\beta}(\beta\qu)_{\nu,m}. 
\]
By using the same argument of the proof of Theorem $\ref{theorem 5.4.4}$, we have 
\begin{equation}
\tr[{\rm Ext}_N f] =\sum_{\beta\in\N_0^n}\sum_{\nu\in\N_0}\sum_{m\in\Z^n}\lambda_{\nu,m}^{\beta}\tr[(\beta\qu)_{\nu,m}]. \label{6.24} 
\end{equation} 
The partial summation $\displaystyle \sum_{m\in\Z^n}\lambda_{\nu,m}^{\beta}\tr[(\beta\qu)_{\nu,m}]$ in $(\ref{6.24})$ is a bounded set of $A_{p(\cdot),q(\cdot)}^{s(\cdot)}(\R^n)$. 
Hence there exists $\kappa>0$ and $C>0$ such that 
\begin{equation*}
\left\| 
\sum_{m\in\Z^n}\lambda_{\nu,m}^{\beta}\tr[(\beta\qu)_{\nu,m}]
\right\|_{A_{\tilde{p}(\cdot),\tilde{q}(\cdot)}^{\tilde{s}(\cdot)-\frac{1}{\tilde{p}(\cdot)}}(\R^{n-1})} \le  C2^{-\kappa|\beta|}, 
\end{equation*}
where $\kappa$ and $C$ do not depend on $\nu$ and $\beta$. 
Let $\delta>0$. 
If we change $s(\cdot)$ into $s(\cdot)-\delta$, then 
\begin{equation*}
\left\| 
\sum_{m\in\Z^n}\lambda_{\nu,m}^{\beta}\tr[(\beta\qu)_{\nu,m}]
\right\|_{A_{\tilde{p}(\cdot),\tilde{q}(\cdot)}^{\tilde{s}(\cdot)-\frac{1}{\tilde{p}(\cdot)}-\delta}(\R^{n-1})} \le  C2^{-\kappa|\beta|-\delta\nu} 
\end{equation*}
holds by the definition of $(\beta\qu)_{\nu,m}$. 
Since $\displaystyle A_{\tilde{p}(\cdot),\tilde{q}(\cdot)}^{\tilde{s}(\cdot)-\frac{1}{\tilde{p}(\cdot)} -\delta}(\R^{n-1}) \hookrightarrow B_{\infty,\infty}^{\tilde{s}(\cdot)-\frac{n}{\tilde{p}(\cdot)}-\delta}(\R^{n-1})$ by Proposition \ref{Sobolev}, 
we have 
\begin{equation}
\left\| 
\sum_{m\in\Z^n}\lambda_{\nu,m}^{\beta}\tr[(\beta\qu)_{\nu,m}]
\right\|_{B_{\infty,\infty}^{\tilde{s}(\cdot)-\frac{n}{\tilde{p}(\cdot)}-\delta}(\R^{n-1})} \lesssim 2^{-\kappa|\beta|-\delta\nu}. \label{6-28-1}
\end{equation} 
Then we prove the limit  
\begin{equation}
\tr[{\rm Ext}_N f] =\lim_{\epsilon\downarrow 0}\left( \sum_{\beta\in\N_0^n}\sum_{\nu\in\N_0}\sum_{m\in\Z^n}\lambda_{\nu,m}^{\beta}\tr[(\beta\qu)_{\nu,m}(\cdot',\cdot_n+\epsilon)]\right) 
\label{6-28}
\end{equation}
exists in $\mathcal{S}'(\R^{n-1})$. 
Let take $\varphi\in\mathcal{S}(\R^{n-1})$ arbitrary. Since 
\[
\left| 
\left\langle
\sum_{m\in\Z^n}\lambda_{\nu,m}^{\beta}\tr[(\beta\qu)_{\nu,m}], \varphi
\right\rangle
\right|
\lesssim 2^{-\kappa|\beta|-\delta\nu}, 
\]
there exists a finitely set $\mathcal{U}\subset\N_0^n\times\N$ for any $\epsilon'>0$ such that 
\begin{equation}
\sum_{(\beta,\nu)\in\N_0^n\times\N\setminus\mathcal{U}}
\left| 
\left\langle
\sum_{m\in\Z^n}\lambda_{\nu,m}^{\beta}\tr[(\beta\qu)_{\nu,m}(\cdot',\cdot_n+\epsilon)], \varphi
\right\rangle
\right|
\le C\frac{1}{2}\epsilon', \label{6-28-2}
\end{equation}
where constant number $C>0$ does not depend on $\epsilon\in (0,1)$. 
Let 
\[
A_j= \left\langle
 \sum_{m\in\Z^n}\lambda_{\nu,m}^{\beta}\tr[(\beta\qu)_{\nu,m}(\cdot',\cdot_n+\epsilon_j)], \varphi
\right\rangle
\]
for $j=1, 2$, where $0<\epsilon_1, \epsilon_2 \ll 1$. 
Since 
\[
\lim_{\epsilon\downarrow 0}\sum_{(\beta,\nu)\in\mathcal{U} }\sum_{m\in\Z^n}\lambda_{\nu,m}^{\beta}\tr[(\beta\qu)_{\nu,m}(\cdot',\cdot_n+\epsilon)] 
\]
exists, we see that 
\begin{equation}
\sum_{(\beta,\nu)\in\mathcal{U}}\left| A_1-A_2\right| <\frac{1}{2}\epsilon'. \label{6-28-3}
\end{equation}
By $(\ref{6-28-2})$ and $(\ref{6-28-3})$, the limit $(\ref{6-28})$ exists in the sense of $\mathcal{S}(\R^{n-1})$.  
Hence we obtain 
\begin{equation}
\tr[{\rm Ext}_Nf] =\lim_{\epsilon\downarrow 0}\tr[{\rm Ext}_N f(\cdot',\cdot_n+\epsilon)]. \label{6.25}
\end{equation}
Therefore, we see that $\utr$ does not depend on $N$ because $(\ref{6.23})$ and $\tr[{\rm Ext}_N f(\cdot',\cdot_n+\epsilon)]$ does not depend on $N$ for any $\epsilon>0$. 

Next, we prove that $\utr$ is a surjection. 
By the surjective of the operator $\tr$ on $\R^n$, we prove 
\begin{equation}
\tr f = \begin{cases} \utr [f|_{\R^n_+}]\in B_{\tilde{p}(\cdot),\tilde{q}(\cdot)}^{\tilde{s}(\cdot)-\frac{1}{\tilde{p}(\cdot)}}(\R^{n-1}) 
&\ \ \text{if } f\in B_{\tilde{p}(\cdot),\tilde{q}(\cdot)}^{\tilde{s}(\cdot)}(\R^n), \\ 
\utr [f|_{\R^n_+}]\in F_{\tilde{p}(\cdot),\tilde{p}(\cdot)}^{\tilde{s}(\cdot)-\frac{1}{\tilde{p}(\cdot)}}(\R^{n-1}) 
&\ \ \text{if } f\in F_{\tilde{p}(\cdot),\tilde{q}(\cdot)}^{\tilde{s}(\cdot)}(\R^n). 
\end{cases}
\label{6.26}
\end{equation}
By using same argument of $(\ref{6.25})$, we see that 
\begin{equation}
\begin{cases}
\tr f=\lim_{\epsilon\downarrow 0}\tr f(\cdot',\cdot_n+\epsilon), \\ 
\utr [f|_{\R^n_+}] = \lim_{\epsilon\downarrow 0}\tr{\rm Ext}_N[f(\cdot',\cdot_n+\epsilon)|_{\R^n_+}]
\end{cases}
\label{6.27}
\end{equation}
in $\mathcal{S}'(\R^{n-1})$. Here we have $f(\cdot',\cdot_n+\epsilon)={\rm Ext}_N[f(\cdot',\cdot_n+\epsilon)|_{\R_+^n}]$ on $\{x\in\R^n\,:\,x_n\ge -\epsilon/2\}$. 
Hence we obtain 
\[
\tr f=\lim_{\epsilon\downarrow 0}\tr f(\cdot',\cdot_n+\epsilon)=\lim_{\epsilon\downarrow 0}\tr{\rm Ext}_N[f(\cdot',\cdot_n+\epsilon)|_{\R^n_+}]=\utr[f|_{\R^n_+}]. 
\]
This complete the proof of $(1)$. 

The assertion $(2)$ follows from $(1)$ and Theorem $\ref{theorem 5.4.4}$. 
\end{proof}

\section*{Acknowledgement}
I would like to express my gratitude to Professor Yoshikazu Kobayashi for great support in any manner and for his valuable suggestions in many 
discussions.  
 I also would like to express my gratitude to Professor Yoshihiro Sawano for sending his book \cite{Sawano}. 
 I have obtained many ideas from the book \cite{Sawano}.

% Use this code if you wish to generate your bibliography with BibTeX;
% please replace first the string "demo" below with the name(s) of
% the BibTeX data base(s) you want to use.
% The resulting bibliography-output (the contents of the .bbl file)
% must be pasted into this file before submission.
% 
% \bibliographystyle{mn}
% \bibliography{demo}

\begin{flushleft}
Takahiro Noi \\ 
Department of Mathematics and Information Science, Tokyo Metropolitan University \\ 
1-1 Minami osawa, hachioji-city, Tokyo. \\  
E-mail : taka.noi.hiro@gmail.com
\end{flushleft}

\end{document}